%% file: BTD_2_SIMAX_revision_2_final_black_and_red.tex
\documentclass[final]{siamart190516}
\input{defs_BTD_2}

\usepackage[T1]{fontenc}
\usepackage[normalem]{ulem}
\usepackage{array}
\usepackage[caption=false]{subfig}
\usepackage{ragged2e}
\usepackage{arydshln}
\input{ex_shared}

\ifpdf
\hypersetup{
  pdftitle={\TheTitle},
  pdfauthor={\TheAuthors}
}
\fi

\begin{document}

\maketitle

\begin{abstract}
Canonical Polyadic Decomposition (CPD) represents a third-order tensor as the minimal sum of rank-$1$ terms.
Because of its uniqueness properties the  CPD has found many concrete applications
in telecommunication, array processing, machine learning, etc.
On the other hand, in several applications the rank-$1$ constraint on the terms is too restrictive. A multilinear rank-$(M,N,L)$ constraint  (where a rank-$1$ term is the special case for which $M=N=L=1$) could be more realistic, while it still yields a decomposition with attractive uniqueness properties.

In this paper  we focus on the decomposition of a tensor $\mathcal T$ into a sum of  multilinear rank-$(1,L_r,L_r)$ terms, $r=1,\dots,R$.
This particular decomposition type has already found  applications in wireless communication, chemometrics and the  blind signal separation of signals that can be modelled as exponential polynomials and rational functions.
We find  conditions on the terms which guarantee that  the decomposition is  unique and can be computed by means of
the eigenvalue decomposition of a matrix \tcr{even in the cases where none of the factor matrices has full column rank}.  We consider both the case where the  decomposition is exact and the case where the decomposition holds only approximately. We show that in both cases the number of the terms $R$ and their ``sizes'' $L_1,\dots,L_R$ do not have to be known a priori and can be estimated as well. The conditions for uniqueness are easy to verify, especially  for terms that can be considered ``generic''.
 \tcr{In particular, we obtain  the following two generalizations of  a well known result on generic uniqueness of the CPD (i.e., the case $L_1=\dots=L_R=1$): we show that the multilinear rank-$(1,L_r,L_r)$ decomposition of an $I\times J\times K$ tensor is generically unique if i)  $L_1=\dots=L_R=:L$ and $R\leq \min((J-L)(K-L),I)$ or if ii) $\sum L_R\leq \min((I-1)(J-1),K)$ and $J\geq \max(L_i+L_j)$. }
\end{abstract}



\section{Introduction}
\subsection{Terminology and problem setting}\label{sec:theveryfirstsubsection}
Throughout the paper $\fF$ denotes the field of real or complex numbers. 

By definition, a  tensor $\mathcal T=(t_{ijk})\in\fF^{I\times J\times K}$     is  {\em multiLinear rank-$(1,L,L)$ (ML rank-$(1,L,L)$)} if it equals the outer product of a nonzero vector $\mathbf a\in\fF^I$ and a rank-$L$  matrix $\mathbf E=(e_{ij})\in\fF^{J\times K}$:
$\mathcal T=\mathbf a\circ\mathbf E$, which means that 	$t_{ijk}=a_i e_{jk}$ for all values of indices. If it is only known that the rank of $\mathbf E$  is bounded by $L$, then 
we say that $\mathcal T=\mathbf a\circ\mathbf E$ is ML rank at most $(1,L,L)$ and write  ``$\mathcal T$ is \MLatmostLL{L}''.

In this paper we study the  {\em decomposition} of $\mathcal T\in\fF^{I\times J\times K}$ into a sum of such
 terms of \MLatmost\footnote{\tcr{The results of this paper can also be applied for the decomposition into a sum of max ML rank-$(L_r,1,L_r)$  (resp. -$(L_r,L_r,1)$) terms by switching the first and second (resp. third) dimensions of $\mathcal T$.}}:
 \begin{equation}
\mathcal T = \sum_{r=1}^R\mathbf a_r\circ\mathbf E_r, \qquad \mathbf a_r\in\fF^I\setminus\{\mathbf 0\},\qquad \mathbf E_r\in\fF^{J\times K},\qquad 
r_{\mathbf E_r}\leq L_r,
\label{eq:LrLr1}
\end{equation}
where $\mathbf 0$ denotes the zero vector and $r_{\mathbf E_r}$ denotes the rank of $\mathbf E_r$. If exactly $r_{\mathbf E_r}= L_r$ for all $r$, then we call \cref{eq:LrLr1}     \tcr{``the   decomposition of $\mathcal T$ into a sum of `\ML terms'' or, briefly, its ``\ML decomposition''}.

In this paper we study the uniqueness and computation of \cref{eq:LrLr1}. For uniqueness we use the following basic definition.

\begin{definition}\label{def:unic_dec} 
	Let $L_1,\dots, L_R$ be fixed positive integers.
The decomposition of  $\mathcal T$ into a sum of \MLatmost  terms
is   {\em unique} if for any two decompositions of the form \cref{eq:LrLr1} one can be obtained from another by a permutation of summands. 
\end{definition}
\tcr{Thus, the uniqueness is not affected by the trivial ambiguities in \cref{eq:LrLr1}:   permutation of  the \MLatmost terms and (nonzero) scaling/counterscaling
$\lambda\mathbf a_r$ and  $\lambda^{-1}\mathbf E_r$.} \tcr{\Cref{def:unic_dec} implies that if the decomposition is  unique, then it is necessarily  minimal, that is, if \cref{eq:LrLr1} holds with $r_{\mathbf E_r}=L_r$, then a decomposition of the form 
\cref{eq:LrLr1} with smaller $L_r$ does not exist, in particular, a decomposition with smaller number of terms does not exist.}

We will not only investigate  the ``global'' uniqueness of decomposition \cref{eq:LrLr1} but also particular instances of ``partial'' uniqueness. Let us call the matrix 
$$
\mathbf A=[\mathbf a_1\ \dots\ \mathbf a_R]
$$
{\em the first factor matrix} of the decomposition of  $\mathcal T$ into a sum of \MLatmost  terms. 
For uniqueness of $\mathbf A$, we will resort to the following definition.
\begin{definition}  \label{Def:overall_uniqueness}
	Let $L_1,\dots, L_R$ be fixed positive integers.
The first factor matrix of  the decomposition of  $\mathcal T$ into a sum of \MLatmost  terms is {\em unique} if for any two decompositions of the form \cref{eq:LrLr1}  
their first factor matrices  coincide up to column permutation and  \tcr{(nonzero)} scaling.
\end{definition}
It follows from \cref{Def:overall_uniqueness} that if $\mathcal T$ admits a decomposition of the form \cref{eq:LrLr1} with fewer than $R$ terms, then the first factor matrix   is not unique. On the other hand, as a preview of one result, \cref{example:1} will illustrate
that the first factor matrix may be unique without the overall ML rank decomposition being unique.

\Cref{def:unic_dec,Def:overall_uniqueness} concern deterministic forms of uniqueness. We will also develop generic uniqueness results.
To make the rank constraints $r_{\mathbf E_r}\leq L_r$ in \cref{eq:LrLr1}  easier to handle and to present the definition of generic uniqueness, we 
factorize $\mathbf E_r$ as $\mathbf B_r\mathbf C_r^T$, where the matrices $\mathbf B_r\in\fF^{J\times L_r}$ and $\mathbf C_r\in\fF^{K\times L_r}$   are rank at most $L_r$.
Thus, \cref{eq:LrLr1} can be rewritten as
\begin{equation}
\begin{gathered}
\mathcal T = \sum_{r=1}^R\mathbf a_r\circ(\mathbf B_r\mathbf C_r^T),\\ \mathbf a_r\in\fF^I\setminus\{\mathbf 0\},\ \mathbf B_r\in\fF^{J\times L_r},\ \mathbf C_r\in\fF^{K\times L_r},\ r_{\mathbf B_r}\leq L_r,\ r_{\mathbf C_r}\leq L_r,\ r= 1,\dots, R. 
\end{gathered}
 \label{eq:LrLr1mainBC}
\end{equation}
Throughout the paper, we set
\begin{align*}
&\mathbf B=[\mathbf B_1\ \dots\ \mathbf B_R]\in\fF^{J\times \sum L_r},\quad \mathbf B_r = [\mathbf b_{1,r}\ \dots\ \mathbf b_{L_r,r}]=(b_{jl,r})_{j,l=1}^{J,L_r}\\
&\mathbf C=[\mathbf C_1\ \dots\ \mathbf C_R]\in\fF^{K\times \sum L_r},\quad \mathbf C_r = [\mathbf c_{1,r}\ \dots\ \mathbf c_{L_r,r}]=(c_{kl,r})_{k,l=1}^{K,L_r}.
\end{align*}
We call the matrices $\mathbf B$ and $\mathbf C$ {\em the  second and  third factor matrix} of $\mathcal T$, respectively.
Decomposition \cref{eq:LrLr1mainBC} can then be represented in   matrix form as
\begin{align}
\unf{T}{1}&:=[\operatorname{vec}(\mathbf H_1)\ \dots\ \operatorname{vec}(\mathbf H_I)]=[\operatorname{vec}(\mathbf E_1)\ \dots\ \operatorname{vec}(\mathbf E_R)]\mathbf A^T,\label{eq:unf_T_1}\\
\unf{T}{2}&:=[\mathbf H_1\ \dots\ \mathbf H_I]^T = [\mathbf a_1\otimes\mathbf C_1\ \dots\ \mathbf a_R\otimes\mathbf C_R]\mathbf B^T=\sum\limits_{r=1}^R\mathbf a_r\otimes\mathbf E_r^T,\label{eq:unf_T_2}\\
\unf{T}{3}&:=[\mathbf H_1^T\ \dots\ \mathbf H_I^T]^T = [\mathbf a_1\otimes\mathbf B_1\ \dots\ \mathbf a_R\otimes\mathbf B_R]\mathbf C^T=\sum\limits_{r=1}^R\mathbf a_r\otimes\mathbf E_r,\label{eq:unf_T_3}
\end{align}
where $\mathbf H_1,\dots,\mathbf H_I\in\fF^{J\times K}$ denote the horizontal slices of $\mathcal T$, $\mathbf H_i:= (t_{ijk})_{j,k=1}^{J,K}$,
$\operatorname{vec}(\mathbf H_i)$ denotes the $JK\times 1$ column vector obtained by stacking the columns of the matrix $\mathbf H_i$ on top of one another, and
``$\otimes$'' denotes the Kronecker product. The matrices
$\unf{T}{1}\in\fF^{JK\times I}$, $\unf{T}{2}\in\fF^{IK\times J}$, and $\unf{T}{3}\in\fF^{IJ\times K}$ are called {\em the matrix unfoldings}\footnote{\tcr{Some papers, e.g., \cite{Kolda}, define the matrix unfoldings as  the transposed matrices $\unf{T}{1}^T$, $\unf{T}{2}^T$, and $\unf{T}{3}^T$.}} of $\mathcal T$. One can easily verify that $\mathcal T$ is ML rank-$(1,L,L)$ if and only if $r_{\unf{T}{1}}=1$ and $r_{\unf{T}{2}}=r_{\unf{T}{3}}=L$.

We have now what we need to formally define generic uniqueness.
\begin{definition}\label{def:genericuniqueness}
	Let $L_1,\dots, L_R$ be fixed positive integers and let
	$\mu$ be a measure on $\fF^{I\times R}\times\fF^{J\times\sum L_r}\times \fF^{K\times\sum L_r}$ that is absolutely continuous with respect to the Lebesgue measure. The
  decomposition of an $I\times J\times K$ tensor into a sum of \MLatmost terms is  {\em generically unique} if 
$$
	\mu\{(\mathbf A,\mathbf B,\mathbf C):\
	\text{  decomposition \cref{eq:LrLr1mainBC} is not unique} \}=0.
$$
\end{definition}
 Thus, if the entries of the matrices $\mathbf A$, $\mathbf B$, and $\mathbf C$ are randomly sampled from an absolutely continuous distribution, then generic uniqueness means uniqueness that holds with probability one.

If $L_1=\dots=L_R=1$, then the minimal decomposition of the form \cref{eq:LrLr1} is known as the Canonical Polyadic Decomposition
(CPD) (aka CANDECOMP/PARAFAC). Because of their uniqueness properties both  CPD and   decomposition into a sum of \MLatmost terms have found many concrete applications
in telecommunication, array processing, machine learning, etc. \cite{Kolda,LievenCichocki2013,ComoJ10,TensRev2017}.
For the decomposition into a sum of \MLatmost terms we  mention in particular applications in wireless communication \cite{LDL2008}, chemometrics \cite{Bro2009} and blind signal separation of signals that can be modeled as exponential polynomials \cite{LievenLrLr1} and rational functions \cite{Otto2016}. Some advantages of  a blind separation method that relies on decomposition of the form \cref{eq:LrLr1} over the methods that rely on PCA, ICA, and CPD are discussed in \cite{LievenCichocki2013,TensRev2017}. As a matter of fact, it is a profound advantage of the tensor setting over the common vector/matrix setting that data components do not need to be rank-$1$ to admit a unique recovery, i.e., terms such as the ones in  \cref{eq:LrLr1} allow us to model more general contributions to observed data. It is also worth noting that
if $R\leq I$, then  \cref{eq:LrLr1} can reformulated as a problem of finding a basis consisting of low-rank matrices, namely
the basis $\{\mathbf E_1,\dots,\mathbf E_R\}$ of the matrix subspace spanned by the horizontal slices of $\mathcal T$,
$\sspan\{\mathbf H_1,\dots,\mathbf H_I\}$ \cite{lowrankbasis2017}.

In this paper we find  conditions on the factor matrices  which guarantee that the decomposition of a tensor
into a sum of	\MLatmost terms
is   unique (in the deterministic or in the generic sense).  We also derive conditions under which, perhaps surprisingly, the decomposition can essentially be  computed  by means of a
\newcolumntype{C}{ >{\centering\arraybackslash} m{6mm} }
\newcolumntype{D}{ >{\centering\arraybackslash} m{3mm} }
\newcolumntype{M}{ >{\centering\arraybackslash} m{3mm} }
\newcolumntype{F}{ >{\centering\arraybackslash} m{3mm} }
\newcolumntype{E}{ >{\centering\arraybackslash} m{19mm} }
\begin{table}
	\captionsetup{position=top} 
	\caption{Known and some of the new bounds on $R$ and $L_1,\dots,L_R$ under which the decomposition of an $I\times J\times K$ tensor into a sum of \MLatmost terms is generically unique, where  $\min(I,J,K,R)\geq 2$. Additional  bounds can  be obtained by switching $J$ and $K$ in rows $2$, $5$, $6$, and $8$. The  boxed line in each cell with bounds indicates which factor matrices are required to have  full column rank (f.c.r). (Since we are in the generic setting,  full column rank of the first, second, and third factor matrix is equivalent to  $I\geq R$, $J\geq \sum L_r$, and $K\geq \sum L_r$, respectively.)  The check mark in the ``$\lambda$''-column indicates that the result on uniqueness comes with an EVD based algorithm. 
		The bounds in rows $4$ and $6$ hold upon verification that a particular matrix has full column rank. For row $4$ no exceptions have been reported.
		We have verified the bounds in row $6$ for $\max(I,J)\leq 5$. For the case where not all $L_r$ are identical we found three exceptions in which the matrix does not have full column rank;
			 for the case $L_1=\dots=L_R=L$ we haven't found exceptions. (For more details on the bounds in row $6$ see \cref{Appendix:FF}). 
			 The bounds in row $8$ imply that generic uniqueness does hold for two of three exceptions.
		 }
	\label{tab:KoMa14}
	\centering
	\subfloat[Known bounds (\cref{sec:secfortable})]{
		\begin{tabular}{|D|C|m{41mm}|m{54mm}|F|} \hline
			\#&	ref &\multicolumn{1}{c|}{$L_1\leq\dots\leq L_R$}&  \multicolumn{1}{c|}{$L_1=\dots=L_R=:L$} & 
			$\lambda$\\ \hline 
			1&	\cite{LDLBTDPartII}  &\center{$\boxed{\textstyle J\geq \sum L_r,\  K\geq \sum L_r}$}&\center{$\boxed{\textstyle J\geq RL,\  K\geq RL}$ }&     $\checkmark$ \\ \hline 
			2&	\cite{BTD1paper}  &\centering{$\boxed{\textstyle I\!\geq\! R,\ J\geq \sum L_r}$}\\   \justifying \raggedright{$K\geq  L_R +1$}  & \centering{$\boxed{\textstyle I\!\geq\!R,\ J\geq RL}$}\\   \justifying \raggedright{$K\geq  L +1$}&    $\checkmark$ \\ \hline 
			3&	\cite{LDLBTDPartII} &\centering{$\boxed{\textstyle I\!\geq\!R}$}\\ \justifying 
			\raggedright{
			$ J \geq \sum L_{p}+\dots+L_R$ and\\
			$ K \geq \sum L_{q}+\dots+L_R$,	\\
			for some $p+q\leq R$}			
					& 
			\centering{$\boxed{\textstyle I\!\geq\!R}$}\\ \justifying 
			\raggedright{$\min(\lfloor \frac{J}{L}\rfloor,R) + \min(\lfloor \frac{K}{L}\rfloor,R)\geq R+2$,
			\\ where
			$\lfloor x\rfloor $ denotes the greatest \newline integer less than or equal to $x$
		}
			&       \\ \hline 
			4&	\cite{MikaelCoupledPII}   & \multicolumn{1}{c|}{\em not applicable}& 
			\centering{(\em upon verification)} \\
			$\boxed{\textstyle I\geq R}$\\ \justifying
			\raggedright{$\rubinom{J}{L+1}\rubinom{K}{L+1}\geq \rubinom{R+L}{L+1}-R$}
			&    $\checkmark$ \\ \hline 
	\end{tabular}}
	
	\subfloat[New bounds (\cref{sec:genuniq})]{
		\begin{tabular}{|M|E|m{43mm}|m{39mm}|F|} \hline
			\#&	ref  &\multicolumn{1}{c|}{$L_1\leq\dots\leq L_R$}&  \multicolumn{1}{c|}{$L_1=\dots=L_R=:L$} & 	$\lambda$\\ \hline
			5&
			Theorem \ref{thm:maingenshort}
			& \centering{\fbox{no f.c.r. assumptions}}\\ \justifying
			   \raggedright{ $K\geq L_2+\dots+L_R+1$ and\\
			   $J\geq L_{\min(I,R)-1}+\dots+L_R$}
			&\centering{\fbox{no f.c.r.  assumptions} }\\ \justifying
			\raggedright{$K\geq (R-1)L+1$ and\\
			$J\geq (R-\min(R,I)+2)L$}
			&     $\checkmark$ \\ \cline{1-1}\hline
			\multirow[]{2}{*}[-1cm]{6} &  \centering{Theorem}\newline
			\centering{\ \ \ref{thm:maingen} \ref{item:genstate2}}   & 	 \centering{(\em upon verification)}\\
			$\boxed{\textstyle K \geq \sum L_r}$ \\ \justifying
			\raggedright{$J\geq L_{R-1}+L_R$ and\\
				$\rubinom{I}{2}\rubinom{J}{2}\geq \sum\limits_{r_1<r_2} \hspace{-2mm}L_{r_1}L_{r_2}$}  &  \centering{(\em upon verification) 
				$\boxed{K\geq RL}$}\\ \justifying 
			\raggedright{$J\geq 2L$ and\\
			$\rubinom{I}{2}\rubinom{J}{2}\geq \rubinom{R}{2}L^2$} & \multirow{2}{*}[-1cm]{$\checkmark$}\\
			\cdashline{2-4}
			&  
			verification \newline
			mechanism \newline
			is explained\newline
			\phantom{\quad}in\newline
			\centering{Appendix \ref{Appendix:FF}}  & 	 \raggedright{exceptions for\\ $\max(I,J)\leq 5$:\\
			3 tuples\\ $(I,J,R,L_1,\dots,L_R)$ with\\ $L_1=\dots,L_{R-1}=1$,\\ $L_R=4$, $J=5$, and $(I,R)\in\{(2,3), (4,9),(5,12)\}$}  &  there are no exceptions \newline for $\max(I,J)\leq 5$ &\\ \hline
			7&	Theorem \ref{thm:maingenLL1}   & \multicolumn{1}{c|}{\em not applicable}&  \centering{$\boxed{I\!\geq\!R}$}\\ \justifying 
		\raggedright{$(J-L)(K-L)\geq R$}
			&     \\ \hline 
			8&	Theorem \ref{thm:maingenStrassen}  &\centering{$\boxed{\textstyle  K\geq \sum L_r}$} \\ \justifying
			\raggedright{$J\geq L_{R-1}+L_R$ and\\
			$(I-1)(J-1)\geq \sum L_r$} &\centering{$\boxed{K\geq RL}$}\\ \justifying \raggedright{$J\geq 2L$ and\\ $(I-1)(J-1)\geq RL$} &      \\ \hline 
	\end{tabular}}
\end{table}
 \hspace{-0.6cm}matrix eigenvalue decomposition (EVD).
This will be possible even in  cases where none of the factor matrices has full column rank. The main results are formulated in \cref{thm: maintheorem,thm: maintheoremABC,thm:maingen,thm:maingenLL1,thm:maingenStrassen} below. \tcr{\cref{tab:KoMa14} summarizes known and new\footnote{\tcr{One of the  new results, namely,
			 the part of \cref{item:genstate2} in \cref{thm:maingen} that relies on the assumption  $I\geq R$,  
			is not mentioned  in the table because its presentation requires additional notations.}}  results for generic decompositions. By way of comparison,
the known results  guarantee that the decomposition of an  $8\times 8\times 50$ tensor into a sum of $R-1$  ML rank-$(1,1,1)$ terms and one  ML rank-$(1,2,2)$ term is generically unique up to $R\leq 8$ (row $3$) and can be computed by means of EVD   up to $R\leq 7$ (rows $1$ and $2$), while the  results obtained in the paper imply that  generic uniqueness holds up to $R\leq 48$ (row $8$) and that computation is possible up to $R\leq 39$ (row $6$). }

A final word of caution is in order. It may happen that a tensor admits  more than one  decomposition into a sum of \MLatmost terms among which only one is  exactly \ML  (see \cref{ex:2.4} below). In this case one can thus say  that the \ML decomposition of the tensor is unique. In this paper however, we will always present conditions for  uniqueness of the decomposition into a sum of {\em max} \ML terms.
It is clear that such conditions  imply also uniqueness of the (exactly) \ML decomposition.

Throughout the paper  \tcr{$\mathbf O$, $\mathbf I$, and $\mathbf I_n$ denote the zero matrix, the identity matrix, and the specific identity matrix of size $n\times n$, respectively;} $\nullsp{\cdot}$ denotes the null space of a matrix; 
``$^T$'',  ``$^H$'', and ``$^\dagger$'' denote the  transpose,  hermitian transpose, and pseudo-inverse, respectively.
\tcr{We will also use the shorthand notations
	$
	\sum L_r
	$, $
	\sum d_r
	$, and  $\min L_r$ for $\sum\limits_{r=1}^RL_r$, $\sum\limits_{r=1}^Rd_r$, and $\min\limits_{1\leq r\leq R} L_r$, respectively.}

\tcr{All numerical experiments in the paper were performed in MATLAB R2018b. To make  the results reproducible,  the random number generator was initialized	using the built-in function \texttt{rng('default')} (the Mersenne Twister with seed $0$).}

\subsection{Previous results}
\subsubsection{Results on  decomposition into a sum of \MLatmost  terms}\label{sec:secfortable}
In the following two theorems  it is assumed that at least two factor matrices have full column rank.
The first result is well-known. Its proof is essentially obtained by picking two generic mixtures of slices of $\mathcal T$ and  computing their generalized EVD.
The values $L_1,\dots,L_R$  need not   be known in advance and can be found  as  multiplicities of the  eigenvalues.
\begin{theorem}\label{thm:ll1_gevd}\cite[Theorem 4.1]{LDLBTDPartII}
	Let $\mathcal T$ admit decomposition \cref{eq:LrLr1mainBC}. Assume that any two columns of $\mathbf A$ are linearly independent and that the matrices $\mathbf B$ and $\mathbf C$ have full column rank.
Then    the decomposition of $\mathcal T$ into a sum of \MLatmost terms is unique and can be computed by means of EVD. Moreover, any
decomposition of $\mathcal T$ into a sum of $\hat R$ terms of max ML rank-$(1,\hat{L}_{\hat r},\hat{L}_{\hat r})$ for which $\sum\limits_{\hat r=1}^{\hat R} \hat{L}_{\hat r}=\sum\limits_{r=1}^R L_r$  should necessarily coincide with decomposition \cref{eq:LrLr1mainBC}.		
\end{theorem}
\begin{theorem}\label{thm:ll1_btd1}\cite[Corollary 1.4]{BTD1paper} 
		Let $\mathcal T$ admit  \ML decomposition  \cref{eq:LrLr1mainBC}
		 and let at least one of the following assumptions hold:
		\begin{assumptions}
			\item  $\mathbf A$ and $\mathbf B$  have full column rank and
			$r_{[\mathbf C_i\ \mathbf C_j]}\geq \max(L_i,L_j)+1$ for all $1\leq i<j\leq R$;
			\item  \label{assum:bth1.5}$\mathbf A$ and $\mathbf C$  have full column rank and
			$r_{[\mathbf B_i\ \mathbf B_j]}\geq \max(L_i,L_j)+1$ for all $1\leq i<j\leq R$.
		\end{assumptions}
		Then   the decomposition of $\mathcal T$ into a sum of \MLatmost terms is unique  and can be computed by means of EVD.
\end{theorem}
\par The uniqueness and computation of  the decomposition into a sum of \MLatmost terms, \tcr{where  $L_1=\dots=L_R:=L$}, was also studied in \cite[Subsection 5.2]{MikaelCoupledPII} \tcr{and \cite{Nion_LDL_LL1}}. 
We do not reproduce the results  from \cite{MikaelCoupledPII}  \tcr{(resp. \cite{Nion_LDL_LL1})} here because this would require many specific notations. We just mention that one of the assumptions in \cite{MikaelCoupledPII} \tcr{(resp. \cite{Nion_LDL_LL1})} is that the first factor matrix \tcr{(resp. the second or third factor matrix)} has full column rank
 and another   assumption implies that the dimensions of $\mathcal T$  satisfy the inequality  \tcr{$\rubinom{\min(J,RL)}{L+1}\rubinom{\min(K,RL)}{L+1}\geq \rubinom{R+L}{L+1}-R$} \tcr{(resp. the \tcr{inequality} $\rubinom{\min(I,R)}{2}\rubinom{\min(J,K,LR)}{2}\geq \rubinom{R}{2}L^2$)},
where \tcr{$\rubinom{n}{k}$} denotes the binomial coefficient  
\begin{equation*}
\tcr{
\rubinom{n}{k}:=\frac{n!}{k!(n-k)!}.
}
\end{equation*} 
\tcr{To present the next result we need}  the  definitions of $k$-rank  of a matrix \tcr{(``$k$'' refers to J.B. Kruskal)} and $k'$-rank of a block matrix. 
\begin{definition}
	The $k$-rank of the matrix $\mathbf A=[\mathbf a_1\ \dots\ \mathbf a_R]$  
	 is the largest number $k_{\mathbf A}$ such that any $k_{\mathbf A}$ columns of $\mathbf A$ are linearly independent.
\end{definition}
\begin{definition}\cite[Definition 3.2]{LDLBTDPartII}
	The $k'$-rank of the   matrix $\mathbf B=[\mathbf B_1\ \dots\ \mathbf B_R]$ is the largest number $k_{\mathbf B}'$ such that any set \tcr{$\{\mathbf B_i\}$} of $k_{\mathbf B}'$ blocks of $\mathbf B$ yields a
	set of linearly independent columns. 
 \end{definition}
\tcr{In the following theorem none of the factor matrices is required to have full column rank.}
\begin{theorem}\cite[Lemma 4.2]{LDLBTDPartII}\label{thm;LievenBTDKruskal}
	Let $\mathcal T$ admit \ML decomposition \cref{eq:LrLr1mainBC} with $L_1=\dots=L_R$. Assume that
	\begin{equation*}
	k_{\mathbf A}+k_{\mathbf B}'+k_{\mathbf C}'\geq 2R+2.
	\end{equation*}
	Then the first factor matrix in the  \MLatmost  decomposition of $\mathcal T$ is unique. If additionally,
	$r_{\mathbf A}=R$, then the overall \MLatmost decomposition of $\mathcal T$ is unique.
\end{theorem}
\par In the following theorem we summarize the known results  on generic uniqueness of the decomposition into a sum of \MLatmost terms.
Statements 1), 2)-3), and 4) are just generic counterparts of \cref{thm:ll1_gevd}, \cref{thm:ll1_btd1}, and
\cref{thm;LievenBTDKruskal}, respectively. Some of the statements have also appeared in \cite{LDLBTDPartII, BTD1paper,Yang2014,StrangeLL1paper}.
\begin{theorem}\label{thm:genericknown}
Let $L_1\leq \dots\leq L_R$. Then each of the following conditions implies that 
the decomposition of an $I\times J\times K$ tensor into a sum of \MLatmost terms
is generically unique:
	\begin{statements}
	\item 
	$I\geq 2$, $J\geq \sum L_r$, and $K\geq \sum L_r$;
	\item 
	$I\geq R$, $J\geq \sum L_r$, and $K\geq  L_R +1$;
	\item 
	$I\geq R$, $J\geq  L_R +1$, and $K\geq \sum L_r$;	
	\item 
	$I\geq R$ and $\kgp{B} +  \kgp{C}\geq R+2$,  where
		\end{statements}
	\begin{equation*}
	\begin{split}
\kgp{B}&:= \max\{p:\ L_{R-p+1}+\dots+L_R\leq J\},\\
\kgp{C}&:= \max\{q:\ L_{R-q+1}+\dots+L_R\leq K\}.
	\end{split} 
	\end{equation*}
 \end{theorem} 
\subsubsection{An auxiliary result on symmetric joint block diagonalization problem}\label{subsub:alg1}
In \cref{subsubsection213} we will establish a link between  decomposition \cref{eq:LrLr1} and  a special case of the 
Symmetric Joint Block Diagonalization (S-JBD) problem introduced in this subsection. In particular, we will show in \cref{subsubsection213} that
uniqueness and computation of the first factor matrix in \cref{eq:LrLr1} follow from  uniqueness and computation  of a solution of the S-JBD problem. We will consider both the cases where   decomposition \cref{eq:LrLr1} is exact and the case where the decomposition holds only approximately. In the latter case, decomposition \cref{eq:LrLr1} is just fitted to the given tensor $\mathcal T$. Thus,  in this subsection, we also consider both the cases where the  S-JBD is exact and the case where the S-JBD holds approximately.

{\bf Exact S-JBD.} Let $\matrixU_1,\dots,\matrixU_Q$ be  $K\times K$ symmetric matrices that can be jointly block diagonalized  as
	\begin{equation}
	\begin{gathered}
	\matrixU_q =\mathbf N\mathbf D_q\mathbf N^T,\quad \mathbf N=[\mathbf N_1\ \dots\ \mathbf N_R],\quad \mathbf N_r\in\fF^{K\times d_r},\\
 	\mathbf D_q=\Bdiag(\mathbf D_{1,q},\dots,\mathbf D_{R,q}),\quad
	\mathbf D_{r,q}=\mathbf D_{r,q}^T\in\fF^{d_r\times d_r},\quad q=1,\dots,Q,
	\end{gathered}\label{eq:JBDaux}
	\end{equation}
	where  $d_1,\dots,d_R, Q$ are positive integers,
	and $\Bdiag(\mathbf D_{1,q},\dots,\mathbf D_{R,q})$
	denotes a block-diagonal matrix with  the matrices $\mathbf D_{1,q}$, $\dots, \mathbf D_{R,q}$ on the diagonal.
	It is worth noting that the columns of $\mathbf N$ are not required to be orthogonal and that we  deal with the non-hermitian transpose in \cref{eq:JBDaux} even if $\fF=\mathbb C$.
	Let $\mathbf \Pi$ be a $\sum d_r\times \sum d_r$ permutation matrix such that $\mathbf N\mathbf \Pi$  admits the same block partitioning as $\mathbf N$ and
	let $\mathbf D$ be a nonsingular symmetric  block diagonal matrix whose diagonal  blocks  have dimensions $d_1,\dots,d_R$.
	Then obviously  $\matrixU_1,\dots,\matrixU_Q$ can also be jointly block diagonalized  as
	\begin{equation*}
	\matrixU_q =(\mathbf N\mathbf D\mathbf \Pi)(\mathbf \Pi^T\mathbf D^{-1}\mathbf D_q\mathbf D^{-T}\mathbf \Pi)(\mathbf N\mathbf D\mathbf \Pi)^T=:
	\tilde{\mathbf N}\tilde{\mathbf D}_q\tilde{\mathbf N}^T,\quad q=1,\dots,Q.
	\end{equation*}
	We say that the solution of  the  S-JBD problem \cref{eq:JBDaux} {\em is unique}, if for any two solutions 
	$$
	\matrixU_q =\mathbf N\mathbf D_q\mathbf N^T=\tilde{\mathbf N}\tilde{\mathbf D}_q\tilde{\mathbf N}^T
	,\qquad q=1,\dots,Q
	$$
	there exist matrices $\mathbf D$ and $\mathbf \Pi$ such that
	$$
	\tilde{\mathbf N}=\mathbf N\mathbf D\mathbf \Pi, \quad
	\tilde{\mathbf D}_q=\mathbf \Pi^T\mathbf D^{-1}\mathbf D_q\mathbf D^{-T}\mathbf \Pi,\quad q=1,\dots,Q.
	$$ 
Thus, if the solution of \cref{eq:JBDaux} is unique, then the number of blocks $R$ in \cref{eq:JBDaux} is minimal and the column spaces of $\mathbf N_1,\dots,\mathbf N_R$ (as well as their dimensions $d_1,\dots,d_R$) can be identified up to permutation. For a  thorough study of JBD we refer to \cite{Cay2017SIMAX} and the references therein. 

In  \cref{subsubsection213} we will  rework \cref{eq:LrLr1mainBC} into a problem of the form  \cref{eq:JBDaux}.
\tcr{In the   case $d_1=\dots=d_R=1$ the  S-JBD problem \cref{eq:JBDaux} is reduced to a special case of the classical symmetric joint diagonalization (S-JD) problem  (a.k.a. simultaneous diagonalization by congruence), where ``special'' means that the number of matrices $Q$ equals the size $R$ of the diagonal  matrices. It is well known and can easily be derived from 
\cite[Theorem 4.5.17]{HornJohnson} that if there exists a rank-$R$ linear combination of	$\matrixU_1,\dots, \matrixU_Q$, then  the solution of S-JD is unique and can be computed by means of (simultaneous)  EVD. The  following theorem states that a similar result
also holds for S-JBD problem \cref{eq:JBDaux}.}
\begin{theorem}\label{thm:JBDaux}
	Let  $Q:=\rubinom{d_1+1}{2}+\dots+\rubinom{d_R+1}{2}$, \tcr{$\min (d_1,\dots,d_R)\geq 2$} and let $\matrixU_1,\dots,\matrixU_Q$ be  $K\times K$ symmetric matrices that can be jointly block diagonalized  as in \cref{eq:JBDaux}.
	Assume that
	\begin{assumptions}
		\item $\mathbf N$ has full column rank;
		\item the matrices $\mathbf D_1,\dots,\mathbf D_Q$ are linearly independent.
	\end{assumptions}
	Then the solution of S-JBD problem \cref{eq:JBDaux} is unique and can be computed by means of (simultaneous)  EVD\footnote{The simultaneous EVD problem consists of finding a similarity transform that reduces a set of (commuting) matrices to   diagonal form.}. 
\end{theorem}
\begin{proof}
Let $\lambda_1,\dots,\lambda_Q\in\fF$ be generic.
Since $Q$ is equal to the  dimension of the subspace of all $\sum \tcr{d}_r\times \sum \tcr{d}_r$ symmetric block diagonal matrices,
the block diagonal matrix $\sum\lambda_q\mathbf D_q$ in $\sum\lambda_q\matrixU_q = \mathbf N (\sum\lambda_q\mathbf D_q) \mathbf N^T$ is also generic.
Thus, replacing each equation in  \cref{eq:JBDaux} by a (known) generic linear combination of all equations, we can assume without loss of generality (w.l.o.g.) that the matrices $\mathbf D_q$ are generic.
By 	\cite[Theorem 1.10]{BTD1paper}, the solution of the  obtained S-JBD problem  is unique and can be computed by means of (simultaneous)  EVD if we have at least $3$ equations, which is the case since $Q\geq \tcr{\rubinom{2+1}{2}=}3$. 
\end{proof}
\par The   algebraic procedure related to \cref{thm:JBDaux} is summarized in \cref{Alg:auxjbd} (see \cite[Subsection 2.3]{Cay2017SIMAX} and \cite[Algorithm 1 and Theorem 1.10]{BTD1paper}), where we assume w.l.o.g. that $K=\sum d_r$.
The value $R$ and the matrices $\mathbf U_1,\dots,\mathbf U_R$ in step 1 can be computed as follows.
Vectorizing the   matrix equation $\mathbf O =\mathbf U\matrixU_q-\matrixU_q\mathbf U^T$, we obtain that
$\mathbf 0=
(\matrixU_q^T\otimes \mathbf I)\operatorname{vec}(\mathbf U) - (\mathbf I\otimes \matrixU_q)\operatorname{vec}(\mathbf U^T)=
(\matrixU_q^T\otimes \mathbf I - (\mathbf I\otimes \matrixU_q)\mathbf P)\operatorname{vec}(\mathbf U)$,
where $\mathbf P$ denotes the $K^2\times K^2$ permutation matrix that transforms the vectorized form of a $K\times K$ matrix into the vectorized form of its transpose. Let $\mathbf M$ denote  the $K^2Q\times K^2$ matrix formed by the rows of $\matrixU_q^T\otimes \mathbf I - (\mathbf I\otimes \matrixU_q)\mathbf P$, $q=1,\dots,Q$. Then we obtain $R=\dim\nullsp{\mathbf M}$ and choose 
$\mathbf U_1,\dots,\mathbf U_R$ such that 
$\operatorname{vec}(\mathbf U_1),\dots\operatorname{vec}(\mathbf U_R)$ form a basis of $\nullsp{\mathbf M}$.
 
 \begin{algorithm}
 	\caption{Computation of S-JBD problem \cref{eq:JBDaux} under the conditions in  \cref{thm:JBDaux} }
 	\label{Alg:auxjbd}	
 	\begin{algorithmic}[1]
 		\INPUT{$K\times K$  symmetric matrices $\matrixU_1,\dots,\matrixU_Q$ with the property that there exist matrices  $\mathbf N$ and $\mathbf D_1,\dots,\mathbf D_Q$ such that $\matrixU_1,\dots,\matrixU_Q$ can be factorized as in \cref{eq:JBDaux},  the assumptions in \cref{thm:JBDaux} hold and $K=\sum d_r$
 		}
 		\STATE{ Find  $R$ and the matrices $\mathbf U_1,\dots,\mathbf U_R$ that form a basis of the subspace \\
 			\begin{center}$\{\mathbf U\in\fF^{K\times K}:\ \mathbf U\matrixU_q=\matrixU_q\mathbf U^T,\ q=1,\dots,Q \}$
 		\end{center}}
 		\STATE{Find $\mathbf N$  and the values $d_1,\dots,d_R$ from the simultaneous EVD\\
 			\begin{center}
 				$
 				\mathbf U_r=\mathbf N\Bdiag(\lambda_{1r}\mathbf I_{d_1},\dots,\lambda_{Rr}\mathbf I_{d_R})\mathbf N^{-1},\qquad r=1,\dots,R
 				$
 			\end{center}
 		} 
 		\STATE{For each $q=1,\dots,Q$ compute $\mathbf D_q=\mathbf N^{-1}\matrixU_q\mathbf N^{-T}$}
 		\OUTPUT{Matrices $\mathbf N$, $\mathbf D_1,\dots,\mathbf D_Q$ and the values $R$, $d_1,\dots,d_R$ such that  \cref{eq:JBDaux} holds}
 	\end{algorithmic}
 \end{algorithm}
 
It is worth noting that the computations in steps 1 and 2  can be simplified  as follows. From the proof of \cref{thm:JBDaux} it follows that the matrices $\matrixU_1,\dots,\matrixU_Q$ in step 1 can be replaced by three generic linear combinations.
It was also proved in \cite{Cay2017SIMAX} that the  simultaneous EVD in step 2 can be replaced by the EVD of a single matrix $\mathbf Z$, namely, a generic linear combination of $\mathbf U_1,\dots,\mathbf U_R$.
Then the values $d_1,\dots,d_R$ can be computed as the multiplicities of $R$ (distinct) eigenvalues of $\mathbf Z$. 

{\bf Approximate S-JBD.}  Optimization based schemes  for the approximate S-JBD problem are discussed in the  recent paper \cite{Cherrak2017} (see also \cite{Cay2017SIMAX,BTD1paper,VanDerVeen1996} and references therein).   The authors of \cite{Cay2017SIMAX} proposed a variant of \cref{Alg:auxjbd} in which   the null space of $\mathbf M$ in step 1 is replaced\footnote{In  noisy cases, the exact null space of $\mathbf M$ is always one-dimensional and spanned by the vectorized identity matrix.\label{ftn:test}}
   by  the subspace spanned by the $\tilde{R}\leq R$ smallest right singular vectors of $\mathbf M$, $\operatorname{vec}({\mathbf U_1}),\dots,\operatorname{vec}({\mathbf U_{\tilde R}})$, and   the simultaneous EVD problem in step 2 is replaced by  the EVD of single matrix
  $\mathbf Z$, where $\mathbf Z$ is a generic linear combination of $\mathbf U_1,\dots,\mathbf U_{\tilde R}$. 
  The block-diagonal matrices $\mathbf D_q$ in step 3 can be found without explicitly computing the inverse of $\mathbf N$ by solving the linear set of equations $\mathbf N\mathbf D_q\mathbf N^T= \matrixU_q$ in the least squares sense.
%
%
  Although the simultaneous EVD in step 2 is replaced by the EVD  of a single matrix $\mathbf Z$, the  experiments in \cite{Cay2017SIMAX} show that  the proposed variant of  \cref{Alg:auxjbd} may outperform optimization based algorithms. On the other hand, it is clear that \tcr{the loss of “diversity” when} replacing the $\tilde R$ matrices in step $2$ by a single generic linear combination may result in a poor estimate of $\mathbf N$  and  also in a wrong detection of $d_1,\dots,d_R$ \tcr{(cf. also the discussion for CPD in \cite{NVNagain})}.  That is why in this paper we will use the following  (still simple but more robust) procedure to compute an approximate   solution of the simultaneous EVD in step 2. (Note that the simultaneous EVD is (obviously) a new concept by itself, for which no dedicated numerical algorithms are available yet and  their derivation is outside the scope of this paper.)  First, we stack the matrices  $\mathbf U_1,\dots,\mathbf U_{\tilde R}$ into an $\tilde R\times K\times K$ tensor $\mathcal U$ and interpret the  simultaneous EVD in step 2 as a 
      structured  decomposition of $\mathcal U$ into a sum of ML rank-$(1,1,1)$ terms (i.e., just rank-$1$ terms): 
    \begin{equation}
    \mathcal U=\sum\limits_{k=1}^{K}\mathbf a_k\circ(\mathbf b_k\mathbf c_k^T)\ \text{ or }
    \mathbf U_r = \mathbf C\operatorname{diag}(a_{r1},\dots,a_{rK})\mathbf B^T,\quad r=1,\dots,\tilde R,
    \label{eq:somePD}
    \end{equation}
      where 
$    \mathbf B^T=\mathbf P^T\mathbf   N^{-1}$, $ \mathbf C=\mathbf N\mathbf P$ (implying that $\mathbf B=\mathbf C^{-T}$),
    \begin{equation}
        \operatorname{diag}(a_{r1},\dots,a_{rK}) = \mathbf P^T\Bdiag(\lambda_{1r}\mathbf I_{d_1},\dots,\lambda_{Rr}\mathbf I_{d_R})\mathbf P,\quad
     r=1,\dots,\tilde R.
     \label{eq:diaga=diaglambda}
    \end{equation}
    and $\mathbf P$ is an arbitrary permutation matrix. If $\mathbf P=\mathbf I_K$, then, by \cref{eq:diaga=diaglambda},
    \begin{equation}
    \mathbf a_1=\dots=\mathbf a_{d_1}=[\lambda_{11}\ \dots\ \lambda_{1\tilde R}]^T, \mathbf a_{d_1+1}=\dots=\mathbf a_{d_1+d_2}=
    [\lambda_{21}\ \dots\ \lambda_{2\tilde R}]^T,\dots 
    \label{eq:a1dotsaK}
    \end{equation}
         If $\mathbf P$ is not the identity, then the vectors $\mathbf a_1,\dots,\mathbf a_K$ can be permuted such that    
    \cref{eq:a1dotsaK} holds.    
    It can   easily be shown that,  in the exact case,   decomposition \cref{eq:somePD} is minimal, that is,  \cref{eq:somePD} is a CPD of $\mathcal U$, and that the  constraint  $\mathbf B=\mathbf C^{-T}$ holds for any solution of \cref{eq:somePD}.
         
    There exist many optimization based algorithms  that can compute the CPD of $\mathcal U$ in the  least squares sense (see, for instance, \cite{tensorlab3.0}).  Recall from  \cref{ftn:test} that, also in the noisy case,  $\mathbf U_{\tilde R}$ can be taken equal to a scalar multiple of the identity matrix.  This actually allows us to enforce the constraint $\mathbf B=\mathbf C^{-T}$  by setting $\mathbf U_{\tilde R}=\omega\mathbf I_K$, where   $\omega$ is a weight coefficient chosen by the user.
     Finally, clustering the  $K$ vectors $\mathbf a_k\in\mathbb F^{\tilde R}$ into $R$  clusters (modulo sign and scaling) we obtain the values $d_1,\dots,d_R$ as the sizes of clusters and also the permutation matrix $\mathbf P$. Then we set
     $\mathbf N=\mathbf C\mathbf P^T$.
   
 \section{Our contribution} \label{sec:contribfortable}
   Before stating the main results (\cref{subsubsection213,sec:genuniq}),  we present necessary conditions for uniqueness (\cref{subsubsectionnecess}),
  explain the key idea behind our derivation (\cref{subsubsectionkeyidea}),  introduce some  notations (\cref{subsec:222}) and a convention (\cref{subsubs:convention}).
  \subsection{Necessary conditions for uniqueness}\label{subsubsectionnecess}
 Let $\mathcal T\in\fF^{I\times J\times K}$ admit the \ML  decomposition \cref{eq:LrLr1}. It was shown in \cite[Theorem 2.4]{LievenLrLr1}
 that if the decomposition of $\mathcal T$ into a sum of \MLatmost terms is unique, then
  $\mathbf A$ does not have proportional columns (trivial) and  the following condition holds:
 \begin{equation}
 \begin{split}
 \text{for every vector }\mathbf{\mathbf w}\in\fF^R\ \text{that has at least two nonzero entries,} \\
 \text{the rank of the matrix }\sum\limits_{r=1}^R w_r\mathbf E_r \text{ is greater than } \max\limits_{\{r: w_r\ne 0\}}L_r.
 \end{split}\label{eq:necesswEr}
 \end{equation}
 In the following theorem we generalize well-known necessary conditions
 	for uniqueness of the CPD (see \cite{PartI} and  references therein) to the decomposition into a sum of \MLatmost terms.
The condition in  \cref{item:necc2} is more restrictive than \cref{eq:necesswEr} but is easier to check. 

 \begin{theorem}\label{thm:necessity}
 	Let $\mathcal T\in\fF^{I\times J\times K}$ admit the \ML  decomposition \cref{eq:LrLr1mainBC},
 	i.e., $r_{\mathbf B_r}=r_{\mathbf C_r}=L_r$ for all $r$.
 	 	If the decomposition of $\mathcal T$ into a sum of \MLatmost terms is unique, then the following statements hold:
 	\begin{statements}
 		\item\label{item:necc2}  the matrix  $[\operatorname{vec}(\mathbf E_1)\ \dots\ \operatorname{vec}(\mathbf E_R)]$ has full column rank, where
 		$\mathbf E_r:=\mathbf B_r\mathbf C_r^T$ for all $r$;
 		\item\label{item:necc4}  the matrix  $[\mathbf a_1\otimes \mathbf B_1\ \dots\ \mathbf a_R\otimes \mathbf B_R]$   has full column rank;
 		\item\label{item:necc3}  the matrix  $[\mathbf a_1\otimes \mathbf C_1\ \dots\ \mathbf a_R\otimes \mathbf C_R]$  has full column rank.
 	\end{statements}
 \end{theorem}
 \begin{proof}The three statements come from the three matrix representations \cref{eq:unf_T_1}, \cref{eq:unf_T_3}, and \cref{eq:unf_T_2}. The details of the proof are given in \cref{sec:simplethms}.
 \end{proof}
   \subsection{The key idea}\label{subsubsectionkeyidea}
  Let $\mathcal T\in\fF^{I\times J\times K}$ admit the \ML decomposition  \cref{eq:LrLr1}, 
  and let   $\mathbf T_1,\dots,\mathbf T_K\in\fF^{I\times J}$ denote the frontal slices of $\mathcal T$, $\mathbf T_k:=(t_{ijk})_{i,j=1}^{I,J}$.
  It is clear that
  \begin{equation}
  f_1\mathbf T_1 + \dots +f_K\mathbf T_K = \sum\limits_{k=1}^Kf_k\sum\limits_{r=1}^R\mathbf a_r\mathbf e_{k,r}^T=\sum\limits_{r=1}^R\mathbf a_r\sum\limits_{k=1}^K \mathbf e_{k,r}^Tf_k
  =
  \sum\limits_{r=1}^R\mathbf a_r(\mathbf E_r\mathbf f)^T,\label{eq:lincombfT}
  \end{equation}
  where $\mathbf e_{k,r}$ denotes the $k$th column of $\mathbf E_r$. Thus, if $\mathbf f$ belongs to the null space of all but one of the matrices $\mathbf E_1,...,\mathbf E_R$ , then
  $f_1\mathbf T_1 + \dots +f_K\mathbf T_K$ is rank-$1$ and its column space is spanned by a column of $\mathbf A$.
  We will make  assumptions on $\mathbf A$ and $\mathbf E_1,\dots,\mathbf E_R$ that  guarantee that the    identity $f_1\mathbf T_1 + \dots +f_K\mathbf T_K=\mathbf z\mathbf y^T$  holds if and only if   $\mathbf z$ is proportional to a column of $\mathbf A$  and  $\mathbf f$ belongs to the null space of all matrices $\mathbf E_1,\dots,\mathbf E_R$ but one:
  \begin{align}
  f_1\mathbf T_1 + \dots +f_K\mathbf T_K=\mathbf z\mathbf y^T \ \Leftrightarrow\ &\exists r \text{ such that } \mathbf z=c\mathbf a_r,\ \mathbf Z_r  \mathbf f=\mathbf 0  \text{ and } \tcr{\mathbf E_r\mathbf f\ne \mathbf 0},
  \label{eq:iffrank1}\\
  &\ \text{ where }\mathbf Z_r := [\mathbf E_1^T\ \dots\ \mathbf E_{r-1}^T\ \mathbf E_{r+1}^T\ \dots\ \mathbf E_R^T]^T.\nonumber 
  \end{align}
    In our algorithm we use $\mathcal T$ to construct a $\rubinom{I}{2}\rubinom{J}{2}\times K^2$ matrix $\mathbf R_2(\mathcal T)$ such that
 the following equivalence holds true:
 \begin{equation}
 \mathbf f\in\fF^K\ \text{is a solution of }\ 
\mathbf R_2(\mathcal T)(\mathbf f\otimes \mathbf f)=\mathbf 0 \qquad \Leftrightarrow \qquad  r_{f_1\mathbf T_1 + \dots +f_K\mathbf T_K}\leq 1.
 \label{eq:rank1eqR20}
 \end{equation}
  By \cref{eq:lincombfT,eq:rank1eqR20,eq:iffrank1},  the set of all solutions of  
  \begin{equation}
\mathbf R_2(\mathcal T)(\mathbf f\otimes \mathbf f)=\mathbf 0\label{eq:R2T0}
  \end{equation}
   is the union of the subspaces $\nullsp{\mathbf Z_1},\dots,\nullsp{\mathbf Z_R}$ and any nonzero solution of \cref{eq:R2T0} gives us a column of $\mathbf A$.   We  establish a link between \cref{eq:R2T0} and    S-JBD problem \cref{eq:JBDaux}. By solving the S-JBD problem we will be able to find  the subspaces $\nullsp{\mathbf Z_1},\dots,$ $\nullsp{\mathbf Z_R}$
   and  the  entire factor matrix $\mathbf A$,  which will then be used to recover the overall decomposition. 
\subsection{Construction of the matrix  $\mathbf R_2(\mathcal T)$ and its submatrix $\mathbf Q_2(\mathcal T)$} \label{subsec:222} 
In this subsection we present the explicit construction of the matrix $\mathbf R_2(\mathcal T)$ in \cref{eq:rank1eqR20}.
In fact, the construction  follows directly from \cref{eq:rank1eqR20}.
It is clear that
\begin{equation}
r_{f_1\mathbf T_1 + \dots +f_K\mathbf T_K}\leq 1 \quad \Leftrightarrow \quad 
\text{ all } 2\times 2 \text{ minors of } f_1\mathbf T_1+ \dots +f_K\mathbf T_K\ \text{ are zero}.
\label{eq:idea}
\end{equation}
Since there are $\rubinom{I}{2}\rubinom{J}{2}$ minors and since each minor is a weighted sum of $K^2$ 
monomials $f_i f_j$, $1\leq i,j\leq K$, the condition in the RHS of  \cref{eq:idea} can be rewritten as
$ \mathbf R_2(\mathcal T)(\mathbf f\otimes \mathbf f)=\mathbf 0$,
where $\mathbf R_2(\mathcal T)$ is a  $\rubinom{I}{2}\rubinom{J}{2}\times K^2$ matrix whose entries are the second degree polynomials in the entries of $\mathcal T$.
Variants of the following explicit construction of $\mathbf R_2(\mathcal T)$ can be found in  \cite{DeLathauwer2006,LinkGEVD,MikaelCoupledPII}.
\begin{definition}\label{def:R2}
	The
\begin{equation}
\left((i_1+\rubinom{i_2-1}{2}-1)\rubinom{J}{2}+j_1+\rubinom{j_2-1}{2}, (k_2-1)K+k_1\right)\text{-th}\label{eq:ttttindex}
\end{equation}
entry of the $\rubinom{I}{2}\rubinom{J}{2}\times K^2$ matrix  $\mathbf R_2(\mathcal T)$ equals 
\begin{equation}
t_{i_1j_1k_1}t_{i_2j_2k_2}+t_{i_1j_1k_2}t_{i_2j_2k_1}-t_{i_1j_2k_1}t_{i_2j_1k_2}-t_{i_1j_2k_2}t_{i_2j_1k_1},
\label{eq: tttt}
\end{equation}
where
$$
1\leq i_1<i_2\leq I,\ 1\leq j_1<j_2\leq J,\ 1\leq k_1,k_2\leq K.
$$
\end{definition}
\par Since the expression in \cref{eq: tttt} is invariant under the permutation $(k_1,k_2)\rightarrow (k_2,k_1)$, the $((k_2-1)K+k_1)$-th column of  the matrix $\mathbf R_2(\mathcal T)$ coincides with its $((k_1-1)K+k_2)$-th column. 
In other words, the rows of $\mathbf R_2(\mathcal T)$ are vectorized $K\times K$ symmetric matrices,
implying that    $\rubinom{K-1}{2}$ columns of $\mathbf R_2(\mathcal T)$ are repeated twice.
Hence  $\mathbf R_2(\mathcal T)$ is of the form 
\begin{equation}
\mathbf R_2(\mathcal T) = \mathbf Q_2(\mathcal T){\mathbf P}_K^T,
\label{eq:R2Q2P}
\end{equation}
where $\mathbf Q_2(\mathcal T)$ holds the $ \rubinom{K+1}{2}$ unique columns of 
$\mathbf R_2(\mathcal T)$
and ${\mathbf P}_K^T\in\fF^{\rubinom{K+1}{2}\times K^2}$ is a binary $(0/1)$ matrix with exactly one element equal to ``1'' per column.
Formally, $\mathbf Q_2(\mathcal T)$ is defined as follows.
\begin{definition}\label{def:Q2}
 $\mathbf Q_2(\mathcal T)$ denotes the $\rubinom{I}{2}\rubinom{J}{2}\times \rubinom{K+1}{2}$ submatrix of $\mathbf R_2(\mathcal T)$ formed by   the  columns with indices $(k_2-1)K+k_1$, where
$1\leq k_1\leq k_2\leq K$. 
\end{definition}
It can be easily checked that \cref{eq:R2Q2P} holds for $\mathbf P_K$ defined by
\begin{equation}
(\mathbf P_K)_{(k_1-1)K+k_2,j}=
\begin{cases}
1,&\text{if } j=\min(k_1,k_2)+\rubinom{\max(k_1,k_2)}{2},\\
0,&\text{otherwise},
\end{cases}
\label{eq:matrixP_K}
\end{equation}
where $1\leq k_1,k_2\leq K$.

In our algorithm we will work with the smaller matrix $\mathbf Q_2(\mathcal T)$ while in the theoretical development we will use
$\mathbf R_2(\mathcal T)$. 
More specifically, a vector $\mathbf f\in\fF^K$ is a solution of \cref{eq:R2T0} if and only if   $\mathbf f\otimes\mathbf f$ belongs to the intersection of the null space of $\mathbf R_2(\mathcal T)$ and   the subspace of vectorized $K\times K$ symmetric matrices,
\begin{equation}
\vecsym{K}:=\{\operatorname{vec}(\mathbf M):\ \mathbf  M\in\fF^{K\times K},\ \mathbf M=\mathbf M^T\},\quad \dim(\vecsym{K})=\rubinom{K+1}{2}.
\label{eq:defvecsymM}
\end{equation}
By \cref{eq:R2Q2P}, the intersection can actually be recovered from the null space of $\mathbf Q_2(\mathcal T)$ as
\begin{equation}
\nullsp{\mathbf R_2(\mathcal T)}\cap\vecsym{K} = {\mathbf P}_K({\mathbf P}_K^T{\mathbf P}_K)^{-1} \nullsp{\mathbf Q_2(\mathcal T)}.
\label{eq:nullR2viaNullQ2}
\end{equation}
It is worth noting that the matrix  $\mathbf D:={\mathbf P}_K({\mathbf P}_K^T{\mathbf P}_K)^{-1}$ in \cref{eq:nullR2viaNullQ2} has the following simple form
\begin{equation}
(\mathbf D)_{(k_1-1)K+k_2,j}=
\begin{cases}
1,&\text{if } j=k_1+\rubinom{ k_1 }{2} \text{ and } k_1=k_2,\\
\frac{1}{2},&\text{if }  j=\min(k_1,k_2)+\rubinom{\max(k_1,k_2)}{2} \text{ and } k_1\ne k_2,\\
0,&\text{otherwise}.
\end{cases}
\label{eq:matrixP_KP_K}
\end{equation}
\subsection{Convention $r_{\unf{T}{3}}=K$}\label{subsubs:convention}
The results of this paper rely on equivalence \cref{eq:iffrank1}, which does not hold if the frontal slices  $\mathbf T_1,\dots,\mathbf T_K$ of the tensor $\mathcal T$ are linearly dependent. One can easily verify that $\unf{T}{3}=[\operatorname{vec}(\mathbf T_1)\ \dots\ \operatorname{vec}(\mathbf T_K)]$, implying that  linear independence of $\mathbf T_1,\dots,\mathbf T_K$
is equivalent to  full column rank of $\unf{T}{3}$, i.e., to the condition $r_{\unf{T}{3}}=K$.

Thus, to apply the results of the paper for tensors with $r_{\unf{T}{3}} < K$, one should first ``compress'' $\mathcal T$ to an $I\times J\times \tilde{K}$ tensor $\tilde{\mathcal T}$ such that  $r_{\tilde{\mathbf T}_{(3)}}=\tilde{K}$. Such a compression can, for instance, be done by taking $\tilde{\mathcal T}$ with $\tilde{\mathbf T}_{(3)}$ equal to the ``U'' factor in the compact SVD of $\unf{T}{3}=\mathbf U\mathbf S\mathbf V^H$. In this case, by \cref{eq:unf_T_3},
$$
\tilde{\mathbf T}_{(3)}:=\mathbf U=\unf{T}{3}\mathbf V\mathbf S^{-1}=[\mathbf a_1\otimes\mathbf B_1\ \dots\ \mathbf a_R\otimes\mathbf B_R](\mathbf S^{-1}\mathbf V^T\mathbf C)^T,
$$
implying that  $\tilde{\mathcal T}$ and $\mathcal T$ share the first two factor matrices and that \tcr{the  slices of $\tilde{\mathcal T}$ are obtained from
linear mixtures of the \tcr{$I\times J$ matrix} slices of $\mathcal T$}. 
If the  decomposition of $\tilde{\mathcal T}$ into a sum of \MLatmost terms is unique, then, by \cref{item:necc4} of \cref{thm:necessity},
the matrix $[\mathbf a_1\otimes\mathbf B_1\ \dots\ \mathbf a_R\otimes\mathbf B_R]$ has full column rank.
Thus, when the matrices $\mathbf A$ and $\mathbf B$ are obtained from $\tilde{\mathcal T}$, the remaining matrix $\mathbf C$ can be found from \cref{eq:unf_T_3} as
$
\mathbf C=\left([\mathbf a_1\otimes\mathbf B_1\ \dots\ \mathbf a_R\otimes\mathbf B_R]^{\dagger}\unf{T}{3}\right)^T.
$
For future reference, we summarize the above discussion in \cref{item:thm:sonvention:statement1} of the following theorem.
\Cref{item:thm:sonvention:statement2}  is the generic version of \cref{item:thm:sonvention:statement1} and   can be proved in a similar way.
\begin{theorem}\label{thm:convention} \qquad 
	\begin{statements}
		\item	\label{item:thm:sonvention:statement1}
	Let $\mathcal T$ be an $I\times J\times K$ tensor and let $\tilde{\mathcal T}$ be an $I\times J\times \tilde K$ tensor formed by 
	$\tilde K$ linearly independent mixtures of the  $I\times J$ matrix  slices of $\mathcal T$. If the 	decomposition of  $\tilde{\mathcal T}$  into a sum of \MLatmost terms  i) is    unique or, moreover, ii)  is  unique and can be computed by means of (simultaneous) EVD, then  the  same holds true for $\mathcal T$.
\item	\label{item:thm:sonvention:statement2}
	If the 	decomposition of an $I\times J\times \tilde K$ tensor into a sum of \MLatmost terms  i) is   generically unique or, moreover, ii) is generically unique and can \tcr{generically} be computed by means of (simultaneous) EVD, then  the  same holds true for tensors with dimensions $I\times J\times K$, where $K\geq \tilde{K}$.
	\end{statements}
\end{theorem}  

Thus,   in the cases  where the assumption $r_{\unf{T}{3}} =K$ (resp. the assumptions $IJ\geq\sum L_r\geq K$) allows us to simplify the presentation, namely, in \cref{thm: maintheorem,thm: maintheoremABC} (resp. in \cref{thm:maingen}),  we will assume w.l.o.g. that $r_{\unf{T}{3}}=K$ (resp. $\sum L_r\geq K$). 
\subsection{Main uniqueness results and algorithm}\label{subsubsection213}
In \cref{subsub251} we present results on uniqueness and computation of the exact \ML decomposition   \cref{eq:LrLr1}.
In \cref{subsub252} we explain how to compute an approximate solution in the case where the decomposition is not exact. In  \cref{subsub253}
we illustrate our results by examples.
\subsubsection{ Exact  \ML decomposition}\label{subsub251}
In the following theorem both assumptions \cref{item:th1cond1}, \cref{item:th1cond2} need to hold, and at least one  of the assumptions \cref{eq:ranksofFand G} and \cref{item:th1cond3}. In \cref{sta:lemmaApp4} of \cref{lemma: Q2viaABC} below we will show that  \cref{eq:ranksofFand G} actually implies \cref{item:th1cond3}.

By itself, \cref{thm: maintheorem} can be used to show uniqueness of a decomposition, but not only that. As we will explain later, the theorem comes with an algorithm for the actual computation of the decomposition (namely, \cref{algorithm:1}). In this respect, another comment is in order. 
	If one wishes to use \cref{thm: maintheorem} to show uniqueness, and if one wishes to do so via \cref{eq:ranksofFand G}, then there is no need to construct the matrix $\mathbf Q_2(\mathcal T)$ in \cref{item:th1cond3}. On the other hand, \cref{thm: maintheorem} comes with
	\cref{algorithm:1} for the actual computation of the decomposition. In this algorithm we work via the null space of $\mathbf Q_2(\mathcal T)$ (and not just its dimension as in \cref{item:th1cond3}), i.e., matrix $\mathbf Q_2(\mathcal T)$ is constructed, also in cases where the uniqueness by itself follows from \cref{eq:ranksofFand G}.
\begin{theorem} \label{thm: maintheorem} 
	Let $\mathcal T\in\fF^{I\times J\times K}$ admit the \ML decomposition  \cref{eq:LrLr1}, i.e., $r_{\mathbf E_r}=L_r$ for all $r$.
	Assume that 
	\begin{align}
 r_{\unf{T}{3}}&=K \label{item:th1cond1}\ \text{and}\\
 d_r:=\dim\nullsp{\mathbf Z_r}&\geq 1,\qquad r=1,\dots,R,\label{item:th1cond2}
 	\end{align} 
 	where	$\unf{T}{3}$ is defined in \cref{eq:unf_T_3} and  
 	$\mathbf Z_r := [\mathbf E_1^T\ \dots\ \mathbf E_{r-1}^T\ \mathbf E_{r+1}^T\ \dots\ \mathbf E_R^T]^T$.
  	Assume also that
	\begin{gather}
	\begin{gathered}
	k_{\mathbf A}\geq 2\ \text{ and rank of}\ 
	\mathbf F:=[\mathbf E_{r_1}\ \mathbf E_{r_2}\ \dots\ \mathbf E_{r_{R-r_{\mathbf A}+2}}] \text{ is }
	L_{r_1}+\dots+L_{r_{R-r_{\mathbf A}+2}}\\
	 \text{ for all }\ 1\leq r_1<\dots<r_{R-r_{\mathbf A}+2}\leq R \label{eq:ranksofFand G}
	\end{gathered}\\
\shortintertext{or} 
	\dim\nullsp{\mathbf Q_2(\mathcal T)}=\sum\limits_{r=1}^R \rubinom{d_r+1}{2}=:Q, \label{item:th1cond3}
	\end{gather}
	where 	$\mathbf Q_2(\mathcal T)$ is constructed by \cref{def:Q2}.
	 Consider the following conditions:
	\begin{conditions}
		\renewcommand{\theenumi}{\alph{enumi}}    
		\item  \label{item:th1cond3a} $K\geq \sum L_r-\min L_r +1$ and $k_{\mathbf A}\geq 2$;
		\item  \label{item:th1cond5} the matrix $\mathbf A$ has full column rank, i.e., $r_{\mathbf A}=R$;
		\item  \label{item:th1cond6}  $k_{\mathbf A}=r_{\mathbf A}<R$, 
		assumption \cref{eq:ranksofFand G} holds and 
		\begin{equation}
		\begin{gathered}
        \text{rank of } \mathbf G:=[\mathbf E_{r_1}^T\ \mathbf E_{r_2}^T\ \dots\ \mathbf E_{r_{R-r_{\mathbf A}+2}}^T]
         \text{ is }   L_{r_1}+\dots+L_{r_{R-r_{\mathbf A}+2}}\\
        \text{ for all }\ 1\leq r_1<\dots<r_{R-r_{\mathbf A}+2}\leq R;
		\label{eq:ranksofFand G2} 
		\end{gathered}
		\end{equation}
		\item  \label{item:th1cond4}  \tcr{the matrix $[\mathbf E_1^T\ \dots\ \mathbf E_R^T]^T$ has maximum possible rank, namely, $\sum L_r$;} 
		\item  \label{eq:ineqforuni1fm} the inequality 
		\begin{equation*}
		\rubinom{K+1}{2} - Q > \tcr{-\tilde{L}_1\tilde{L}_2} + \sum\limits_{1\leq r_1<r_2\leq R} L_{r_1}L_{r_2} 
		\end{equation*}
		holds, where $\tilde{L}_1$ and $\tilde{L}_2$ denote the two smallest values in $\{L_1,\dots,L_R\}$.
			\end{conditions}
	The following statements hold.
	\begin{statements}
	\item	\label{item:th1stat1}
	        The matrix $\mathbf A$ in the \ML decomposition \cref{eq:LrLr1} can be computed by means of (simultaneous) EVD up to  column  permutation and scaling. 
    \item  \label{item:th1stat2}
	        If  either \cref{item:th1cond5}  or \cref{item:th1cond6} holds,
        	then the overall ML rank-\\ $(1,L_r,L_r)$ decomposition \cref{eq:LrLr1} can be computed by means of (simultaneous) EVD.
  \item  \label{item:th1newstat} If \cref{item:th1cond3a} holds, then  any decomposition of $\mathcal T$ into a sum of \MLatmost terms  has $R$ nonzero terms and  its first factor matrix  \tcr{can be chosen as} $\mathbf A\mathbf P$, where every column of  $\mathbf P\in\fF^{R\times R}$ contains precisely a single $1$ with zeros everywhere else.
   \item \label{item:th1firstfmhard} If \cref{item:th1cond3a,eq:ineqforuni1fm} hold,
    	then   	the first factor matrix of the decomposition of $\mathcal T$ into a sum of \MLatmost terms is unique
   	and can be computed by means of (simultaneous) EVD.
   \item   \label{item:th1stat4}
	       If \cref{item:th1cond3a,item:th1cond5} hold, or \cref{item:th1cond3a,item:th1cond6} hold, or \cref{item:th1cond4} holds,
	       then  the decomposition of $\mathcal T$ into a sum of \MLatmost terms is unique and can be computed by means of (simultaneous) EVD.
 \end{statements}
\end{theorem}
\begin{proof}
	See \cref{sec:proof_of_main_thm}.
\end{proof}
\par We make the following comments on the  assumptions, conditions,  and statements in  \cref{thm: maintheorem}.

1) \tcr{\Cref{item:th1stat1} says that $\mathbf A$ can be computed by means of EVD. On the other hand, \cref{item:th1firstfmhard} says that the first factor matrix is unique and can be computed by means of EVD, under a more restrictive condition. A similar observation can be made for the computation of the entire decomposition  in \cref{item:th1stat2,item:th1newstat}, respectively. What we mean is the following.
All assumptions and conditions in \cref{thm: maintheorem}, except \cref{item:th1cond1}, are formulated in terms of a specific
\ML decomposition of $\mathcal T$, namely, in terms of the matrices $\mathbf A$ and $\mathbf E_1,\dots,\mathbf E_R$. 
 There is a subtlety in the sense that $\mathcal T$ may admit alternative decompositions for which  the assumptions \eqref{item:th1cond2} and \eqref{item:th1cond3} and   \cref{item:th1cond5,item:th1cond6} do not all hold and which cannot necessarily be (partially) found by means of EVD.
 The more restrictive conditions in \cref{item:th1firstfmhard,item:th1stat4} exclude the existence of such alternative decompositions. \Cref{item:th1newstat}
 is a ``transition statement'' in which the alternatives for the first factor matrix are restricted.}
 Thus,   \cref{item:th1stat1,item:th1stat2} \tcr{are mainly meant to cover}   cases  where the  first factor matrix and the overall decomposition, respectively, are not   unique in the sense that there may be alternatives for which the assumptions/conditions do not hold. See \cref{ex:2.4} below for an illustration.

2) The matrix $\mathbf P$ in \cref{item:th1newstat} is a column selection matrix, \tcr{possibly with repeated columns}. Thus, \cref{item:th1newstat} says that  the first factor matrix of any decomposition of $\mathcal T$ into a sum of \MLatmost
terms can be obtained  \tcr{by selecting columns of $\mathbf A$, where   column repetition is allowed but}  the total number of columns \tcr{should be} equal to $R$. 

3) The assumptions in \cref{thm:ll1_gevd}, \cref{thm:ll1_btd1}, and  \cref{thm;LievenBTDKruskal} are symmetric with respect to the last two dimensions while the assumptions and conditions
in \cref{thm: maintheorem} are not. To get another  set of conditions on uniqueness and computation one can just  permute the last two dimensions of $\mathcal T$. 

4)  As in \cref{thm:ll1_gevd} and \cref{thm:ll1_btd1}, the number of \ML terms and the values of $L_r$ are not required to be known in advance;  they are found by the algorithm.

\tcr{5) Assumption \cref{item:th1cond3} means that we require the subspace $\dim\nullsp{\mathbf Q_2(\mathcal T)}$ to have the minimal possible dimension
	(see \cref{sta:lemmaApp3} of \cref{lemma: Q2viaABC} below).}

\tcr{6)  It can be shown that \Cref{item:th1stat4} is a criterion that is ``effective'' in the sense of \cite{2016NVNeffective}.}

\tcr{Instead of the matrices $\mathbf A$ and $\mathbf E_1,\dots,\mathbf E_R$,} \cref{thm: maintheorem}   can also be given in terms of the factor matrices
$\mathbf A$, $\mathbf B$, and $\mathbf C$ (cf.  \cref{thm:ll1_gevd,thm:ll1_btd1,thm;LievenBTDKruskal}).  Namely, substituting  $\mathbf E_r=\mathbf B_r\mathbf C_r^T$ and $\mathcal T = \sum\mathbf a_r\circ(\mathbf B_r\mathbf C_r^T)$, in the expressions for $\mathbf Z_r$, $\mathbf F$, $\mathbf G$, $[\mathbf E_1^T\ \dots\ \mathbf E_R^T]^T$ and $\mathbf Q_2(\mathcal T)$, respectively,   we obtain the following result. 
\begin{theorem} \label{thm: maintheoremABC}  
	Let $\mathcal T\in\fF^{I\times J\times K}$ admit the \ML decomposition  \cref{eq:LrLr1mainBC}, i.e., $r_{\mathbf B_r}=
	r_{\mathbf C_r}=L_r$ for all $r$.
	Assume that 
	\begin{align}
	&\text{the matrix }\ [\mathbf a_1\otimes\mathbf B_1\ \dots\ \mathbf a_R\otimes\mathbf B_R]\mathbf C^T\ \text{ has full column rank and} \label{item:th1cond1ABC}\\
	&d_r:=\dim\nullsp{\mathbf Z_{r,\mathbf C}}\geq 1,\qquad r=1,\dots,R,\label{item:th1cond2ABC}
	\end{align}
	where  
	$\mathbf Z_{r,\mathbf C} := [\mathbf C_1\ \dots\ \mathbf C_{r-1}\ \mathbf C_{r+1}\ \dots\ \mathbf C_R]^T$. Assume also that
	\begin{gather}
	 k_{\mathbf A}\geq 2 \text{  and } k_{\mathbf B}'\geq R-r_{\mathbf A}+2\label{eq:ranksofFand GABC} \\
	\shortintertext{or\footnotemark}
	\dim\nullsp{\Phi(\mathbf A,\mathbf B)\mathbf S_2(\mathbf C)^T}=\sum\limits_{r=1}^R \rubinom{d_r+1}{2}=:Q, \label{item:th1cond3ABC}
	\end{gather}
	\footnotetext{In \cref{sta:lemmaApp4} of \cref{lemma: Q2viaABC} below we show that \cref{eq:ranksofFand GABC} implies \cref{item:th1cond3ABC}.}
	where the matrices $\Phi(\mathbf A,\mathbf B)$ and $\mathbf S_2(\mathbf C)$ are defined in \cref{eq:matrixPhi,eq:matrixS2}   below\footnote{The  definitions of $\Phi(\mathbf A,\mathbf B)$ and $\mathbf S_2(\mathbf C)$  require additional notations and
			are postponed to \cref{sec:mainidentity} for the sake of readability. Here we just mention that each entry of 
		$\Phi(\mathbf A,\mathbf B)$ is a product of a $2\times 2$ minor of $\mathbf A$ and a $2\times 2$ minor of $\mathbf B$ and  that each entry of 
	$\mathbf S_2(\mathbf C)$ is of the form $c_{i_1j_1}c_{i_2j_2}+c_{i_1j_2}c_{i_2j_1}$.}.
	Consider the following conditions:
	\begin{conditions}
		\renewcommand{\theenumi}{\alph{enumi}}    
		\item  
		$K\geq \sum L_r-\min L_r +1$ and $k_{\mathbf A}\geq 2$;
		\item  
		 the matrix $\mathbf A$ has full column rank, i.e., $r_{\mathbf A}=R$;
		\item  \label{item:th1cond6ABC}  $k_{\mathbf A}=r_{\mathbf A}<R $, \cref{eq:ranksofFand GABC} holds and \tcr{$ k_{\mathbf C}'\geq R-r_{\mathbf A}+2$};
	\item  \label{item:th1cond4ABC}    $K=\sum\limits_{r=1}^R L_r$ (implying that $\mathbf C$ is $ K\times K$ nonsingular  and that
		$d_r = L_r$ for all $r$);
		\item   the inequality 
		\begin{equation*}
		\rubinom{K+1}{2} - Q > \tcr{-\tilde{L}_1\tilde{L}_2} + \sum\limits_{1\leq r_1<r_2\leq R} L_{r_1}L_{r_2} 
		\end{equation*}
		holds, where $\tilde{L}_1$ and $\tilde{L}_2$ denote the two smallest values in $\{L_1,\dots,L_R\}$.
	\end{conditions}
	Then \cref{item:th1stat1,item:th1stat2,item:th1newstat,item:th1firstfmhard,item:th1stat4} in \cref{thm: maintheorem} hold.
	\end{theorem}
\begin{proof}
 	The proof is given in \cref{sec:simplethms}.
\end{proof}
\Cref{item:th1stat4} in \cref{thm: maintheoremABC}/\cref{thm: maintheorem} allows us  to trade full column rank of the factor matrices $\mathbf B$ and $\mathbf C$ for a higher $k$-rank of $\mathbf A$ than
in \cref{thm:ll1_gevd}.
 In particular the following result   \tcr{can be used in  cases where none of the factor matrices has full column rank}.
 \begin{corollary}\label{corollary:verynew}
 	Let $\mathcal T\in\fF^{I\times J\times K}$ admit the \ML decomposition  \cref{eq:LrLr1mainBC}, i.e., $r_{\mathbf B_r}=
 	r_{\mathbf C_r}=L_r$ for all $r$. Assume that 
 	\begin{equation}
 r_{\mathbf C}\geq \sum L_r-\min L_r +1,\ \  \ \ k_{\mathbf B}'\geq R-r_{\mathbf A}+2 
 	\ \ \text{ and }\ \ k_{\mathbf A}\geq 2.\label{eq:corollary261}
 	\end{equation}
 	Then the decomposition of $\mathcal T$ into a sum of \MLatmost terms is unique and can be computed by means of (simultaneous) EVD 
 	if
 	\begin{equation}
 	 \text{either } \ r_{\mathbf A}=R\ \ \ \ \ \ \text{ or  }\ \ \ \ \ \ k_{\mathbf A}=r_{\mathbf A}<R\ \text{ and  }\ k_{\mathbf C}'\geq R-r_{\mathbf A}+2.\label{eq:corollary262}
 	\end{equation}
  \end{corollary}
\begin{proof}
	The proof is given in \cref{sec:simplethms}.
\end{proof}

 The   algebraic procedure that will result from     \cref{thm: maintheorem} (or \cref{thm: maintheoremABC}) is summarized in \cref{algorithm:1}. In this subsection we explain how     \cref{algorithm:1} computes the exact  \ML decomposition   \cref{eq:LrLr1}. In \cref{subsub252}  we  will explain  how the steps in  \cref{algorithm:1} can be modified   to compute an approximate  \ML decomposition  of $\mathcal T$.
 
 In Phase I we recover the first factor matrix.
In steps $1-3$ we compute a basis $\vectoru_1,\dots,\vectoru_Q$ of the subspace $\nullsp{\mathbf R_2(\mathcal T)}\cap\vecsym{K}$. The computation relies on 
identity \cref{eq:nullR2viaNullQ2}:  we construct the  smaller matrix $\mathbf Q_2(\mathcal T)$, compute a basis of  $\nullsp{\mathbf Q_2(\mathcal T)}$ and map it to a basis of $\nullsp{\mathbf R_2(\mathcal T)}\cap\vecsym{K}$. In steps 4 and 5 we construct S-JBD problem \cref{eq:JBDaux} and solve it by \cref{Alg:auxjbd}.  
 
 It will be proved \tcr{(see proof of the first statement of \cref{thm: maintheorem})} that  submatrix $\mathbf N_r\in\fF^{K\times d_r}$ of the matrix  $\mathbf N=[\mathbf N_1\ \dots\ \mathbf N_R]$ computed in step 5 holds a basis of
$\nullsp{\mathbf Z_{r}}$, $r=1,\dots,R$. In addition, it can   be easily verified that
$\nullsp{\mathbf Z_{r}}=\nullsp{\mathbf Z_{r,\mathbf C}}$, so we have that
\begin{equation}
\mathbf N_r^T[\mathbf C_1\ \dots\ \mathbf C_{r-1}\ \mathbf C_{r+1}\ \dots\ \mathbf C_R] =\mathbf O,\qquad r = 1,\dots,R.
\label{eq:NrinnullCk} 
\end{equation}  
In step 6 we use \cref{eq:NrinnullCk}  to compute the columns of $\mathbf A$: since by \cref{eq:NrinnullCk} and \cref{eq:unf_T_3},
 \begin{equation}
 \begin{split}
 [\mathbf N_r^T\mathbf H_1^T\ \dots\ \mathbf N_r^T\mathbf H_I^T]=&\mathbf N_r^T\unf{T}{3}^T= \mathbf N_r^T
 \mathbf C[\mathbf a_1\otimes\mathbf B_1\ \dots\ \mathbf a_R\otimes\mathbf B_R]^T
      =\\
 &\mathbf N_r^T\mathbf C_r(\mathbf a_r^T\otimes\mathbf B_r^T)=
  ( 1\otimes\mathbf N_r^T\mathbf C_r)(\mathbf a_r^T\otimes\mathbf B_r^T)=\\
 &\mathbf a_r^T\otimes (\mathbf N_r^T\mathbf C_r\mathbf B_r^T)=
 \mathbf a_r^T\otimes (\mathbf N_r^T\mathbf E_r^T),\qquad r=1,\dots,R,
 \end{split}
 \label{eq:a_rviaN_r}
 \end{equation}
%
\begin{algorithm}
	\caption{Computation of \ML decomposition   \cref{eq:LrLr1}   under various conditions expressed in \cref{thm: maintheorem}
	} 
	\label{algorithm:1}	
	\begin{algorithmic}[1]
		\INPUT{tensor $\mathcal T\in\fF^{I\times J\times K}$ admitting decomposition \cref{eq:LrLr1}
		}
		
		{\bf Phase I} (computation of $\mathbf A$)
		\STATE{Construct the $\rubinom{I}{2}\rubinom{J}{2}$-by-$\rubinom{K+1}{2}$ matrix $\mathbf Q_2(\mathcal T)$ as in \cref{def:Q2}}
		\STATE{Find  $\mathbf g_q \in\fF^{\rubinom{K+1}{2}}$, $q=1,\dots,Q$ that form a basis of $\nullsp{\mathbf Q_2(\mathcal T)}$, where  $Q=\rubinom{d_1+1}{2}+\dots+\rubinom{d_R+1}{2}$}
		\STATE{Compute  $\vectoru_q  := \mathbf D\mathbf g_q\in\fF^{K^2}$, $q=1,\dots,Q$, where $\mathbf D$ is defined in \cref{eq:matrixP_KP_K}
		}
		\STATE{For each $q=1,\dots,Q$  reshape $\vectoru_q$ into the $K\times K$ symmetric matrix $\matrixU_q$}
		\STATE{Compute $\mathbf N$ and the values $R$, $d_1,\dots,d_R$} in S-JBD problem \cref{eq:JBDaux} by  \cref{Alg:auxjbd}
		\STATE{For each $r=1,\dots,R$ take $\mathbf a_r$ equal to the vector that generates the row space of $[\operatorname{vec}(\mathbf N_r^T\mathbf H_1^T)\ \dots\ \operatorname{vec}(\mathbf N_r^T\mathbf H_I^T)]$\tcr{, where $\mathbf H_i:= (t_{ijk})_{j,k=1}^{J,K}$}}
		\qquad\\
		\qquad\\
		{\bf Phase II} (computation of the overall decomposition under one of the conditions \ref{item:th1cond4}, \ref{item:th1cond5}, or \ref{item:th1cond6})\\
		{\em Case 1: \cref{item:th1cond4} in \cref{thm: maintheorem} holds} 
		\STATE{\qquad For each $r=1,\dots,R$ compute the vector that generates the column space of\\ \qquad  $[\operatorname{vec}(\mathbf N_r^T\mathbf H_1^T)\ \dots\ \operatorname{vec}(\mathbf N_r^T\mathbf H_I^T)]$ and reshape it into the matrix $\mathbf B_r$}
		\STATE{\qquad Compute $\mathbf C$ from  the set of linear equations \\
			{\centering $\unf{T}{3}=[\mathbf a_1\otimes\mathbf B_1\ \dots\ \mathbf a_R\otimes\mathbf B_R]\mathbf C^T$\par}
		}
		\STATE{\qquad For each $r=1,\dots, R$ set $\mathbf E_r=\mathbf B_r\mathbf C_r^T$}\\
		\qquad\\
		{\em Case 2: \cref{item:th1cond5} in \cref{thm: maintheorem} holds} 
		\STATE{
			\qquad Compute $\mathbf E_1,\dots,\mathbf E_R$ by solving the  set of linear equations\\
			{\centering $\unf{T}{1}=[\operatorname{vec}(\mathbf E_1)\ \dots\ \operatorname{vec}(\mathbf E_R)]\mathbf A^T$\par}
		}
		\qquad\\
		{\em Case 3: \cref{item:th1cond6} in \cref{thm: maintheorem} holds}
		\STATE{\qquad Choose (possibly overlapping) subsets $\Omega_1,\dots,\Omega_M\subset\{1,\dots,R\}$ such that\\
			{\centering
				\qquad $\operatorname{card}(\Omega_1)=\dots=\operatorname{card}(\Omega_M)=R-r_{\mathbf A}+2$ and $\{1,\dots,R\}=\Omega_1\cup\dots\cup \Omega_M$\par}
		}
		\STATE{\qquad {\bf for} each $m=1,\dots,M$ {\bf do}}
		\STATE{\quad \qquad Find linearly independent vectors $\mathbf h_1,\mathbf h_2\in\fF^I$ that belong to the column\\ 
			\quad \qquad  space of $\mathbf A$ and satisfy\\ 
			{\centering 
				$\mathbf a_r^T\mathbf h_1=\mathbf a_r^T\mathbf h_2=0$ for all $r\in \{1,\dots,R\}\setminus\Omega_m$
				\par}}
		\STATE{\quad \qquad Compute the $2\times J\times K$ tensor $\mathcal Q^{(m)}$ with  $\mathbf Q_{(1)}^{(m)}=\unf{T}{1}[\mathbf h_1\ \mathbf h_2]$}
		\STATE{\quad \qquad Compute the ML rank-$(1,L_r,L_r)$ decomposition of $\mathcal Q^{(m)}$ by the EVD \\ 
			\quad \qquad in \cref{thm:ll1_gevd}:\\
			{\centering
				$\mathcal Q^{(m)}=\sum\limits_{r\in\Omega_m} \hat{\mathbf a}_r\circ\hat{\mathbf E}_r$\qquad (the vectors $\hat{\mathbf a}_r$ are a by-product)  \par}
		}
		\STATE{\qquad {\bf end for}}
		\STATE{\qquad Compute   $\mathbf x$ from the linear equation\\
			{\centering $[\mathbf a_1\otimes\operatorname{vec}(\hat{\mathbf E}_1)\ \dots\ \mathbf a_r\otimes\operatorname{vec}(\hat{\mathbf E}_R)]\mathbf{x}=\operatorname{vec}(\unf{T}{1})$\par}
		}
		\STATE{\qquad For each $r=1,\dots, R$ set $\mathbf E_r=x_r\hat{\mathbf E}_r$}
		\qquad\\ \qquad\\
		\OUTPUT{{Matrices $\mathbf A\in\fF^{I\times R}$, $\mathbf E_1,\dots,\mathbf E_R\in\fF^{J\times K}$  such that  \cref{eq:LrLr1} holds}}
	\end{algorithmic}
\end{algorithm}
\raggedright{it follows that}\justifying
\begin{equation}\label{eq:22prime}
[\operatorname{vec}(\mathbf N_r^T\mathbf H_1^T)\ \dots\ \operatorname{vec}(\mathbf N_r^T\mathbf H_I^T)]=\operatorname{vec}(\mathbf N_r^T\mathbf E_r^T)\mathbf a_r^T,\qquad r=1,\dots,R,
\end{equation}
implying that $\mathbf a_r$ is the vector that generates the row space of only right singular vector of $[\operatorname{vec}(\mathbf N_r^T\mathbf H_1^T)\ \dots\ \operatorname{vec}(\mathbf N_r^T\mathbf H_I^T)]$
that corresponds  to a nonzero singular value.

In Phase II we   recover the overall decomposition. Since, by \cref{thm: maintheorem} (or \cref{thm: maintheoremABC}), the computation is possible if  at least one of the conditions \ref{item:th1cond4}, \ref{item:th1cond5}, or \ref{item:th1cond6}    holds, we consider three cases.

{\em Case 1:}   \cref{item:th1cond4ABC} in \cref{thm: maintheoremABC} implies  that  $\mathbf C$ is a $K\times K$ nonsingular matrix and that $K=\sum d_r=\sum L_r$. 
Since the $K\times \sum d_r$ matrix $\mathbf N$ computed in step 5 has full column rank, it follows that
$\mathbf N$ is also  $K\times K$ nonsingular. Since, by \cref{eq:NrinnullCk},  
$$
\mathbf N^T\mathbf C= [\mathbf N_1\ \dots \mathbf N_R]^T [\mathbf C_1\ \dots\ \mathbf C_R]=
\Bdiag(\mathbf N_1^T\mathbf C_1,\dots,\mathbf N_R^T\mathbf C_R),
$$
we have that $\mathbf C = \mathbf N^{-T}\Bdiag(\mathbf N_1^T\mathbf C_1,\dots,\mathbf N_R^T\mathbf C_R)$.
%
Since $\mathbf C$ and $\mathbf N$ are  nonsingular, the matrices $\mathbf N_r^T\mathbf C_r\in\fF^{L_r\times L_r}$ are also nonsingular. 
  To compute $\mathbf B_1,\dots,\mathbf B_R$ we use  identity \cref{eq:22prime}.
In step 7  we compute  $\operatorname{vec}(\mathbf N_r^T\mathbf E_r^T)$ as the vector that generates the column space of \tcr{the}  left  singular vector of $
[\operatorname{vec}(\mathbf N_r^T\mathbf H_1^T)\ \dots\ \operatorname{vec}(\mathbf N_r^T\mathbf H_I^T)]$ \tcr{corresponding to the only nonzero singular value}. In addition, $(\mathbf N_r^T\mathbf E_r^T)^T=\mathbf B_r(\mathbf N_r^T\mathbf C_r)^T$ by definition of $\mathbf E_r$.  W.l.o.g. we set $\mathbf B_r$ equal to $(\mathbf N_r^T\mathbf E_r^T)^T$, as the nonsingular factor $(\mathbf N_r^T\mathbf C_r)^T$ can be compensated for in the factor $\mathbf C$. As such, in step 8 we finally recover $\mathbf C$ from  \cref{eq:unf_T_3}.

It is worth noting that   the vectors $\mathbf a_r$ in step 6 and the matrices $\mathbf B_r$ in step 7 can be computed simultaneously.  Indeed, by \cref{eq:22prime}, $\mathbf B_r$ and $\mathbf a_r$, can be found from $\operatorname{vec}(\mathbf B_r)\mathbf a_r^T=[\operatorname{vec}(\mathbf N_r^T\mathbf H_1^T)\ \dots\ \operatorname{vec}(\mathbf N_r^T\mathbf H_I^T)]$.

{\em Case 2}:  \cref{item:th1cond5} implies that $\mathbf A$ has full column rank. Hence, by \cref{eq:unf_T_1},
$[\operatorname{vec}(\mathbf E_1)\ \dots\ \operatorname{vec}(\mathbf E_R)]=\unf{T}{1}(\mathbf A^T)^\dagger$.

{\em Case 3:} We assume that \cref{item:th1cond6} holds. In steps $11-18$ we use the matrix $\mathbf A$ estimated in Phase I and the tensor $\mathcal T$ to recover the matrices $\mathbf E_1,\dots,\mathbf E_R$. \tcr{There exist $\rubinom{R}{R-r_{\mathbf A}+2}$ subsets of $\{1,\dots,R\}$ of cardinality $R-r_{\mathbf A}+2$. In principle, one can choose any $M$ of them that cover the set $\{1,\dots,R\}$.} \tcr{(One can, for instance, choose $M=\lceil\frac{R}{R-\tcr{r_{\mathbf A}}+2}\rceil$ and set 
	$\Omega_m = \{(m-1)(R-\tcr{r_{\mathbf A}}+2)+1,\dots,m(R-\tcr{r_{\mathbf A}}+2)\}$ for $m=1,\dots,M-1$ and
	$\Omega_M = \{\tcr{r_{\mathbf A}}-1,\dots,R\}$, where $\lceil x\rceil$ denotes the least integer greater than or equal to $x$.)}
To explain  steps $12-16$ we assume for  simplicity that, in step 11, $\Omega_1=\{1,\dots,R-\tcr{r_{\mathbf A}}+2\}$. In steps $13$ and $14$ we project out the last $\tcr{r_{\mathbf A}}-2$ terms in the \ML decomposition of $\mathcal T$. It can be shown that the tensor $\mathcal Q^{(1)}$ constructed in step $14$ admits the \ML decomposition $\mathcal Q^{(1)}=\sum\limits_{r=1}^{R-\tcr{r_{\mathbf A}}+2} \hat{\mathbf a}_r\circ\hat{\mathbf E}_r$, where $\hat{\mathbf a}_r=[\mathbf h_1\ \mathbf h_2]^T\mathbf a_r\in\fF^2$ and  $\hat{\mathbf E}_r$ is proportional to $\mathbf E_r$, $r=1,\dots,R-\tcr{r_{\mathbf A}}+2$.
\tcr{By  \cref{item:th1cond6},} $\mathcal Q^{(1)}$ satisfies the assumptions in \cref{thm:ll1_gevd}. Thus, the \ML decomposition $\mathcal Q^{(1)}$ is unique and can be computed by means of (simultaneous) EVD. The remaining matrices $\mathbf E_{R-\tcr{r_{\mathbf A}}+3},\dots,\mathbf E_{R}$ can be estimated up to scaling factors in a similar way by choosing other subsets $\Omega_m$. 
In step 17 we use \cref{eq:unf_T_1} to compute the scaling factors $x_1,\dots,x_R$ such that $\mathcal T=\sum\limits_{r=1}^R\mathbf a_r\circ(x_r\hat{\mathbf E}_r)$.

 One may wonder  what to do if several of conditions \ref{item:th1cond5}, \ref{item:th1cond6} or \ref{item:th1cond4} hold together. \Cref{item:th1cond5,item:th1cond6} are mutually exclusive.
 If \cref{item:th1cond4,item:th1cond5} hold, then uniqueness and computation  follow already from
\cref{thm:ll1_btd1}. Indeed, \cref{item:th1cond4,item:th1cond5} in \cref{thm: maintheoremABC}  imply that the matrices
$\mathbf A$ and $\mathbf C$ have full column rank, and, by \cref{Lemma:redtobtdI}, assumption \cref{item:th1cond3ABC} is more restrictive than the assumption $r_{[\mathbf B_i\ \mathbf B_j]}\geq \max(L_i,L_j)+1$ for all $1\leq i<j\leq R$. It is less clear if 
\cref{algorithm:1} can further be simplified if \cref{item:th1cond6,item:th1cond4} hold together. Since the computation in Case 1 consists basically of step $8$ (it was explained above that step $7$ can be integrated into step $6$) we give   priority to Case 1 over the more cumbersome Case 3  when \cref{item:th1cond6,item:th1cond4} hold together.

The number of \ML terms $R$ and their ``sizes'' $L_1,\dots,L_R$ do not have to be known a priori as they are found in Phase 1 and Phase 2, respectively. Namely,  \cref{Alg:auxjbd} in step $5$ estimates  $R$  as the number of blocks of  $\mathbf N$ and
estimates $d_r$ as the number of columns in the $r$th block. If \cref{item:th1cond4}  in \cref{thm: maintheorem} holds, then we set $L_r:=d_r$.
If \cref{item:th1cond5} or \ref{item:th1cond6} in \cref{thm: maintheorem} holds, then we just set $L_r = r_{\mathbf E_r}$.

It is worth noting  that if  \cref{item:th1cond6} in \cref{thm: maintheorem} holds and if the sets $\Omega_m$ in step $11$ are chosen in  a particular way, then the ``sizes'' $r_{\hat{\mathbf E}_r}=L_r$ of the \ML terms of the tensors $\mathcal Q^{(m)}$, constructed in step $14$, can be computed
by solving an overdetermined system of linear equations. 
That is, the values $L_1,\dots,L_R$ can be found without executing step $15$.
Indeed, one can easily verify that \cref{item:th1cond6} in \cref{thm: maintheorem} implies that the equalities
\begin{equation}
\sum\limits_{r\in\Omega_m} r_{\hat{\mathbf E}_r}=r_{\mathbf Q_{(2)}^{(m)}}=r_{\mathbf Q_{(3)}^{(m)}}
\label{eq:overdetsystem}
\end{equation}
hold for any $\Omega_m$, $m=1,\dots,M$. If $M$ has the maximum possible value, i.e., $M=\rubinom{R}{R-r_{\mathbf A}+2}$, 
then the $M$ identities in \cref{eq:overdetsystem} can be rewritten as the system of linear equations 
$\tilde{\mathbf A}\tilde{\mathbf x}=\tilde{\mathbf b}$, where $\tilde{\mathbf A}$ is a binary ($0/1$) $M\times R$ matrix such that none of the rows are proportional and each row of $\tilde{\mathbf A}$ has exactly $R-r_{\mathbf A}+2$ ones. The vectors  $\tilde{\mathbf x}$ and $\tilde{\mathbf b}$ consist of the values $ r_{\hat{\mathbf E}_r}$, $1\leq r\leq R$ and $r_{\mathbf Q_{(2)}^{(m)}}$, $1\leq m\leq M$, respectively.
One can easily verify that $\tilde{\mathbf A}$ has full column rank, i.e., the unique solution of \cref{eq:overdetsystem} yields the values $L_1,\dots, L_R$.

\tcr{\Cref{algorithm:1} should be seen as an algebraic computational proof-of-concept. It opens a new line of research of numerical aspects and strategies; the development of such dedicated numerical strategies is out of the scope of this paper.}

\tcr{In the given form, the computational cost of \cref{algorithm:1}
	is dominated by steps $1$, $2$, and $5$.
	Since each entry of the $\rubinom{I}{2}\rubinom{J}{2}$-by-$\rubinom{K+1}{2}$ matrix ${\mathbf Q}_2(\mathcal T)$ is of the form \cref{eq: tttt}, step $1$ requires at most $7\rubinom{I}{2}\rubinom{J}{2}\rubinom{K+1}{2}$ flops, i.e. $4$ multiplications and $3$ additions per entry (note that no distinction between complex and real data is made). The cost of finding a basis  $\mathbf g_1,\dots,\mathbf g_Q$ via the SVD is  of order $6\rubinom{I}{2}\rubinom{J}{2}(\rubinom{K+1}{2})^2 + 20(\rubinom{K+1}{2})^3$ when the SVD is implemented via the R-SVD method \cite{GolubVanLoan}. 
		 The cost of step $5$ is dominated by step $1$ in
	 \cref{Alg:auxjbd}. This cost is of order $6(K^2Q)^2 (K^2)^2 + 20 (K^2)^3 =(6Q^2+20)K^6$  (cost of the SVD of a $K^2 Q\times K^2$ matrix\footnote{\tcr{Recall that the vectorized matrices $\mathbf U_1,\dots,\mathbf U_R$ in step $1$ of \cref{Alg:auxjbd} can be found from the SVD of the   $K^2Q\times K^2$ matrix  $\mathbf M$  formed by  the rows of $\matrixU_q^T\otimes \mathbf I - (\mathbf I\otimes \matrixU_q)\mathbf P$, $q=1,\dots,Q$, where
	 		$\mathbf P$ denotes the $K^2\times K^2$ permutation matrix that transforms the vectorized form of a $K\times K$ matrix into the vectorized form of its transpose.}}).
	 Thus, the total computational cost of \cref{algorithm:1} is of order $\mathcal O(I^2J^2K^4+K^6)$. Paper \cite[Section S.1]{MikaelCoupledPII} explains an indirect technique to reduce the  total cost of the steps $1$ and $2$ to  $\mathcal O(\operatorname{max}(IJ^2K^2,J^2K^4))$.
	 In this case,  the total computational cost of \cref{algorithm:1} will be of order $\mathcal O(\operatorname{max}(IJ^2K^2+K^6,J^2K^4+K^6))$.
%
 }

\subsubsection{ Approximate  \ML decomposition}\label{subsub252}
Now we discuss noisy variants of the steps in \cref{algorithm:1}. 
\tcr{We consider two scenarios. }	
	
\tcr{I. In the exact case the matrix $\mathbf Q_2(\mathcal T)$ has exactly $Q$ nonzero singular values, the matrices $\matrixU_q$ obtained in step $6$ are at most rank-$\sum d_r$  and the matrix $\mathbf M$ constructed  in \cref{subsub:alg1} has exactly $R$	nonzero singular values.  In the \textit{first scenario} we assume that the perturbation of the tensor is ``small enough'' to recover the correct values of $Q$, $R$ and $d_1,\dots,d_R$ in Phase I. In this case we proceed as follows.}	
 In step $2$ we set $\mathbf g_q$ equal to the
$q$th smallest right singular vector of 
 $\mathbf Q_2(\mathcal T)$. In step $5$ we use the noisy variant of \cref{Alg:auxjbd} (see the end of \cref{subsub:alg1}) which gives us $R$ and the values $d_1,\dots,d_R$.
  In steps $6$ and $7$ we choose $\mathbf a_r$ and $\mathbf B_r$ such that 
 $\operatorname{vec}(\mathbf B_r)\mathbf a_r^T$ is the best rank-$1$ approximation of the matrix $[\operatorname{vec}(\mathbf N_r^T\mathbf H_1^T)\ \dots\ \operatorname{vec}(\mathbf N_r^T\mathbf H_I^T)]$.
  After steps $10$ and $18$ we replace the matrices $\mathbf E_1,\dots,\mathbf E_R$   by their truncated SVDs.
 Assuming the values of $d_1,\dots,d_R$ computed in step $5$ are correct, the truncation ranks
can generically be determined as 
 \begin{equation}
 L_r=d_r+\frac{K-\sum d_r}{R-1},\qquad r=1,\dots,R.
 \label{eq:L_r_via_d_r}
 \end{equation}
 Indeed, if the matrices $\mathbf Z_{1,\mathbf C},\dots,\mathbf Z_{R,\mathbf C}$ have full column rank, then, by \cref{item:th1cond2ABC}, $d_r=K-\sum\limits_{k=1}^RL_k+L_r$. Hence
 $\sum d_r=RK-R\sum\limits_{k=1}^RL_k+\sum\limits_{k=1}^RL_k$, implying that $\sum\limits_{k=1}^RL_k =\frac{RK-\sum d_r}{R-1}$.
 Thus, $L_r = d_r-K+\sum\limits_{k=1}^RL_k=d_r-K+\frac{RK-\sum d_r}{R-1}=d_r+\frac{K-\sum d_r}{R-1}$. In steps $8$, $10$, and $17$ we solve the linear systems in the least squares sense.
   
 An approximate \ML decomposition of the tensor $\mathcal Q^{(m)}$ in step $15$ can be computed in the least squares sense using optimization based techniques. In this case the values $L_1,\dots,L_R$ should be known in advance. They can be estimated as follows. First the values
 $r_{\unf{Q}{2}^{(m)}}$  and $r_{\unf{Q}{3}^{(m)}}$ in \cref{eq:overdetsystem} should be replaced by their  numerical ranks (with respect to some threshold). Then the system of linear equations   \cref{eq:overdetsystem} should be solved in the least squares sense, subject to positive integer constraints on $r_{\hat{\mathbf E}_r}=L_r$.
 
 II. In the \textit{second scenario} we assume that the perturbation of the tensor is not ``small enough'' to guess the values of $Q$, $R$ and $d_1,\dots,d_R$ in Phase 1. We explain how we proceed if (only) the values of $R$ and $\sum L_r$ are known. Since, generically, $d_r=K-\sum\limits_{k=1}^RL_k+L_r$, we obtain
 	that  $\sum d_r=RK-(R-1)\sum L_r$. In step 2, we replace $Q$ by its lower bound 
  $$
   Q_{min}:= \underset{\sum\hat{d}_r=\sum d_r}{\operatorname{argmin}}\left(\rubinom{\hat{d}_1+1}{2}+\dots+\rubinom{\hat{d}_R+1}{2}\right).
    $$
  In the first scenario,  the matrix $\mathbf N$  was estimated as the third factor matrix in  CPD \cref{eq:somePD} and the 
  partition of $\mathbf N$ into blocks $\mathbf N_1,\dots,\mathbf N_R$ (and, in particular, the values of  $d_1,\dots,d_R$)
  was obtained by clustering the  columns of the first factor matrix in the CPD. In the second scenario, we compute   only matrix $\mathbf N$  in step $5$, without estimating the values of $d_1,\dots,d_R$. Since, by \cref{eq:a_rviaN_r}, $\unf{T}{3}\mathbf N_r=\mathbf a_r\otimes (\mathbf E_r\mathbf N_r)$, it follows that $\unf{T}{3}\mathbf N$ coincides up to permutation of columns with the matrix $[\mathbf a_1\otimes (\mathbf E_1\mathbf N_1)\ \dots\ \mathbf a_R\otimes (\mathbf E_R\mathbf N_R)]$. So,  clustering the columns of $\unf{T}{3}\mathbf N$  into $R$  clusters (modulo sign and scaling) we obtain the values $d_1,\dots,d_R$ as the sizes of clusters and the columns of $\mathbf A$ as their centers. The noisy variants of the remaining steps are the same as in
  the first scenario.
     
 \subsubsection{Examples}\label{subsub253}
 \begin{expl}\label{ex:2.4}
 	In this example we illustrate how to apply   \cref{item:th1stat2} of \cref{thm: maintheorem} for the computation of  a  decomposition  that is not unique but does  satisfy \cref{item:th1cond2}. 
  	Let $R\geq 2$. We consider an $R\times (R+2)\times (R+2)$ tensor $\mathcal T$ generated by 
 		\cref{eq:LrLr1mainBC}
 		 in which 
 		 \begin{gather*}
 		 \mathbf A=[{\mathbf a}_1\ \dots\ {\mathbf a}_R],\\ 
 		 \mathbf B=[{\mathbf b}_1\ {\mathbf b}_2\ {\mathbf b}_3\ {\mathbf b}_1\ {\mathbf b}_2\ {\mathbf b}_4\
 		 {\mathbf b}_5\ \dots {\mathbf b}_{3R-2}], \text{ and }
 		 \mathbf C=[{\mathbf c}_1\ {\mathbf c}_2\ {\mathbf c}_3\ {\mathbf c}_1\ {\mathbf c}_2\ {\mathbf c}_4\
 		  \dots\  {\mathbf c}_1\ {\mathbf c}_2\ {\mathbf c}_{R+2}],
 		\end{gather*}
 			where the entries of ${\mathbf a}_1,\dots,{\mathbf a}_R$, $\mathbf b_1,\dots,\mathbf b_{3R-2}$, and $\mathbf c_1,\dots,\mathbf c_{R+2}$ are independently drawn from the standard normal distribution $N(0,1)$. Thus, $\mathcal T$ is a sum of $R$ ML rank-$(1,3,3)$ terms (i.e., $L_1=\dots =L_R=3$): 		 
  	\begin{equation}
 	\begin{split}
 	&\mathcal T=\sum\limits_{r=1}^R\mathbf a_r\circ \mathbf E_r,\ \text{ where }\\
 	&\mathbf E_1 = [{\mathbf b}_1\ {\mathbf b}_2\ {\mathbf b}_3][{\mathbf c}_1\ {\mathbf c}_2\ {\mathbf c}_3]^T, \qquad
 	\mathbf E_2 = [{\mathbf b}_1\ {\mathbf b}_2\ {\mathbf b}_4][{\mathbf c}_1\ {\mathbf c}_2\ {\mathbf c}_4]^T, \text{ and }\\
 	&\mathbf E_r = [{\mathbf b}_{3r-4}\ {\mathbf b}_{3r-3}\ {\mathbf b}_{3r-2}][{\mathbf c}_1\ {\mathbf c}_2\ {\mathbf c}_{r+2}]^T
 	\qquad \text{ for }r\geq 3.
 	\end{split}
 	\label{eq:tensorTR}
 	\end{equation}

{\em Nonuniqueness.} Let us show that the decomposition of $\mathcal T$ into a sum of max ML rank-$(1,3,3)$ terms is not unique. 	
 	Let $\mathcal T_{2}$  equal the sum of the first two ML rank-$(1,L_r,L_r)$ terms:
 	\begin{equation}
 	\mathcal T_{2} = \mathbf a_1\circ(\mathbf b_1\mathbf c_1^T + \mathbf b_2\mathbf c_2^T+\mathbf b_3\mathbf c_3^T) + 
 	\mathbf a_2\circ(\mathbf b_1\mathbf c_1^T + \mathbf b_2\mathbf c_2^T+\mathbf b_4\mathbf c_4^T).
 	\label{eq:firstdec}
 	\end{equation}
 	It can be proved that $\mathcal T_{2}$ admits exactly three
 	decompositions  into a sum of \MLatmost terms, namely \cref{eq:firstdec} itself and the decompositions
 	\begin{equation}
 	\begin{split}
 	\mathcal T_{2}=\ &
 	\mathbf a_1\circ (\mathbf b_3\mathbf c_3^T-\mathbf b_4\mathbf c_4^T) +(\mathbf a_1+\mathbf a_2)\circ(\mathbf b_1\mathbf c_1^T+\mathbf b_2\mathbf c_2^T +\mathbf b_4\mathbf c_4^T)=\\
 	&(\mathbf a_1+\mathbf a_2)\circ(\mathbf b_1\mathbf c_1^T+\mathbf b_2\mathbf c_2^T +\mathbf b_3\mathbf c_3^T)-
 	\mathbf a_2\circ (\mathbf b_3\mathbf c_3^T-\mathbf b_4\mathbf c_4^T). \label{eq:twodecsofT21}
 	\end{split}
 	\end{equation}
 	Since $\mathcal T_{2}$ admits three decompositions  it follows that $\mathcal T$ admits at least three decompositions
 	for $R\geq 2$. In other words,
 	the decomposition of $\mathcal T$ into a sum of \MLatmost terms is not unique.
 	 	
 	{\em Computation for $R\geq 3$.}
 	Now we show that,  by \cref{item:th1stat2} of \cref{thm: maintheorem},  decomposition \cref{eq:tensorTR} can be computed
 	by means of (simultaneous) EVD, at least for $R=3,\dots,20$ (which are the values of $R$ we have tested).
 			First we show   that assumptions \eqref{item:th1cond1},  \eqref{item:th1cond2}, \eqref{item:th1cond3}, and \cref{item:th1cond5}  hold. Assumption \eqref{item:th1cond1} and \cref{item:th1cond5} are trivial. The values of $d_1,\dots,d_R$  in  \cref{item:th1cond2} can be computed by  \cref{item:th1cond2ABC}, which easily gives $d_1=\dots=d_R=1$.
 	It can also be verified that $\mathbf Q_2(\mathcal T)$ is a $\rubinom{R}{2}\rubinom{R+2}{2}\times \rubinom{R+3}{2}$ matrix  and that  (at least for $R=3,\dots,20$) $\dim\nullsp{\mathbf Q_2(\mathcal T)}
 	=R=\sum\rubinom{d_r+1}{2}$, i.e., \cref{item:th1cond3} holds as well. \tcr{(To compute the null space we used the MATLAB built-in  function \texttt{null}.)}
 	
 	Let us now illustrate how \cref{algorithm:1} recovers the  matrices  $\mathbf A$, $\mathbf E_1,\dots,\mathbf E_R$.   As has been mentioned before, since the matrix $\mathbf N$ computed in  step $5$ consists of the blocks $\mathbf N_1\in\fF^{K\times d_1},\dots, \mathbf N_R\in\fF^{K\times d_R}$ which hold, respectively, bases of the subspaces $\nullsp{\mathbf Z_{1}}=\nullsp{\mathbf Z_{1,\mathbf C}},\dots, \nullsp{\mathbf Z_{R}}=\nullsp{\mathbf Z_{R,\mathbf C}}$, it follows that
 	\cref{eq:NrinnullCk} holds. Since $d_1=\dots=d_R=1$, the  S-JBD problem in step $5$ is actually a symmetric joint diagonalization problem. Thus, in step $5$, we obtain an $(R+2)\times R$ matrix $\mathbf N=[\mathbf n_1\ \dots\ \mathbf n_R ]$   and \cref{eq:NrinnullCk} takes the following form : 	
    $$
     \mathbf n_r^T[{\mathbf c}_1\ {\mathbf c}_2\ {\mathbf c}_{3}\ \dots\ 
     {\mathbf c}_1\ {\mathbf c}_2\ {\mathbf c}_{r+1}\ 
     {\mathbf c}_1\ {\mathbf c}_2\ {\mathbf c}_{r+3}\ \dots
     {\mathbf c}_1\ {\mathbf c}_2\ {\mathbf c}_{R+2}]=\mathbf 0,\qquad r=1,\dots,R.
    $$ 
 	Then in step $6$ we compute $\mathbf a_r$, by \cref{eq:22prime}, i.e., as the vector that generates the row space of only right singular vector of  $[\mathbf H_1\mathbf n_r\ \dots\ \mathbf H_I\mathbf n_r]$ :
 	$$
     [\mathbf H_1\mathbf n_r\ \dots\ \mathbf H_I\mathbf n_r]=
 	[\operatorname{vec}(\mathbf n_r^T\mathbf H_1^T)\ \dots\ \operatorname{vec}(\mathbf n_r^T\mathbf H_I^T)]=\operatorname{vec}(\mathbf n_r^T\mathbf E_r^T)\mathbf a_r^T=	
 	(\mathbf E_r \mathbf n_r)\mathbf a_r^T.
 	$$
 	Finally, in step $12$ we reshape the columns of $\unf{T}{1}(\mathbf A^T)^\dagger$ into the matrices $\mathbf E_1$ and $\mathbf E_2$.

 	It is worth noting that  none of the three decompositions of  $\mathcal T_{2}$ can be computed by \cref{thm: maintheorem} while for \tcr{$R=3,\dots,20$}  decomposition \cref{eq:tensorTR}  of $\mathcal T$, involving additional terms, can be computed by \cref{thm: maintheorem}.
 	Let us explain. First,  one can easily verify that the third matrix unfolding of  $\mathcal T_{2}\in\mathbb F^{R\times(R+2)\times (R+2)}$ is rank-$4$, so, as it was explained in \cref{subsubs:convention}, for investigating properties of $\mathcal T_{2}$, we can  w.l.o.g. focus on  $\mathcal T_{2}\in\fF^{R\times (R+2)\times 4}$. It can be verified that $\mathbf Q_2(\mathcal T_2)$ is a $\rubinom{R}{2}\rubinom{R+2}{2}\times 10$ matrix, that $\dim\nullsp{\mathbf Q_2(\mathcal T_2)} = 5$, and that 
 	for all decompositions in \cref{eq:firstdec,eq:twodecsofT21} we have $(d_1,d_2)\in \{(1,1), (2,1), (1,2)\}$.
 	Thus,  $\rubinom{d_1+1}{2}+\rubinom{d_2+1}{2} \leq 4 < 5 = \dim\nullsp{\mathbf Q_2(\mathcal T_2)}$, implying that assumption \cref{item:th1cond3} does not hold. 
 	
	To explain why \cref{item:th1cond3} does hold for  $\mathcal T$ while it  does not hold for $\mathcal T_2$, we refer to equivalence 
 	\cref{eq:iffrank1}. 
 	From \cref{eq:lincombfT,eq:tensorTR} it follows that
 	\begin{multline}
 	f_1\mathbf T_1+\dots+f_{R+2}\mathbf T_{R+2}=\left((\mathbf a_1+\mathbf a_2)\mathbf b_1^T+\sum\limits_{r=3}^R\mathbf a_r\mathbf b_{3r-4}^T\right)\mathbf f^T\mathbf c_1 +\\
 	\left((\mathbf a_1+\mathbf a_2)\mathbf b_2^T+\sum\limits_{r=3}^R\mathbf a_r\mathbf b_{3r-3}^T\right)\mathbf f^T\mathbf c_2+  
 	(\mathbf a_1\mathbf b_3^T)\mathbf f^T\mathbf c_3+ 	(\mathbf a_2\mathbf b_4^T)\mathbf f^T\mathbf c_4+\\
 	\sum\limits_{r=3}^R(\mathbf a_r\mathbf b_{3r-2}^T)\mathbf f^T\mathbf c_{r+2}.
 	\label{eq:new12}
 	\end{multline}
 	 Above, we have numerically verified that $\dim\nullsp{\mathbf Q_2(\mathcal T)}
 	 	=R=\sum\rubinom{d_r+1}{2}$, which guarantees that \cref{eq:iffrank1} holds for $\mathcal T$, i.e., 
 	 	$f_1\mathbf T_1+\dots+f_{R+2}\mathbf T_{R+2}$ is rank-$1$ if and only if $\mathbf f$ belongs to the null spaces of all
 	 	matrices $[\mathbf c_1\ \mathbf c_2\ \mathbf c_3]^T,\dots,[\mathbf c_1\ \mathbf c_2\ \mathbf c_{R+3}]^T$ but one.
 	  On the other hand, in the case of $\mathcal T_2$, one can easily find a counterexample to the implication ``$\Rightarrow$'' in \cref{eq:iffrank1}. Indeed, for $\mathcal T_2$ the linear combination in the LHS of \cref{eq:new12} of the frontal slices of $\mathcal T_2$ can be rewritten as the RHS
 	  without the terms under the summation signs. Then the implication ``$\Rightarrow$'' in \cref{eq:iffrank1} does not hold for a vector $\mathbf f$ such that $\mathbf c_3^T\mathbf f=\dots=\mathbf c_{R+2}^T\mathbf f=0$ but
 	  $|\mathbf c_1^T\mathbf f|+|\mathbf c_2^T\mathbf f|\ne 0$.
  	   \end{expl}
 \begin{expl}\label{example:3J15}
	We consider a $3\times J\times 15$ tensor generated by \cref{eq:LrLr1mainBC} in which  the entries of $\mathbf A$, $\mathbf B$, and $\mathbf C$ are independently drawn from the standard normal distribution $N(0,1)$ and $L_1=L_2=L_3=2$, $L_4=L_5=3$, and $L_6=4$. Thus, $\mathcal T$ is a sum of $R=6$ terms. For $J\geq 9$,  one can easily check that 
	$d_r=L_r-1$ and that  \cref{item:th1cond1,item:th1cond3a} in \cref{thm: maintheorem} hold.
	We  illustrate   \cref{item:th1stat4,item:th1firstfmhard} of \cref{thm: maintheorem}  by considering $J$ in the sets $\{9,10,11,12,13\}$ and $\{14,15\}$, respectively.
	\begin{enumerate} 
		\item Let $J\in\{9,\dots,12,13\}$.
	 Computations indicate that for $J=9$ the  null space of the $108\times 120$ matrix $\mathbf Q_2(\mathcal T)$ has dimension $15$.
	 \tcr{(To compute the null space we used the MATLAB built-in  function \texttt{null}.)}
	 Since 	 $\sum C^2_{d_r+1}=C^2_2+C^2_2+C^2_2+C^2_3+C^2_3+C^2_4 = 15$, it follows that \cref{item:th1cond3} holds.
 	 It is clear that   \cref{item:th1cond3} will also hold for $J>9$.
	 Since
	 \begin{equation*}
	 \rubinom{K+1}{2} - Q=105 >101= - \tilde{L}_1\tilde{L}_2 + \sum\limits_{1\leq r_1<r_2\leq R} L_{r_1}L_{r_2} ,
	 \end{equation*}
	  it follows that \cref{eq:ineqforuni1fm} also holds. 
	 Hence, by   \cref{item:th1firstfmhard} of \cref{thm: maintheorem}, the first factor  matrix of $\mathcal T$ is unique and can be computed in Phase I of \cref{algorithm:1}. 
	 \item Let $J\in\{14,15\}$. Then \cref{item:th1cond6} in \cref{thm: maintheorem} holds.
	 Hence, by   \cref{item:th1stat4} of \cref{thm: maintheorem}, the overall decomposition is unique and can be computed by \cref{algorithm:1}. 	
	In step $11$ we can, for instance, set $M=2$ and choose $\Omega_1=\{1,2,3,4,5\}$ and $\Omega_2=\{1,2,3,4,6\}$.
	In this case the loop in  steps $12-16$ is executed  twice which yields   matrices  $\hat{\mathbf E}_1,\dots,\hat{\mathbf E}_4,\hat{\mathbf E}_5$ and
	  matrices  $\alpha_1\hat{\mathbf E}_1,\dots,\alpha_4\hat{\mathbf E}_4,\hat{\mathbf E}_6$, respectively, where $\alpha_1,\dots,\alpha_4$
	are nonzero values. The computed matrices $\hat{\mathbf E}_1,\dots, \hat{\mathbf E}_6$ necessarily coincide with the matrices
	$\mathbf E_1,\dots, \mathbf E_6$ in decomposition \cref{eq:LrLr1} up to permutation of indices and scaling factors.
	Note that neither $R$ nor $L_1,\dots,L_R$ should be known a priori.   
	 \end{enumerate}
\end{expl}

In the  following two examples we assume that the decomposition in
\cref{eq:LrLr1} is perturbed with a random additive term. The examples demonstrate the 
computation of the approximate  \ML decomposition   \cref{eq:LrLr1}.


\begin{expl}\label{example2.8} 
%
In this example we illustrate  the computation of $L_1,\dots,L_R$ and the computation of the approximate \ML decomposition  assuming that
the exact decomposition satisfies  \cref{item:th1cond5} in \cref{thm: maintheorem} (i.e., Case 2 in \cref{algorithm:1}).

First we consider the case where the decomposition is exact.
We consider a $3\times 8\times 8$ tensor generated by \cref{eq:LrLr1mainBC} in which  the entries of $\mathbf A$, $\mathbf B$, and $\mathbf C$ are independently drawn from the standard normal distribution $N(0,1)$ and $L_1=2$, $L_2=3$, $L_3=4$. Thus, $\mathcal T$ is a sum of $R=3$ terms.
It can be numerically verified that $d_1=1$, $d_2=2$, $d_3=3$ and that the  null space of the $84\times 36$ matrix $\mathbf Q_2(\mathcal T)$ has dimension $10=\rubinom{d_1+1}{2}+\rubinom{d_2+1}{2}+\rubinom{d_3+1}{2}$. Hence, by   \cref{item:th1stat4} of \cref{thm: maintheorem}, the overall decomposition is unique and can be computed by \cref{algorithm:1} (Case 2). 
Note that if the third dimension is decreased by $1$,
	 then  \cref{item:th1cond3a} in \cref{thm: maintheorem} does not hold. It can also be shown that if the first dimension is decreased by $1$, then
	 assumption 	 \cref{item:th1cond3}	in \cref{thm: maintheorem} does not hold.

Now we consider a noisy variant.
 Since the problem is already challenging we  exclude to some extent random tensors that may pose additional numerical difficulties\footnote{
Note that,  if the first or third matrix unfolding has a large condition number,
	we are approaching, as explained above, a situation in which the conditions in  \cref{thm: maintheorem} 
	and hence the working assumptions in \cref{algorithm:1} are not satisfied.}
  by limiting the condition numbers of the matrix unfoldings $\unf{T}{1}$ \tcr{and} $\unf{T}{3}$. More concretely,
 we select $100$ random tensors with \tcr{$\max(cond(\unf{T}{1}), cond(\unf{T}{3}))\leq 10$, where $cond(\cdot)$ denotes the condition number of a matrix, i.e., the ratio of the largest and smallest singular value}. \tcr{We estimate the  \tcr{ML rank values} and the factor matrices  from $T+c\mathcal N$, where $\mathcal N$ is a perturbation tensor and $c$ controls the signal-to-noise level. The entries of $\mathcal N$ 
 	are independently drawn from the standard normal distribution $N(0,1)$ and
the following Signal-to-Noise Ratio (SNR) measure is used: $SNR\ [dB] = 10\log (\|\mathcal T\|^2_F/c^2\|\mathcal N\|_F^2)$, where $\|\cdot\|_F$ denotes the Frobenius norm of a tensor.}
To compute the decomposition \tcr{of $\mathcal T+c\mathcal N$}  we use the  noisy version of \cref{algorithm:1} \tcr{explained in \cref{subsub252} (the second scenario). We assume that  $R=3$ and $\sum L_r=9$ are known. 
Since we are in a generic setting, $\sum d_r=RK-(R-1)\sum L_r=6$. Assuming that $d_1\leq d_2\leq d_3$,  this implies that the triplet  $(d_1,d_2,d_3)$
coincides with one of the triplets	$(1,1,4)$, $(1,2,3)$, $(2,2,2)$.
The respective values for  $\rubinom{d_1+1}{2}+\rubinom{d_2+1}{2}+\rubinom{d_3+1}{2}$ are $8$, $10$, and $9$. Consequently, in our computations we replace $Q$ by
$Q_{min}=\min(8,10,9)=8$. 
	}

The  \tcr{ matrix $\mathbf A$ and the values of} $d_1$, $d_2$, and $d_3$ are estimated  as \tcr{in \cref{subsub252} (the second scenario). 
	 The 	matrix $\mathbf N$ in  the simultaneous EVD in step $2$ of \cref{Alg:auxjbd} was found in two ways: i) from  the EVD of a single generic linear combination of $\mathbf U_1,\dots,\mathbf U_R$ and 	ii) by computing  CPD \cref{eq:somePD}.
	Since we are in a generic setting, the values of $L_1$, $L_2$, and $L_3$ can be found from the values of $d_1$, $d_2$, and $d_3$  by \cref{eq:L_r_via_d_r}.}
\tcr{This means that if $L_1\leq L_2\leq L_3$, then the triplet  $(L_1,L_2,L_3)$ necessarily coincides with one of the triplets $(2,2,5)$, $(2,3,4)$, $(3,3,3)$. \Cref{table:example2.8} shows the frequencies with which each triplet occurs as a function of the SNR.}
 To measure the performance we compute the relative error on the estimates of the first factor matrix $\mathbf A$ and   on the estimates of the matrix formed by the vectorised multilinear terms, $[\mathbf a_1\otimes\operatorname{vec}(\mathbf E_1)\ \dots\ \mathbf a_R\otimes\operatorname{vec}(\mathbf E_R)]$. (We compensate for scaling and permutation ambiguities.)
The results are shown in \cref{fig:testfig1}. Note that the accuracy of the estimates is of about the same order as the accuracy of the given tensors.

\begin{table}[htbp]\label{table:example2.8}
		\caption{Frequencies with which the ML rank values have been estimated correctly (second row) or incorrectly  (first and third row) (see \cref{example2.8})}
	\centering
	\begin{tabular}{|c|cccccccc|}
		\hline
\multirow{2}{*}{$L_1$, $L_2$, $L_3$}	&\multicolumn{8}{c|}{SNR (dB)}	\\
	\cline{2-9}
         &      15&        20&        25&        30&        35&        40&        45&        50\\
         \hline
         \hline
\multicolumn{1}{|c|}{2, 2, 5}&      21&        12&         8&         -&         -&         -&         -&         -\\  
\multicolumn{1}{|c|}{2, 3, 4}&      63&        79&        89&        96&        100&        99&        100&        100 \\ 
\multicolumn{1}{|c|}{3, 3, 3}&      16&        9&         3&         4&         -&         1&         -&         -\\
\hline 	
\end{tabular}
\end{table}
\end{expl}
\begin{figure}[htbp]
  \centering
  \label{fig:1}
   \includegraphics[scale=0.35]{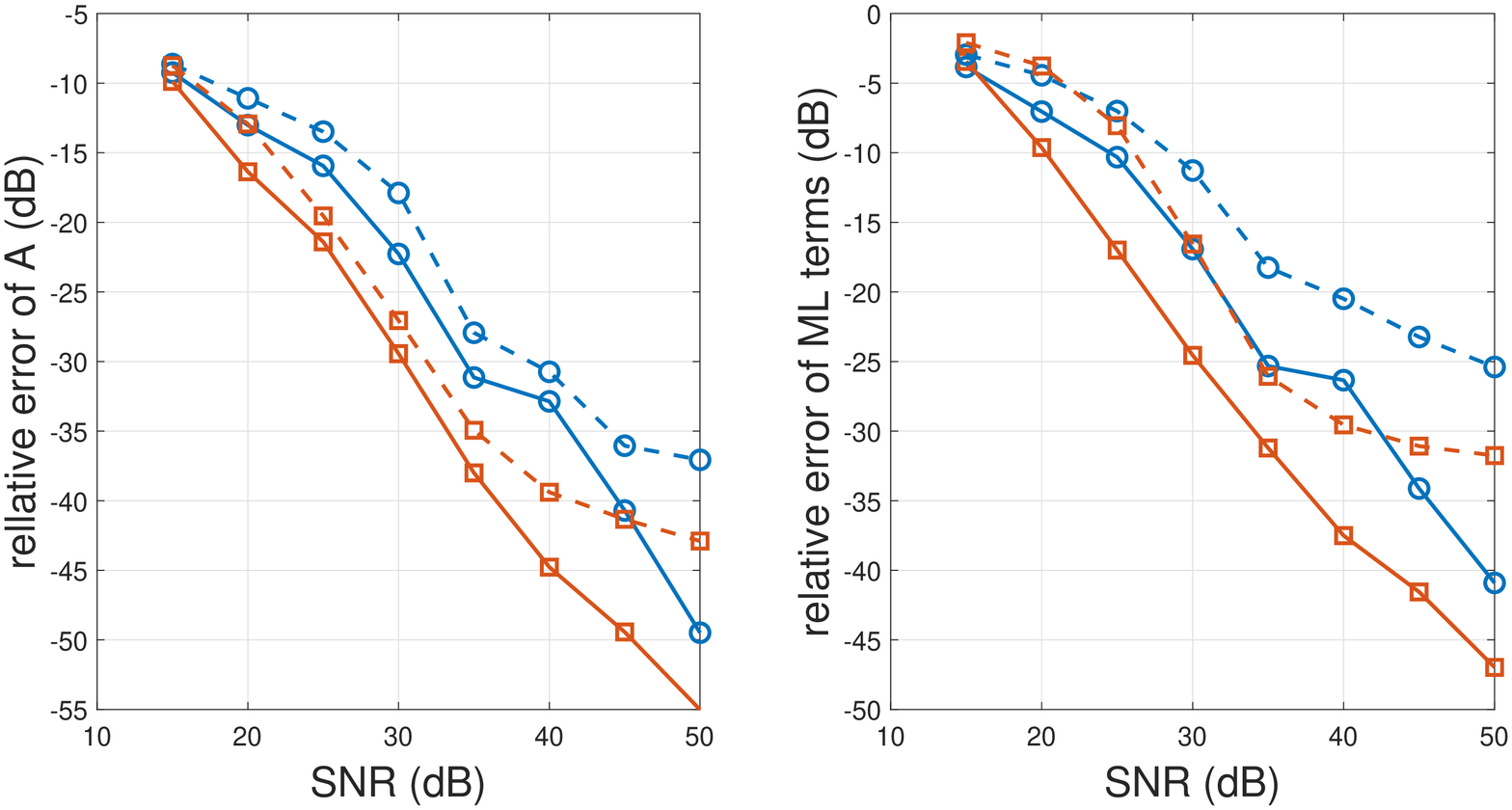}
  \caption{ Mean (\textcolor{blue}{$\Circle$}) and median (\textcolor{red}{$\Square$}) curves for the relative errors on the first factor matrix $\mathbf A$ (left plot) and the matrix formed by the vectorized ML terms $[\mathbf a_1\otimes\operatorname{vec}(\mathbf E_1)\ \dots\ \mathbf a_R\otimes\operatorname{vec}(\mathbf E_R)]$ (right plot). \tcr{The dashed and solid line correspond to the version of \cref{Alg:auxjbd} where the solution $\mathbf N$ of the simultaneous EVD in step $2$  is obtained from  the EVD of a single generic linear combination and from the CPD \cref{eq:somePD}, respectively} (see \cref{example2.8}).}
  \label{fig:testfig1}
\end{figure}
 \begin{expl} \label{example2.9}
    In this example we illustrate  the computation of $L_1,\dots,L_R$ and the computation of the approximate \ML decomposition  assuming that
 the exact decomposition satisfies  \cref{item:th1cond4} in \cref{thm: maintheorem} (i.e., Case 1 in \cref{algorithm:1}).
 
 We consider a $3\times 9\times 10$ tensor generated by \cref{eq:LrLr1mainBC} in which  the entries of $\mathbf A$, $\mathbf B$, and $\mathbf C$ are independently drawn from the standard normal distribution $N(0,1)$ and $L_1=1$, $L_2=2$, $L_3=3$, and $L_4=4$. Thus, $\mathcal T$ is a sum of $R=4$ terms.
We find numerically that $d_1=1$, $d_2=2$, $d_3=3$, $d_4=4$ and that the  null space of the $216\times 55$ matrix $\mathbf Q_2(\mathcal T)$ has dimension $20=\rubinom{d_1+1}{2}+\rubinom{d_2+1}{2}+\rubinom{d_3+1}{2}+\rubinom{d_4+1}{2}$. Hence, by   \cref{item:th1stat4} of \cref{thm: maintheorem}, the overall decomposition is unique and can be computed by \cref{algorithm:1} (Case 1). 
 It can be shown that  in this example we are again in a bordering case with respect to working assumptions in \cref{algorithm:1}, i.e., 
 if \tcr{the first or third} dimension is decreased by  $1$, then the decomposition cannot be computed by \cref{algorithm:1}.
 \tcr{As in \cref{example2.8},  we use the  noisy version of \cref{algorithm:1}  explained in \cref{subsub252} (the second scenario). We assume that  $R=4$ and $\sum L_r=10$ are known. Since we are in a generic setting, $\sum d_r=RK-(R-1)\sum L_r=10$. One can easily verify that there exist exactly  $9$  tuples $(d_1,d_2,d_3,d_4)$ such that $d_1\leq d_2\leq d_3\leq d_4$ and $\sum d_r=10$. Since $K=\sum L_r$ we have that $L_r=d_r$.
 	   The possible tuples $(L_1,L_2,L_3,L_4)$ ($=(d_1,d_2,d_3,d_4)$) are shown in the first column of \cref{table:example2.9}. 
 	   The respective $9$ values for  $\rubinom{d_1+1}{2}+\rubinom{d_2+1}{2}+\rubinom{d_3+1}{2}+\rubinom{d_4+1}{2}$ are $31$, $26$, $23$, $22$, $22$, $20$, $19$, $19$ and $18$. Consequently, in our computations we replace $Q$ by   $Q_{min}=18$. The
     matrix $\mathbf N$ was found in two ways: i) from  the EVD of a single generic linear combination of $\mathbf U_1,\dots,\mathbf U_R$ and 
    ii) by computing  CPD \cref{eq:somePD}. In the latter case the last frontal slice of $\mathcal U$ in \cref{eq:somePD}, i.e., the matrix $\mathbf U_R$, was replaced by $\omega\mathbf U_R$ with  $\omega=2$ (see explanation at the end of \cref{subsub:alg1}).} The results are shown in \cref{table:example2.9,fig:testfig2}. Again, despite the difficulty of the problem the accuracy of the  estimates is of about the same order as the accuracy of the given tensors. 
 
 \begin{table}[htbp]\label{table:example2.9}
 	\caption{Frequencies with which the ML rank values have been estimated correctly (sixth row) or incorrectly  (remaining rows) (see \cref{example2.9})}
 	\centering
 	\begin{tabular}{|c|cccccccc|}
 		\hline
 		\multirow{2}{*}{$L_1$, $L_2$, $L_3$, $L_4$}	&\multicolumn{8}{c|}{SNR (dB)}	\\
 		\cline{2-9}
 		&      15&        20&        25&        30&        35&        40&        45&        50\\
 		\hline
 		\hline
 	    \multicolumn{1}{|c|}{1, 1, 1, 7}&      1&        -&        -&        -&        -&        -&        -&        - \\ 
 	    \multicolumn{1}{|c|}{1, 1, 2, 6}&      5&        1&        -&        -&        -&        -&        -&        - \\ 
 	    \multicolumn{1}{|c|}{1, 1, 3, 5}&      8&        2&        2&        -&        -&        -&        -&        - \\ 
 	    \multicolumn{1}{|c|}{1, 1, 4, 4}&      4&        4&        1&        3&        -&        1&        -&        - \\ 
 	    \multicolumn{1}{|c|}{1, 2, 2, 5}&      13&        10&        5&        -&        -&        -&        -&        - \\ 
 		\multicolumn{1}{|c|}{1, 2, 3, 4}&      54&        73&        88&        96&        100&        99&        100&        100 \\ 
 		\multicolumn{1}{|c|}{1, 3, 3, 3}&      6&        3&        2&        -&        -&        -&        -&        - \\ 
 		\multicolumn{1}{|c|}{2, 2, 2, 4}&      3&        2&        2&        -&        -&        -&        -&        - \\ 
 		\multicolumn{1}{|c|}{2, 2, 3, 3}&      6&        5&        -&        1&       -&        -&        -&        - \\ 
 		\hline 	
 	\end{tabular}
 \end{table}
\end{expl} 
  \begin{figure}[htbp]
   \centering
   \label{fig:2}
     \includegraphics[scale=0.35]{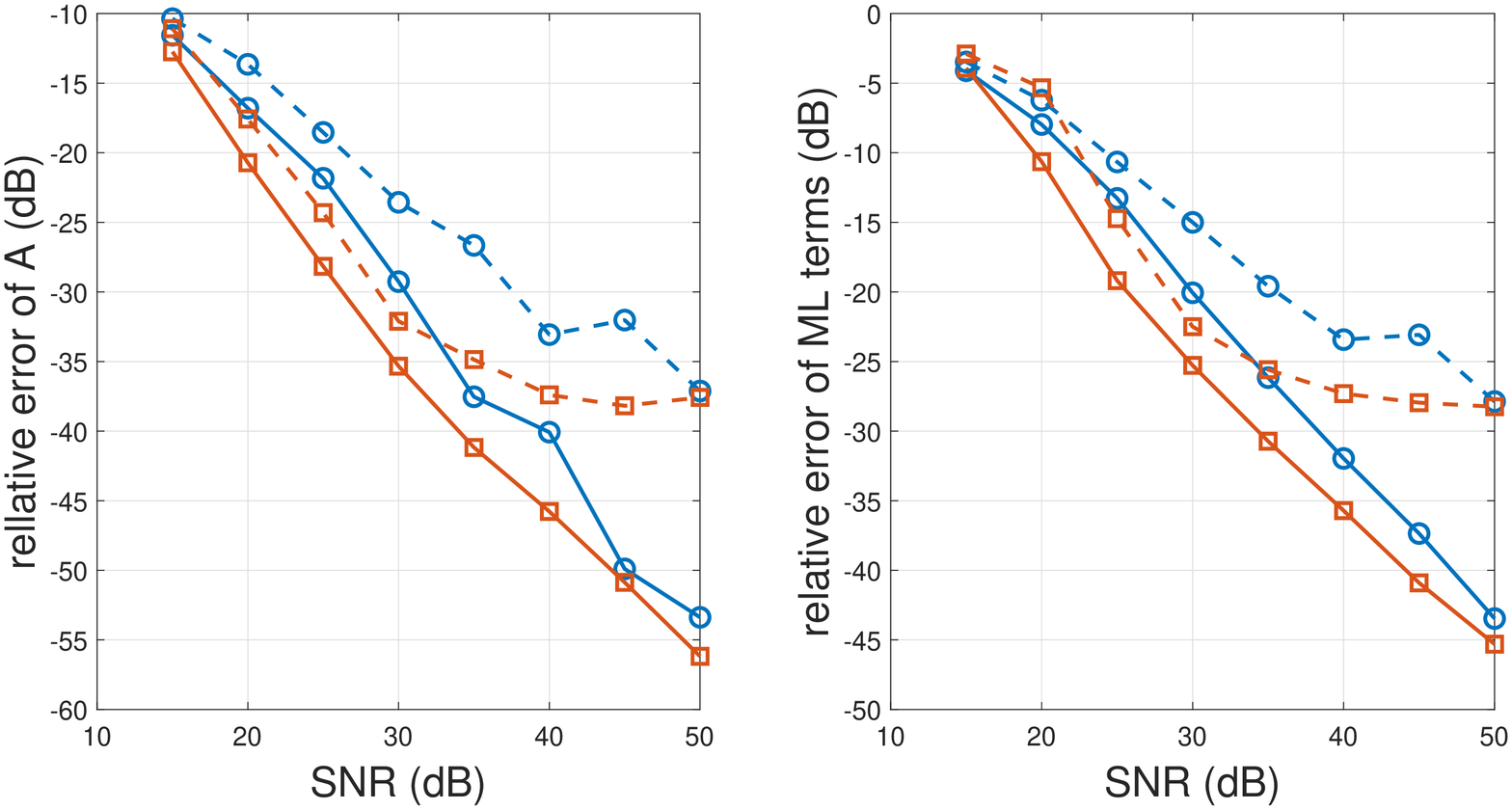}
   \caption{Mean (\textcolor{blue}{$\Circle$}) and median (\textcolor{red}{$\Square$}) curves for the relative errors on the first factor matrix  $\mathbf A$ (left plot) and the matrix formed by the vectorized ML terms $[\mathbf a_1\otimes\operatorname{vec}(\mathbf E_1)\ \dots\ \mathbf a_R\otimes\operatorname{vec}(\mathbf E_R)]$ (right plot). \tcr{The dashed and solid line correspond to the version of \cref{Alg:auxjbd} where the solution $\mathbf N$ of the simultaneous EVD in step $2$  is obtained from  the EVD of a single generic linear combination and from the CPD \cref{eq:somePD}, respectively} (see \cref{example2.9}).}
        \label{fig:testfig2}
 \end{figure}
\subsection{Results   for  \tcr{generic  decompositions}}\label{sec:genuniq}
The main results  of this subsection are   summarized in \cref{tab:KoMa14}(b). The results in \cref{subsec:genericcounterparts} are  generic counterparts of \cref{corollary:verynew} and \cref{thm: maintheorem} and therefore are
sufficient for generic uniqueness and guarantee that  a generic decomposition can be computed by means of EVD. In \cref{subsec:necccondgen} we discuss a necessary condition for generic uniqueness that  is more restrictive than
 generic versions of the conditions in  \cref{thm:necessity} at least for $\mathbb F=\mathbb C$. In \cref{subsub:genstrassen} we present two results on generic uniqueness of decompositions with a factor matrix that has full column rank. These results are   generalizations of  
 Strassen's  result on generic uniqueness of the CPD. The  conditions  are very mild are and easy to verify but they do not imply an algorithm. 

\subsubsection{Generic counterparts of the results from \cref{subsub251}}\label{subsec:genericcounterparts}
The first two results of this subsection are the generic counterparts of \cref{corollary:verynew} and \cref{thm: maintheorem} (or \cref{thm: maintheoremABC}). 
To simplify the presentation and w.l.o.g.  we assume 
that $L_1\leq \dots\leq L_R$.
  It is clear that the assumptions  $J\geq L_{\min(I,R)-1}+\dots+L_R$ \tcr{and $I\geq 2$ in \cref{thm:maingenshort} are, respectively, the generic  version of the assumption  $k_{\mathbf B}'\geq R-r_{\mathbf A}+2$ and $k_{\mathbf A}\geq 2$} in \cref{eq:corollary261}.
 The generic version of the condition $k_{\mathbf C}'\geq R-r_{\mathbf A}+2$ in
\cref{eq:corollary262}   coincides with $K\geq L_{\min(I,R)-1}+\dots+L_R$, which always holds because of the assumption $K\geq L_2+\dots+L_R+1$ in 
\cref{eq:corollary261gen}. Hence, in the generic setting, the conditions in \cref{eq:corollary262} can be dropped.
Thus, we have  the following result.
\begin{theorem}\label{thm:maingenshort}
	Let $L_1\leq \dots\leq L_R\leq \min(J,K)$ and let	 $\mathcal T\in\fF^{I\times J\times K}$ admit decomposition \cref{eq:LrLr1mainBC}, where
the entries of the matrices $\mathbf A\in\fF^{I\times R}$, $\mathbf B\in\fF^{J\times \sum L_r}$, and $\mathbf C\in\fF^{K\times \sum L_r}$ are randomly sampled from an absolutely continuous distribution. Assume that  
\begin{gather}
 K\geq L_2+\dots+L_R+1,\label{eq:corollary261gen}\\
J\geq L_{\min(I,R)-1}+\dots+L_R,
\ \text{ and }
I  \geq 2.\label{eq:genc}
\end{gather}
Then the decomposition of $\mathcal T$ into a sum of \MLatmost terms
is unique and can be computed by means of (simultaneous) EVD.	
\end{theorem}
 In the following theorem, assumptions \cref{item:th1gencond1}, \cref{item:th1gencond2}, \cref{eqLgenericdimQ_2}, 
conditions \cref{eq:ineqforuni1fmgen,eq:genb,eq:gend} and \cref{item:genstate1,item:genstate2ad,item:th1firstfmhardgen,item:genstate2}
correspond, respectively,  to assumptions \cref{item:th1cond1}, \cref{item:th1cond2}, \cref{item:th1cond3},
conditions  \ref{eq:ineqforuni1fm}, \ref{item:th1cond5},  \ref{item:th1cond4} and
statements \ref{item:th1stat1}, \ref{item:th1newstat}, \ref{item:th1firstfmhard},  \ref{item:th1stat4} 
 in \cref{thm: maintheorem}. 
The convention $L_1\leq \dots\leq L_R$ implies that
$d_1 := K-\sum\limits_{k=1}^RL_k+L_1\leq\dots\leq d_R:=K-\sum\limits_{k=1}^RL_k+L_R$. 
Thus, the  $R$ constraints in \cref{item:th1cond2} are   replaced by the single constraint $d_1\geq 1$ in 
\cref{item:th1gencond2}, which  moreover coincides with \cref{item:th1cond3a} in \cref{thm: maintheorem}. Hence,
in a generic setting, \cref{item:th1stat2} in \cref{thm: maintheorem} becomes the part of \cref{item:th1stat4} that relies on \cref{item:th1cond3a}.
That is why the following result contains fewer statements than \cref{thm: maintheorem}.
  \begin{theorem}\label{thm:maingen}
	Let $L_1\leq \dots\leq L_R\leq \min(J,K)$ and let	 $\mathcal T\in\fF^{I\times J\times K}$ admit decomposition \cref{eq:LrLr1mainBC}, where
	the entries of the matrices $\mathbf A\in\fF^{I\times R}$, $\mathbf B\in\fF^{J\times \sum L_r}$, and $\mathbf C\in\fF^{K\times \sum L_r}$ are randomly sampled from an absolutely continuous distribution. Assume that\footnote{\tcr{The inequality $\sum L_r\geq K$ in \cref{item:th1gencond1} is added for notational purposes; it simplifies the formulation of \cref{item:th1gencond2,eqLgenericdimQ_2}. By  \cref{item:thm:sonvention:statement2} of \cref{thm:convention},  uniqueness and computation of a generic decomposition of an $I\times J\times K$ tensor with $K\geq\sum L_r$ follow 	from uniqueness and computation of a generic decomposition of an $I\times J\times \sum L_r$ tensor. In other words,  the assumption $\sum L_r \geq K$ in \eqref{item:th1gencond1} is not a constraint:  if $K\geq\sum L_r$, then the assumptions and conditions in \cref{thm:maingen} should be verified for $K=\sum L_r$.}}
	\begin{gather}
	IJ\geq \tcr{\sum\limits_{r=1}^R L_r\geq}  K,\label{item:th1gencond1}\\  
	d_1:=K-\sum\limits_{r=1}^R L_r+L_1\geq 1, \label{item:th1gencond2}
	\end{gather}
	 and that
%
	there exist vectors $\tilde{\mathbf a}_r\in\fF^{I}$, and matrices $\tilde{\mathbf B}_r\in\fF^{J\times L_r}$, $\tilde{\mathbf C}_r\in\fF^{K\times L_r}$  such that
	\begin{equation}
	\dim\nullsp{\mathbf Q_2(\tilde{\mathcal T})}=\sum\limits_{r=1}^R \rubinom{d_r+1}{2},\label{eqLgenericdimQ_2}
	\end{equation}
	where $\tilde{\mathcal T}=\sum\tilde{\mathbf a}_r\circ(\tilde{\mathbf B}_r\tilde{\mathbf C}_r^T)$ and $d_r := K-\sum\limits_{k=1}^RL_k+L_r$, $r=1,\dots,R$.
	The following statements hold \tcr{generically}.
	\begin{statements}
			\item \label{item:genstate1} 
		The matrix $\mathbf A$  in \cref{eq:LrLr1mainBC} can be computed by means of (simultaneous) EVD. 
		\item \label{item:genstate2ad}	
        Any  decomposition of $\mathcal T$ into a sum of \MLatmost terms  has
        	$R$ nonzero terms and  its first factor matrix  is equal to $\mathbf A\mathbf P$, where every column of  $\mathbf P\in\fF^{R\times R}$ contains precisely a single $1$ with zeros everywhere else.
         \item \label{item:th1firstfmhardgen} If 
        	\begin{equation}
              K\geq -\frac{1}{2} -\sqrt{\frac{1}{4}+\frac{2L_1L_2}{R-1}}+\sum\limits_{r=1}^R L_r, \label{eq:ineqforuni1fmgen}
        	\end{equation}
         then the first factor matrix of  the decomposition of  $\mathcal T$ into a sum of \MLatmost  terms is  unique. 
         \item \label{item:genstate2}
        The decomposition of $\mathcal T$ into a sum of \MLatmost terms
          is unique and can be computed by means of (simultaneous) EVD if
		any of the following \tcr{two} conditions holds:
		 \begin{align}
		 I&\geq R\label{eq:genb}, \\
		 K&=\sum\limits_{r=1}^R L_r.\label{eq:gend}
		 		 \end{align}
	\end{statements}
	\end{theorem}
\begin{proof}
 	The proof is   given in \cref{sec:simplethms}.
 \end{proof}

\tcr{\tcr{To verify the uniqueness and EVD-based computability of a generic decomposition in the case $I\geq R$,} one can use   \cref{thm:maingenshort} (i.e.,   verify the assumptions $K-\sum L_r+L_1\geq 1$  and $J\geq L_{\min(I,R)-1}+\dots+L_R=L_{R-1}+L_R$) or \cref{thm:maingen} (i.e.,  verify the assumptions $IJ\geq \sum L_r$, $K-\sum L_r+L_1\geq 1$, 
	and \eqref{eqLgenericdimQ_2}). Let us briefly comment on these two options. From \cref{sta:lemmaApp4} of \cref{lemma: Q2viaABC} below, it follows that for $I\geq R$,
the assumptions in \cref{thm:maingen} are  at least as relaxed as the assumptions in  \cref{thm:maingenshort}.
On one hand, the assumption $J\geq L_{R-1}+L_R$ in \cref{thm:maingenshort} is easy to verify; on the other hand,  it can be more restrictive than assumption \eqref{eqLgenericdimQ_2} in  \cref{thm:maingen}.  For instance, it can be verified that  uniqueness and EVD-based computability of a generic decomposition of a $3\times 6\times 8$ tensor into a sum of  \MLatmost terms with $L_1=L_2=3$ and $L_3=4$ follow from  \cref{thm:maingen} but do not follow from \cref{thm:maingenshort} (indeed, $6=J\geq L_{R-1}+L_R=3+4$ does not hold).}

\tcr{We now  explain  how to  verify assumption \eqref{eqLgenericdimQ_2}.}

\tcr{In the proof of \cref{thm:maingen} we explain that if assumption
	\cref{eqLgenericdimQ_2} holds for  one triplet of matrices $\tilde{\mathbf A}$, $\tilde{\mathbf B}$, and $\tilde{\mathbf C}$, then 
	\cref{eqLgenericdimQ_2} holds also for a generic triplet. The other way around, it suffices to verify \cref{eqLgenericdimQ_2} for a generic triplet, where some care needs to be taken that the algebraic situation is not obfuscated by numerical effects. Hence one possibility to 
    verify  \cref{eqLgenericdimQ_2} is to randomly select matrices $\tilde{\mathbf A}$, $\tilde{\mathbf B}$, and $\tilde{\mathbf C}$, construct
   $\mathbf Q_2(\tilde{\mathcal T})$ and estimate its rank numerically. Because of the rounding errors such  computations cannot be considered as a formal proof of \cref{eqLgenericdimQ_2}, unless it is clear that the rounding did not affect the rank of  $\mathbf Q_2(\tilde{\mathcal T})$. 
   To have a formal proof of \cref{eqLgenericdimQ_2} one can chose  matrices $\tilde{\mathbf A}$, $\tilde{\mathbf B}$, and $\tilde{\mathbf C}$ such that the entries of  $\mathbf Q_2(\tilde{\mathcal T})$  are integers and, possibly, such that  $\mathbf Q_2(\tilde{\mathcal T})$ is sparse, so the
   identity in \cref{eqLgenericdimQ_2} becomes easy to prove. Both possibilities 
   are illustrated in the upcoming \cref{example335}. Another possibility  to have a formal proof of \cref{eqLgenericdimQ_2} is to perform all computations over a finite field. This approach is  explained in \cref{Appendix:FF}.
Note that both approaches can be quite expensive and may require a third-party implementation.
 }
\begin{expl}\label{example335}
	Let $\mathcal T$ be $3\times 3\times 5$ tensor 
	 generated by \cref{eq:LrLr1mainBC} in which  the entries of $\mathbf A$, $\mathbf B$, and $\mathbf C$ are independently drawn from the standard normal distribution $N(0,1)$ and $L_1=L_2=L_3=1$, $L_4=2$.
	To prove  that  the decomposition of $\mathcal T$ into a sum of \MLatmost terms
	is unique and can be computed by means of (simultaneous) EVD we verify assumptions
	\cref{item:th1gencond1}, \cref{item:th1gencond2}, \cref{eqLgenericdimQ_2} and condition \cref{eq:gend} in \cref{thm:maingen}.
	Assumptions 	\cref{item:th1gencond1}, \cref{item:th1gencond2} and condition \cref{eq:gend} obviously hold.
		Let us now  illustrate two possibilities to verify \cref{eqLgenericdimQ_2}.
	
    \textit{I. The matrices $\tilde{\mathbf A}$, $\tilde{\mathbf B}$, and $\tilde{\mathbf C}$ are generic.}
	\tcr{For $5$   randomly generated triplets $(\tilde{\mathbf A}, \tilde{\mathbf B},\tilde{\mathbf C})$ in \cref{example335}, we have obtained
	that the condition number of the $9\times 15$ matrix $\mathbf Q_2(\tilde{\mathcal T})$ took values $223.12 $,  $75.46 $,	   $681.37 $, $2832.9$, and $147.65 $ 	which clearly suggests that $\mathbf Q_2(\tilde{\mathcal T})$ is a full-rank  matrix (i.e., $r_{\mathbf Q_2(\tilde{\mathcal T})}=9$).
	 Hence, by the rank-nullity theorem, $\dim\nullsp{\mathbf Q_2(\tilde{\mathcal T})}=15-9=6$. Since  \cref{eq:gend} holds, it follows that
	$d_r=K-\sum\limits_{k=1}^RL_k+L_r=L_r$, implying that   $\rubinom{d_1+1}{2}+\dots+\rubinom{d_4+1}{2}=1+1+1+3=6$. Thus, assumption
	\cref{eqLgenericdimQ_2} holds  if  we can trust our impression that
		$\mathbf Q_2(\tilde{\mathcal T})$ has full rank generically}.
	
	\textit{II. The matrices $\tilde{\mathbf A}$, $\tilde{\mathbf B}$, and $\tilde{\mathbf C}$ have integer entries.}
	\tcr{We set
	$$	
	\tilde{\mathbf A} = 
	\begin{bmatrix}
	1& 0& 0& 1\\
	0& 1& 0& 1\\
	0& 0& 1& 1
	\end{bmatrix},\qquad
	\tilde{\mathbf B} =
	\begin{bmatrix}
	1& 1& 1& 0& 0\\
	1& 2& 0& 1& 0\\
	1& 3& 0& 0& 1
	\end{bmatrix},\qquad \tilde{\mathbf C} = \mathbf I_5
	$$
	and compute $\tilde{\mathcal T}=\sum\tilde{\mathbf a}_r\circ(\tilde{\mathbf B}_r\tilde{\mathbf C}_r^T)$. It can be easily verified that
		\setcounter{MaxMatrixCols}{15}
	$$
	\mathbf Q_2(\tilde{\mathcal T})=
	\begin{bmatrix*}[r] 
	0&	1&	0&	1&	0&	0&	0&	-1&	0&	0&	0&	0&	0&	0&	0\\
	0&	2&	0&	0&	1&	0&	0&	0&	-1&	0&	0&	0&	0&	0&	0\\
	0&	1&	0&	-1&	1&	0&	0&	3&	-2&	0&	0&	0&	0&	0&	0\\
	0&	0&	-1&	1&	0&	0&	0&	0&	0&	0&	-1&	0&	0&	0&	0\\
	0&	0&	-1&	0&	1&	0&	0&	0&	0&	0&	0&	-1&	0&	0&	0\\
	0&	0&	0&	-1&	1&	0&	0&	0&	0&	0&	0&	0&	0&	0&	0\\
	0&	0&	0&	0&	0&	0&	-2&	1&	0&	0&	-1&	0&	0&	0&	0\\
	0&	0&	0&	0&	0&	0&	-3&	0&	1&	0&	0&	-1&	0&	0&	0\\
	0&	0&	0&	0&	0&	0&	0&	-3&	2&	0&	0&	0&	0&	0&	0
	\end{bmatrix*}
	$$
	and that  the nine nonzero columns of $\mathbf Q_2(\tilde{\mathcal T})$ are linearly independent. Hence, again, by the rank-nullity theorem, $\dim\nullsp{\mathbf Q_2(\tilde{\mathcal T})}=15-9=6$. Thus, assumption \cref{eqLgenericdimQ_2} holds  with certainty.} Note that
	the  matrix $\mathbf Q_2(\tilde{\mathcal T})$ is sparse and the identity in \cref{eqLgenericdimQ_2} is easy to prove because
	 we paid attention to the choice of the entries of $\mathbf A$, $\mathbf B$, and $\mathbf C$.
	   
	   	It is worth noting that  the decomposition of a $3\times 3\times 5$ tensor into a sum of $5$ generic rank-$1$  terms is not  unique. More precisely,
	 it is known that such tensors admit exactly six decompositions \cite{TenBerge2004}. Our example demonstrates  that
	 if two of the rank-$1$ terms are forced to share the same vector in the first mode,  and hence together form an ML rank-$(1,2,2)$ term, then
	 the  decomposition  becomes    unique.
\end{expl}
\subsubsection{\tcr{Necessary condition  for generic uniqueness}}\label{subsec:necccondgen}
The necessity of the conditions 
\begin{equation}
R\leq JK,\quad \sum L_r\leq IJ,\quad \sum L_r\leq IK\label{eq:threegencond}
\end{equation}
follows trivially  from \cref{thm:necessity}.
Next, counting the number of parameters on each side of \cref{eq:LrLr1}, one would expect that uniqueness does not hold 
if the RHS of \cref{eq:LrLr1}  contains more parameters than the LHS:
\begin{equation}
S:=\sum\limits_{r=1}^R (I-1+(J+K-L_r)L_r)< IJK,\label{eq:S=}
\end{equation}
where
the value $S$ is
an upper bound on the number of parameters needed to parameterize\footnote{
	The number of parameters can be computed as follows. Using, for instance,  the LDU factorization  we obtain that   a generic $J\times K$ rank-$L_r$ matrix involves
	$(JL_r-\frac{L_r(L_r+1)}{2})+L_r+(KL_r-\frac{L_r(L_r+1)}{2})= (J+K-L_r)L_r$ parameters, where
		we obviously assume that  $\max L_r\leq\min(J,K)$. Hence,  the $r$th term in \cref{eq:LrLr1}  can be parameterized with $I-1+(J+K-L_r)L_r$ parameters.}
a sum of $R$ generic  \ML terms in the  LHS of  \cref{eq:LrLr1} and 
$IJK$ is equal to the  dimension of the space of  $I\times J\times K$ tensors.
In fact it is known
\cite{Yang2014} and follows from the 
fiber dimension theorem  \cite[Theorem 3.7, p. 78]{Perrin2008}
that \textit{condition \cref{eq:S=}
is necessary} for  generic uniqueness if $\fF=\mathbb C$.
    It can be verified that  condition \cref{eq:S=} is more restrictive than \cref{eq:threegencond}
   \tcr{and, thus, is more interesting at least for $\fF=\mathbb C$}.

Recall that for $L_1=\dots=L_R=1$ the minimal decomposition of form \cref{eq:LrLr1mainBC} corresponds to CPD.
It has been shown in \cite{Nick2014} that, for CPD,  \textit{the condition $S< IJK\leq 15000$ is \tcr{also} sufficient} for  generic uniqueness, with a few known exceptions.
The following example demonstrates that 
for the decomposition into a sum of \MLatmost terms the bound is $S< IJK$ \textit{not} sufficient. 
However, in the example the first factor matrix is generically unique, i.e., the decomposition is generically partially unique. 
  \begin{expl}
	\label{example:1}
	We consider a $2\times 8\times 7$ tensor generated as the sum of $3$ random ML rank-$(1,3,3)$ tensors. More precisely, the tensors are generated by  \cref{eq:LrLr1mainBC} in which the entries of $\mathbf A$, $\mathbf B$, and $\mathbf C$ are independently drawn from the standard normal distribution $N(0,1)$. 
		Since $S=3(2-1+(8+7-3)3)=111$ and $IJK=112$, the inequality $S<IJK$ holds. 	
	    On the other hand, in \cref{AppendixE}  we prove that  tensors generated in this way  admit infinitely many decompositions, namely, we show that there exists \tcr{at least } a two-parameter family of decompositions. 
	   In this example, first, we present a specific tensor $\tilde{\mathcal T}$ that admits a one-parameter family of decompositions all of which share the same factor matrix. Second, we show how $\tilde{\mathcal T}$ can be used to 
	     prove generic uniqueness of the first factor matrix.
	    
	    Let $\tilde{\mathcal T}:=\sum\tilde{\mathbf a}_r\circ(\tilde{\mathbf B}_r\tilde{\mathbf C}_r^T)$ with 
		\begin{align*}
		\tilde{\mathbf A} &=\begin{bmatrix}
		1&1&0\\
		1&0&1
		\end{bmatrix},& &
		\tilde{\mathbf B}_1\tilde{\mathbf C}_1^T=[\mathbf e_5+\mathbf e_7\ \mathbf e_1\ \mathbf e_2\ \mathbf 0\ \mathbf 0\ \mathbf 0\ \mathbf 0],\\
		\tilde{\mathbf B}_2\tilde{\mathbf C}_2^T&=[\mathbf 0\ \mathbf 0\ \mathbf e_5\ \mathbf e_3\ \mathbf e_4\ \mathbf e_5\ \mathbf e_5],& &
				\tilde{\mathbf B}_3\tilde{\mathbf C}_3^T=[\mathbf e_8\ \mathbf 0\ \mathbf e_8\ \mathbf 0\ \mathbf e_8\ \mathbf e_6\ \mathbf e_7],
		\end{align*}
		where $\mathbf e_1,\dots,\mathbf e_8$ denote the vectors of the  canonical basis of $\fF^8$. 
		Let $t\in \fF$, $\mathbf h:=\mathbf e_5-\mathbf e_7$,  $\mathbf g(t):= (t+3)(\mathbf e_5-t\mathbf h)$, and
\begin{align*}
{\mathbf E}_1(t):=&[(2t-1)\mathbf h                      & &\mathbf e_1& &\mathbf g(t)+\mathbf e_5 +\mathbf e_2&              & \mathbf 0  & &-2\mathbf h            & &t\mathbf h             & &t\mathbf h],\\ 
{\mathbf E}_2(t):=&[2(\mathbf e_5-t\mathbf h)            & &\mathbf 0  & &-\mathbf g(t)&                                       & \mathbf e_3& &\mathbf e_4+2\mathbf h & &\mathbf e_5-t\mathbf h & &\mathbf e_5-t\mathbf h],\\
{\mathbf E}_3(t):=&[2(\mathbf e_5-t\mathbf h)+\mathbf e_8& &\mathbf 0  & & -\mathbf g(t)-\mathbf e_5+\mathbf e_8&             &  \mathbf 0 & &2\mathbf h+\mathbf e_8 & &t\mathbf h+\mathbf e_6 & &\mathbf e_7 -t\mathbf h].
 \end{align*}
It can easily be verified that also   $\tilde{\mathcal T}=\sum\tilde{\mathbf a}_r\circ {\mathbf E}_r(t)$, and that the column spaces of  
${\mathbf E}_1(t)$, ${\mathbf E}_2(t)$ and ${\mathbf E}_3(t)$ coincide with $\operatorname{span}\{\mathbf h , \mathbf e_1, \mathbf g(t)+\mathbf e_5 +\mathbf e_2\}$, $\operatorname{span}\{\mathbf e_5-t\mathbf h, \mathbf e_3,\mathbf e_4+2\mathbf h\}$, and  $\operatorname{span}\{2\mathbf h+\mathbf e_8,t\mathbf h+\mathbf e_6,\mathbf e_7 -t\mathbf h\}$, respectively, that is, have dimension $3$. Thus, the decomposition of $\tilde{\mathcal T}$ is not generically unique.

	 \tcr{Generic} uniqueness of the first factor matrix follows from   \cref{item:th1firstfmhardgen} of \cref{thm:maingen}. 
		Indeed,   \cref{item:th1gencond1,item:th1gencond2,eq:ineqforuni1fmgen} are trivial:  $7=K<IJ=16$, $K-\sum L_r+\min L_r = 7 - 9+3=1$, \tcr{$7=K\geq -\frac{1}{2} -\sqrt{\frac{1}{4}+\frac{2L_1L_2}{R-1}}+\sum\limits_{r=1}^R L_r=
		-\frac{1}{2}-\sqrt{\frac{1}{4}+9}+9\approx 5.5$}. \tcr{Condition \cref{eqLgenericdimQ_2} can be verified exactly, i.e., without roundoff errors for the specific $\tilde{\mathbf A}$, $\tilde{\mathbf B}$, and $\tilde{\mathbf C}$ given above. (For this particular choice of $\tilde{T}$, 	the $28\times 28$ matrix $\mathbf Q_2(\tilde{\mathcal T})$ is sparse and its nonzero entries belong to the set $\{-2, -1,0, 1,2\}$).} Moreover, the first factor matrix can be computed in Phase I of  \cref{algorithm:1}. Since $d_r=K - (\sum_{p=1}^R L_p -L_r) = 7-(9-3)=1$, it follows that the S-JBD  in step $5$   reduces to joint diagonalization.
 \end{expl}
\subsubsection{\tcr{Strassen type results: decompositions  with a factor matrix that has  full column rank}}\label{subsub:genstrassen}
In this subsection we narrow the investigation of generic uniqueness to the  situation where one of the factor matrices has full column rank.
Put the other way around,
we generalize the famous Strassen result for generic uniqueness of the CPD for situations in which a factor matrix has full column rank to the decomposition into a sum of \MLatmost terms. While CPD is symmetric in $\mathbf A$, $\mathbf B$ and $\mathbf C$, in the decomposition into a sum of \ML terms factor matrix $\mathbf A$ plays a role that is different from the role of $\mathbf B$ and $\mathbf C$. Consequently, we will consider two cases. In the first case we assume that $R\leq I$, i.e., that the first factor matrix has full column rank (see \cref{thm:maingenLL1}). In the second case we assume that 
$\sum L_r\leq  K$, i.e., that the third factor matrix has full column rank (see \cref{thm:maingenStrassen}). The result for $\sum L_r\leq  J$, i.e., for the case where the second factor matrix has full column rank then follows from  \cref{thm:maingenStrassen} by symmetry. 

\textit{First factor matrix has full column rank.} First we recall the corresponding result for the CPD. One can easily verify that  if $L_1=\dots=L_R=1$ and $R\leq I$, then the bound   $S<IJK$ in \cref{eq:S=} is equivalent to
\begin{equation}
R\leq (J-1)(K-1).
\label{eq:strassen}
\end{equation}
Hence, condition \cref{eq:strassen} is necessary for generic uniqueness of the CPD if $R\leq I$ and $\fF=\mathbb C$.
If $R\leq I$ and $\fF=\mathbb R$, then, in general,  
condition \cref{eq:strassen} is not necessary for  generic uniqueness of  CPD \cite{realvscomplex}.
%
%
On the other hand, it is well-known \cite{Strassen1983} (see also  \tcr{\cite[Corollary 1.7]{AlgGeom1}, \cite{Bocci2013}} and  references therein) that
if  $R\leq I$, then  condition \cref{eq:strassen} is sufficient for generic uniqueness of the CPD for both $\fF=\mathbb R$ and $\fF=\mathbb C$.
Thus, under the assumption $R\leq I$, condition  \cref{eq:strassen} is sufficient  if $\fF=\mathbb R$ and
condition  \cref{eq:strassen} is necessary and sufficient  if $\fF=\mathbb C$. 
%
%
%
The following theorem generalizes this ``Strassen-type'' CPD result for the decomposition into a sum of ML rank-$(1,L,L)$ terms. (One can easily verify
that if $R\leq I$, then the condition $R\leq (J-L)(K-L)$ in \cref{eq:RleqI} is equivalent to  the bound $S<IJK$ in \cref{eq:S=}).
%
%
%
\begin{theorem}\label{thm:maingenLL1}
	Let $\mathcal T$ admit decomposition \cref{eq:LrLr1mainBC}, where
	$$
	L_1=\dots=L_R=:L\leq \min(J,K),\qquad R\leq I
	$$
	and the entries of the matrices $\mathbf A$, $\mathbf B$, and $\mathbf C$ are randomly sampled from an absolutely continuous distribution.
	If $\fF=\mathbb R$  and
	\begin{equation}
	R\leq (J-L)(K-L), \label{eq:RleqI}
	\end{equation}
	then  the decomposition of $\mathcal T$ into a sum of \MLatmost terms is unique.
	If $\fF=\mathbb C$, then the decomposition of $\mathcal T$ into a sum of \MLatmost terms is  unique
		if and only if  \eqref{eq:RleqI} holds.
\end{theorem}
\begin{proof}
 	The proof is  given in \cref{sec:AppendixC}.
\end{proof}


\textit{Second or third factor matrix has full column rank.} Permuting $I$, $J$ and $K$ in the Strassen condition \cref{eq:strassen}, 
 we have that
generic uniqueness of the CPD holds if 
\begin{equation}
R\leq(I-1)(J-1)\ \text{ and }\ R\leq K.\label{eq:strassen2}
\end{equation}
 While \cref{thm:maingenLL1} extended CPD condition \cref{eq:strassen}, 
the following theorem  generalizes   \cref{eq:strassen2} for the decomposition into a sum of \MLatmost terms.
\begin{theorem}\label{thm:maingenStrassen}
		Let $L_1\leq \dots\leq L_R\leq \min(J,K)$ and let	 $\mathcal T$ admit decomposition \cref{eq:LrLr1mainBC}, where
	the entries of the matrices $\mathbf A$, $\mathbf B$, and $\mathbf C$ are randomly sampled from an absolutely continuous distribution.
	If
	\begin{equation}
	 2\leq I,\quad L_{R-1}+L_R\leq J,\qquad\sum\limits_{r=1}^R L_r\leq (I-1)(J-1),\ \ \ \text{ and }\ \ \ \sum\limits_{r=1}^R L_r\leq  K, \label{eq:sumLrleqI1J1}
	\end{equation}
	then  the decomposition of $\mathcal T$ into a sum of \MLatmost terms is unique.
\end{theorem}
\begin{proof}
	 	The proof is  given in \cref{Appendixthm:maingenStrassen}.
\end{proof}
Recall that if $\fF=\mathbb C$, then condition \cref{eq:RleqI} in \cref{thm:maingenLL1} is both necessary and sufficient for generic uniqueness.
 Apparently, condition $\sum\limits_{r=1}^R L_r\leq (I-1)(J-1)$ in \cref{thm:maingenStrassen}  is only sufficient. Indeed,
one can easily verify that if $\sum\limits L_r\leq  K$, then  the necessary bound   $S<IJK$ in \cref{eq:S=} is equivalent to
$
\sum L_r\leq (I-1)(J-1)+ (I-1)\frac{\sum L_r-R}{\sum L_r}+ \frac{\sum L_r^2}{\sum L_r}-1 
$. Thus, the  gap between the necessary bound   $S<IJK$ in \cref{eq:S=} and the sufficient bound  $\sum L_r\leq (I-1)(J-1)$ in \cref{thm:maingenStrassen} is equal to $(I-1)\frac{\sum L_r-R}{\sum L_r}+ \frac{\sum L_r^2}{\sum L_r}-1$. 

\subsection{Constrained decompositions}\label{subsec:constraineddec}
In many applications the factor matrices  $\mathbf A$, $\mathbf B$, and/or $\mathbf C$ in decomposition \cref{eq:LrLr1mainBC}
are subject to constraints like non-negativity \cite{Bro2009},
partial symmetry \cite{Mueller_Smith_2016},
Vandermonde structure of columns \cite{Liu2012},  
etc.

In this subsection we briefly explain how the results from previous sections can be applied to constrained decompositions.

It is clear that \cref{thm: maintheorem} can be applied as is. Indeed, if, for instance,  assumptions
\cref{item:th1cond1,item:th1cond2,eq:ranksofFand G} and \cref{item:th1cond3a,item:th1cond5} in \cref{thm: maintheorem} hold for a constrained decomposition of 
$\mathcal T$, then, by  \cref{item:th1stat4}, the decomposition of $\mathcal T$ into a sum of \MLatmost terms is unique and can be computed by means of (simultaneous) EVD. This result also implies that  \cref{algorithm:1} will  find the constrained decomposition.

Now we discuss variants for generic uniqueness. 
We assume that the factor matrices in the constrained decomposition depend analytically on some  complex or real parameters, which is the case in all instances above. More specifically, we  assume that the entries of $\mathbf A(\mathbf z)$, $\mathbf B(\mathbf z)$, and  $\mathbf C(\mathbf z)$ are analytic functions of $\mathbf z\in\fF^n$ and that the matrix functions  $\mathbf A(\mathbf z)$, $\mathbf B(\mathbf z)$, $\mathbf C(\mathbf z)$ are known.  One can define generic uniqueness of a constrained decomposition similar to the unconstrained case:
 the decomposition of an $I\times J\times K$ tensor into a sum of constrained \MLatmost terms is  generically unique if 
$$
\mu_n\{\mathbf z:\ \text{  decomposition }  
\mathcal T = \sum_{r=1}^R\mathbf a_r(\mathbf z)\circ(\mathbf B_r(\mathbf z)\mathbf C_r(\mathbf z)^T)
  \text{ is not unique} \}=0,
$$
where $\mu_n$ denotes  a measure on $\fF^n$ that is absolutely continuous with respect to the Lebesgue measure. It is clear that
\cref{def:genericuniqueness} corresponds to the case   $n= IR +J\sum L_r+K\sum L_r $. 
Note that depending on structure of the factor matrices, the bounds in the statements of \cref{thm:maingenLL1,thm:maingenStrassen} may not hold or can be further improved. Also, \cref{thm:maingenshort,thm:maingen} cannot be used as is; instead
one should verify that the conditions of \cref{thm: maintheorem} hold for generic $\mathbf z$. Note that, because of the analytical dependency of the factor matrices on $\mathbf z$, it is sufficient to verify the assumptions  and conditions in \cref{thm: maintheorem} for a single  
triplet of constrained factor matrices. 
%
%
\begin{expl}\label{constrainedexample}
In the decomposition considered in \cite{Liu2012}, $\mathbf B$ and $\mathbf C$  are  Vandermonde structured matrices, namely,  
\begin{align*}
\mathbf b_p&=[1\ \exp(jC_1 z_p) \ \dots\ (\exp(jC_1 z_p)^{J-1})]^T,\ p=1,\dots,s\\
\mathbf c_q&=[1\ \exp(jC_2 \sin(z_{s+q})) \ \dots\ \exp(jC_2 \sin(z_{s+q}))^{K-1}]^T,\ q=1,\dots,s,
\end{align*}
 where $C_1$ and $C_2$ are known real values,  $s:=\sum L_r$, and 
 $z_1,\dots,z_{2s}$ are unknown real values. No structure is assumed on $\mathbf A$, so it can parameterized 
 with $IR$ parameters $z_{2s+1},\dots, z_{2s+IR}$ which we will also assume real. Thus, the overall constrained decomposition can be parameterized with  $n=2s+IR$ real parameters. W.l.o.g. we assume that $L_1\leq\dots\leq L_R$. We claim that if
  \begin{equation}
  IJ\geq\sum\limits_{r=1}^R L_r,\quad
  K\geq L_2+\dots+L_R+1,\quad
  R\!\geq\!I\geq 3,\quad  J\geq L_{I-1}+\dots+L_R, \label{eq:thesamecondition}
  \end{equation}
   then the constrained decomposition is generically unique.
    Indeed, generically the matrices $\mathbf B$ and  $\mathbf C$ have maximal $k'$-rank and the matrix $\mathbf A$ has maximal $k$-rank.
  The assumptions in \eqref{eq:thesamecondition} just express the fact that assumptions
  \cref{item:th1cond1,item:th1cond2,eq:ranksofFand G} and \cref{item:th1cond3a,item:th1cond6} in \cref{thm: maintheorem}
  hold generically. Thus, the generic uniqueness of the constrained decomposition follows from  \cref{item:th1stat4} of \cref{thm: maintheorem}.
 %
 %
\end{expl}

\section{Expression of $\mathbf R_2(\mathcal T)$ and $\mathbf Q_2(\mathbf T)$ in terms of $\mathbf A$, $\mathbf B$, and $\mathbf C$}\label{sec:mainidentity}
In this section we explain construction of the matrices 
	 $\Phi(\mathbf A,\mathbf B)$ and $\mathbf S_2(\mathbf C)$ that have appeared in \cref{thm: maintheoremABC}.
	 The results of this section will also  be used later in the proof of \cref{item:th1firstfmhard} of \cref{thm: maintheorem}.  

Let  $\mathbf x,\mathbf y\in\fF^{n}$. Then  $\wprod{\mathbf x}{\mathbf y}$  denotes a $\rubinom{n}{2}\times 1$ vector formed by all $2\times 2$ minors of  $[\mathbf x\ \mathbf y]$  and $\symprod{\mathbf x}{\mathbf y}$ denotes a $\rubinom{n+1}{2}\times 1$ vector formed by all $2\times 2$ permanents of  $[\mathbf x\ \mathbf y]$. More specifically,
\begin{align*}
&\text{the } (n_1+\rubinom{n_2-1}{2})\text{-th entry of } &\wprod{\mathbf x}{\mathbf y}\text{ equals }   &x_{n_1}y_{n_2}-x_{n_2}y_{n_1},& 1\leq n_1< n_2\leq n,\\
&\text{the } (n_1+\rubinom{n_2}{2})\text{-th entry of }   &\symprod{\mathbf x}{\mathbf y}\text{ equals } &x_{n_1}y_{n_2}+x_{n_2}y_{n_1},& 1\leq n_1\leq n_2\leq n.
\end{align*}
\tcr{It can easily be verified that $\wprod{\mathbf x}{\mathbf y}$ and $\symprod{\mathbf x}{\mathbf y}$ coincide with the vectorized strictly upper triangular part
of $\mathbf x\mathbf y^T-\mathbf y\mathbf x^T$ and with the vectorized upper triangular part of  $\mathbf x\mathbf y^T+\mathbf y\mathbf x^T$, respectively.}

We extend the definitions of ``$\wprod{}{}$'' and ``$\symprod{}{}$'' to matrices as follows.
If $\mathbf B_{r_1}\tcr{\in\mathbb F^{J\times L_{r_1}}}$ and  $\mathbf B_{r_2}\tcr{\in\mathbb F^{J\times L_{r_2}}}$ are  submatrices of $\mathbf B$, then 
$ \wprod{\mathbf B_{r_1}}{\mathbf B_{r_2}}$ is the   $\rubinom{J}{2}\times L_{r_1}L_{r_2}$ matrix that has columns
$
\wprod{\mathbf b_{l_1,r_1}}{\mathbf b_{l_2,r_2}}
$,
where $1\leq l_1\leq L_{r_1}$ and $1\leq l_2\leq L_{r_2}$, i.e.,
\begin{equation*}
\wprod{\mathbf B_{r_1}}{\mathbf B_{r_2}}:=[\wprod{\mathbf b_{1,r_1}}{\mathbf b_{1,r_2}}\ \dots \wprod{\mathbf b_{1,r_1}}{\mathbf b_{L_2,r_2}}\ \dots\ 
\wprod{\mathbf b_{L_1,r_1}}{\mathbf b_{1,r_2}}\ \dots \wprod{\mathbf b_{L_1,r_1}}{\mathbf b_{L_2,r_2}}].
\end{equation*}
If $\mathbf C_{r_1}\tcr{\in\mathbb F^{K\times L_{r_1}}}$ and  $\mathbf C_{r_2}\tcr{\in\mathbb F^{K\times L_{r_1}}}$ are submatrices of $\mathbf C$, then 
$ \symprod{\mathbf C_{r_1}}{\mathbf C_{r_2}}$ is the   $\rubinom{K+1}{2}\times L_{r_1}L_{r_2}$ matrix that has columns
$
\symprod{\mathbf c_{l_1,r_1}}{\mathbf c_{l_2,r_2}}
$,
where $1\leq l_1\leq L_{r_1}$ and $1\leq l_2\leq L_{r_2}$, i.e.,
$$
\symprod{\mathbf C_{r_1}}{\mathbf C_{r_2}}:=[\symprod{\mathbf c_{1,r_1}}{\mathbf c_{1,r_2}}\ \ldots \symprod{\mathbf c_{1,r_1}}{\mathbf c_{L_2,r_2}}\ \ldots\ 
\symprod{\mathbf c_{L_1,r_1}}{\mathbf c_{1,r_2}}\ \ldots \symprod{\mathbf c_{L_1,r_1}}{\mathbf c_{L_2,r_2}}].
$$
Let $\mathbf P_n$ denote the $n^2\times \rubinom{n+1}{2}$ matrix defined on all vectors of the form $\symprod{\mathbf x}{\mathbf y}$ by  
\begin{equation}
\mathbf P_n(\symprod{\mathbf x}{\mathbf y})=\mathbf x\otimes \mathbf y+\mathbf y\otimes \mathbf x 
\label{eq:defPn}
\end{equation}
and extended by  linearity. It can be easily checked that for $n=K$ the matrix $\mathbf P_n$ can be constructed as in  \cref{eq:matrixP_K},
so $\mathbf P_n^T$ is a column selection matrix.
\begin{lemma}\label{lemma: Q2viaABC} 
	Let $\mathcal T$ admit decomposition  \cref{eq:LrLr1mainBC}\tcr{, $r_{\mathbf C}=K$, and let the values $d_r$ be defined in \cref{item:th1cond2ABC}}. Define the $\rubinom{I}{2}\rubinom{J}{2}\times \sum\limits_{r_1<r_2} L_{r_1}L_{r_2}$ matrix
	$\Phi(\mathbf A,\mathbf B)$ and $\rubinom{K+1}{2} \times \sum\limits_{r_1<r_2} L_{r_1}L_{r_2}$ matrix $\mathbf S_2(\mathbf C)$ as
	\begin{align}
	\Phi(\mathbf A,\mathbf B)&:=
	\left[(\wprod{\mathbf a_1}{\mathbf a_2})\otimes(\wprod{\mathbf B_1}{\mathbf B_2})\ \dots\ 
	(\wprod{\mathbf a_{R-1}}{\mathbf a_R})\otimes(\wprod{\mathbf B_{R-1}}{\mathbf B_R})
	\right],\label{eq:matrixPhi}\\%
	\mathbf S_2(\mathbf C)&:=[\symprod{\mathbf C_{1}}{\mathbf C_2}\ \dots\ \symprod{\mathbf C_{R-1}}{\mathbf C_{R}}].\label{eq:matrixS2}
	\end{align}
	Then 
	\begin{statements}
		\item\label{sta:lemmaApp1}  	$\mathbf Q_2(\mathcal T)=\Phi(\mathbf A,\mathbf B)\mathbf S_2(\mathbf C)^T$;
		\item\label{sta:lemmaApp2} 
		 	$\mathbf R_2(\mathcal T)=\Phi(\mathbf A,\mathbf B)\mathbf S_2(\mathbf C)^T\mathbf P_K^T$, where $\mathbf P_K$ is defined as in \cref{eq:defPn};
		 \item\label{sta:lemmaApp3}   	\tcr{$\dim\nullsp{\mathbf Q_2(\mathcal T)}\geq\dim\nullsp{\mathbf S_2(\mathbf C)^T}=\sum\rubinom{d_r+1}{2}$;}
		 \item\label{sta:lemmaApp4}    if $r_{\mathbf A}+k_{\mathbf B}'\geq R+2$ and 	 $k_{\mathbf A}\geq 2$, then the matrix  $\Phi(\mathbf A,\mathbf B)$ has full column rank and $\dim\nullsp{\Phi(\mathbf A,\mathbf B)\mathbf S_2(\mathbf C)^T}=\sum\rubinom{d_r+1}{2}$, i.e., \cref{eq:ranksofFand GABC} implies \cref{item:th1cond3ABC}; similarly,   \cref{eq:ranksofFand G} implies \cref{item:th1cond3};
		 \item\label{sta:lemmaApp45} \tcr{If $\Phi(\mathbf A,\mathbf B)$ has full column rank, then
		 $[\mathbf a_1\otimes\mathbf B_1\ \dots\ \mathbf a_R\otimes\mathbf B_R]$ also has full column rank;}
		 \item\label{sta:lemmaApp5} \tcr{If $\Phi(\mathbf A,\mathbf B)$ has full column rank, then $k_{\mathbf B}'\geq 2$.}
	\end{statements}
\end{lemma}
\begin{proof}
	The proofs of \cref{sta:lemmaApp1,sta:lemmaApp2,sta:lemmaApp5}  follow from the construction of the matrices $\mathbf Q_2(\mathcal T)$, $\Phi(\mathbf A,\mathbf B)$, $\mathbf S_2(\mathbf C)$  and are   therefore grouped in \cref{sec:appendixB}. 
	The proof of \cref{sta:lemmaApp3} consists of several steps and is given in a dedicated \cref{sec:appendixLemmas3}.
	The proofs of  \cref{sta:lemmaApp4,sta:lemmaApp45} rely on  \cref{lemma:compound},  which  contains auxiliary results on compound matrices.
	 \Cref{lemma:compound} and statements \ref{sta:lemmaApp4}, \ref{sta:lemmaApp45}	are proved in  \cref{sec:appendixLemmas4}.
	\end{proof}
\begin{corollary}\label{Lemma:redtobtdI}
	Let $\mathcal T\in\fF^{I\times J\times K}$ admit the \ML decomposition  \cref{eq:LrLr1mainBC}. Let also the matrices $\mathbf A$ and $\mathbf C$ have full column rank and  assumptions  \cref{item:th1cond1ABC,item:th1cond2ABC,item:th1cond3ABC}
in \cref{thm: maintheoremABC}  hold. Then  the matrices $[\mathbf B_i\ \mathbf B_j]$ have full column rank for all $1\leq i<j\leq R$.
In particular,  \cref{assum:bth1.5} in \cref{thm:ll1_btd1} holds.
\end{corollary}
\begin{proof}
	The proof is   given in \cref{sec:appendixB}.
\end{proof}
\section{Proof of \cref{thm: maintheorem} }\label{sec:proof_of_main_thm}
We will need the following two lemmas.
\begin{lemma}\label{lemma:new}
	Let $\mathcal T\in\fF^{I\times J\times K}$ admit the \ML decomposition  \cref{eq:LrLr1}. Assume that conditions \cref{item:th1cond1,item:th1cond2} hold.
	Let $\mathbf N_r$ be a $K\times d_r$ matrix whose columns form a basis
	of  $\nullsp{\mathbf Z_r}$ and  let  $\mathbf M_r$ be a $d_r^2\times\rubinom{d_r+1}{2}$ matrix whose columns form a basis of the subspace $\vecsym{d_r}$ (see \cref{eq:defvecsymM}), $r=1,\dots,R$. 
	By definition, set
	$$
	\mathbf N:=[\mathbf N_1\ \dots\ \mathbf N_R],\qquad 
	\mathbf W:=[(\mathbf N_1\otimes \mathbf N_1)\mathbf M_1\ \dots\ (\mathbf N_R\otimes \mathbf N_R)\mathbf M_R].
	$$
	The following statements hold.
	\begin{statements}
		\item\label{eq:Nfcr} The $ K\times\sum d_r$ matrix $\mathbf N$ has full column rank.
		\item \label{eq:NNfcr} The $ K^2\times Q$ matrix $\mathbf W$ has full column rank\tcr{, where  $Q=\rubinom{d_1+1}{2}+\dots+\rubinom{d_R+1}{2}$}.
		\item \label{linindEEE} The matrices $\mathbf E_1,\dots,\mathbf E_R$ are linearly independent.
	\end{statements}
\end{lemma}
\begin{proof}
The proof is   given in \cref{sec:appendixC}.
\end{proof}
\begin{lemma}\label{lemma:lemma3.1} Let $\mathcal T\in\fF^{I\times J\times K}$ admit the  \ML decomposition \cref{eq:LrLr1} in which  the matrices	$\mathbf E_1,\dots,\mathbf E_R$ are linearly independent and  such that either 	\cref{item:th1cond5} or \cref{item:th1cond6} in \cref{thm: maintheorem} holds.
	Then the following statements hold.
	\begin{statements}
		\item\label{st:hardlemma1} If the matrix $\mathbf A$ is known, then the matrices $\mathbf E_1,\dots,\mathbf E_R$ can be computed by means of EVD.
		\item\label{st:hardlemma2} Any  decomposition of $\mathcal T$ of the form
		\begin{equation*}
		\mathcal T = \sum_{r=1}^{\tilde R}\tilde{\mathbf a}_r\circ\tilde{\mathbf E}_r,\ \ \tilde{\mathbf a}_r\text{ is a column of }\mathbf A,\ \ \tilde{\mathbf E}_r\in\fF^{J\times K},\ \ 1\leq r_{\tilde{\mathbf E}_r}\leq L_r,\ \ \tilde R\leq R
		\end{equation*}
		coincides with decomposition  \cref{eq:LrLr1}.
	\end{statements}
\end{lemma}
\begin{proof}
	The proof is   given in \cref{sec:appendixC}.
\end{proof}

\begin{proof}[Proof of \cref{thm: maintheorem}]
 {\em Proof of \cref{item:th1stat1}.}
 Let $\mathbf T_1,\dots,\mathbf T_K$ denote the frontal slices of $\mathcal T$,  $\mathbf T_k:=(t_{ijk})_{i,j=1}^{I,J}$ and let 	 $\mathbf N_r$ be a $K\times d_r$ matrix whose columns form a basis
 	of  $\nullsp{\mathbf Z_r}$.
 If $\mathbf f=\mathbf N_r\mathbf x$ for some nonzero $\mathbf x\in\fF^{d_r}$, then
\begin{equation}
\begin{split}
f_1\mathbf T_1 + \dots +f_K\mathbf T_K = \sum\limits_{k=1}^Kf_k\sum\limits_{q=1}^R\mathbf a_q\mathbf e_{k,q}^T=\sum\limits_{q=1}^R\mathbf a_q\sum\limits_{k=1}^K \mathbf e_{k,q}^Tf_k
=\\
\sum\limits_{q=1}^R\mathbf a_q(\mathbf E_q\mathbf f)^T=\sum\limits_{q=1}^R\mathbf a_q(\mathbf E_q\mathbf N_r\mathbf x)^T=\mathbf a_r(\mathbf E_r\mathbf N_r\mathbf x)^T,
\end{split}
\label{eq:f1T1fKT1K}
\end{equation}
where $\mathbf e_{k,q}$ denotes the $k$th column of $\mathbf E_q$. Thus,
\begin{equation}
r_{f_1\mathbf T_1 + \dots +f_K\mathbf T_K}\leq 1\ \text{for all }\mathbf f=\mathbf N_r\mathbf x,\ \text{where } \mathbf x\in\fF^{d_r},\ r=1,\dots,R.
\label{eq:rank1x}
\end{equation}
In \cref{subsec:222} we have explained that the condition $r_{f_1\mathbf T_1 + \dots +f_K\mathbf T_K}\leq 1$ is equivalent to the condition $\mathbf R_2(\mathcal T)(\mathbf f\otimes\mathbf f)=\mathbf 0$, where the matrix $\mathbf R_2(\mathcal T)$ is constructed in \cref{def:R2}, i.e., that  equality \cref{eq:rank1eqR20} holds. Hence
from \cref{eq:rank1x}, \cref{eq:rank1eqR20} and the identity
$$
\mathbf R_2(\mathcal T)(\mathbf f\otimes\mathbf f)=\mathbf R_2(\mathcal T)((\mathbf N_r\mathbf x)\otimes(\mathbf N_r\mathbf x))=\mathbf R_2(\mathcal T)(\mathbf N_r\otimes\mathbf N_r)(\mathbf x\otimes\mathbf x),
$$
it follows that
\begin{equation}
\mathbf R_2(\mathcal T)(\mathbf N_r\otimes\mathbf N_r)(\mathbf x\otimes\mathbf x)=\mathbf 0,\ \text{ for all  } \mathbf x\in\fF^{d_r}\ \text{ and } r=1,\dots,R.
\label{eq:R2xx}
\end{equation}
Since
$$
\vecsym{d_r}=\sspan\{ \mathbf x\otimes \mathbf x:\ \mathbf x\in\fF^{d_r}\},
$$
it follows that \cref{eq:R2xx} is equivalent to
\begin{equation*}
\mathbf R_2(\mathcal T)(\mathbf N_r\otimes\mathbf N_r)\mathbf m_r=\mathbf 0,\ \text{ for all  } \mathbf m_r\in\vecsym{d_r}\ \text{ and } r=1,\dots,R.
\end{equation*}
In other words,
\begin{equation}
\mathbf R_2(\mathcal T)(\mathbf N_r\otimes\mathbf N_r)\mathbf M_r=\mathbf O,\qquad r=1,\dots,R,
\label{eq:R2NrNrMr0}
\end{equation}
where $\mathbf M_r$ is a $d_r^2\times\rubinom{d_r+1}{2}$ matrix whose columns form a basis of  $\vecsym{d_r}$. By  \cref{eq:NNfcr} of \cref{lemma:new} and \cref{eq:R2NrNrMr0},
$\mathbf R_2(\mathcal T)\mathbf W=\mathbf O$. Since the columns of $\mathbf W$ belong to $\vecsym{K}$, it follows that
\begin{equation}
\text{ column space of } \mathbf W\subseteq \nullsp{\mathbf R_2(\mathcal T)}\cap\vecsym{K}.
\label{eq:rhsQlhsQ}
\end{equation}
By \cref{eq:NNfcr} of \cref{lemma:new}, the column space of $\mathbf W$ has dimension $Q$. On the other hand, from
\cref{eq:nullR2viaNullQ2,item:th1cond3} it follows that the dimension of $\nullsp{\mathbf R_2(\mathcal T)}\cap\vecsym{K}$
is also $Q$. Hence, by \cref{eq:rhsQlhsQ},
\begin{equation}
\text{ column space of } \mathbf W= \nullsp{\mathbf R_2(\mathcal T)}\cap\vecsym{K}.
\label{eq:Wsp=R2sp}
\end{equation}
Let $\vectoru_1,\dots,\vectoru_Q$ be a basis of $\nullsp{\mathbf R_2(\mathcal T)}\cap\vecsym{K}$. Then there exists a nonsingular
$Q\times Q$ matrix $\mathbf M$ such that 
\begin{multline}
[\vectoru_1\ \dots\ \vectoru_Q] = \mathbf W\mathbf M = 
[(\mathbf N_1\otimes\mathbf N_1)\mathbf M_1\ \dots \ (\mathbf N_R\otimes\mathbf N_R)\mathbf M_R]\mathbf M=\\
[\mathbf N_1\otimes\mathbf N_1 \dots \ \mathbf N_R\otimes\mathbf N_R]\Bdiag(\mathbf M_1,\dots,\mathbf M_R)\mathbf M=:
[\mathbf N_1\otimes\mathbf N_1 \dots \ \mathbf N_R\otimes\mathbf N_R]\tilde{\mathbf M},
\label{eq:matrixJBDLL1}
\end{multline}
where
$$
\tilde{\mathbf M}=\Bdiag(\mathbf M_1,\dots,\mathbf M_R)\mathbf M\in\fF^{\sum d_r^2\times Q}.
$$
Let 
$$
\mathbf D_q: =\Bdiag(\mathbf D_{1,q},\dots,\mathbf D_{R,q})\in\fF^{\sum q_r\times \sum q_r},  
$$
where the blocks $\mathbf D_{1,q},\dots,\mathbf D_{R,q}$ are defined as
$$
\begin{bmatrix}
\operatorname{vec}(\mathbf D_{1,q})\\
\vdots\\
\operatorname{vec}(\mathbf D_{R,q})
\end{bmatrix}=\text{ the } q\text{-th column of }\ \tilde{\mathbf M}
$$
and let $\matrixU_q$ denote the $K\times K$   matrix such that $\vectoru_q=\operatorname{vec}(\matrixU_q)$, $q=1,\dots,Q$.
Thus, we can rewrite \cref{eq:matrixJBDLL1} as
\begin{equation}
\matrixU_q = [\mathbf N_1\ \dots\ \mathbf N_R]\mathbf D_q [\mathbf N_1\ \dots\ \mathbf N_R]^T = \mathbf N\mathbf D_q\mathbf N^T,\qquad q=1,\dots,Q.
\label{eq:BTDStat1}
\end{equation}
Since $\matrixU_1,\dots,\matrixU_Q$ are symmetric and since,  by \cref{eq:Nfcr} of \cref{lemma:new}, the matrix $\mathbf N$ has full column rank, it follows easily that
the matrices $\mathbf D_1,\dots,\mathbf D_Q$ are also symmetric. Besides, since $\matrixU_1,\dots,\matrixU_Q$ are linearly independent, the same \tcr{holds} for $\mathbf D_1,\dots,\mathbf D_Q$. Thus, \cref{eq:BTDStat1} is the S-JBD problem of the form \cref{eq:JBDaux}.
By  \cref{thm:JBDaux}, the solution of \cref{eq:BTDStat1} is unique and can be computed by means of (simultaneous) EVD. 
Now we can use the matrices $\mathbf N_r$ to recover the columns of $\mathbf A$.  Recall that  the matrix $\mathbf N_r$ holds a basis of $\nullsp{\mathbf Z_{r}}$, so we can repeat the derivation in \cref{eq:NrinnullCk,eq:a_rviaN_r,eq:22prime} and obtain that 
the column $\mathbf a_r$ is proportional to the right singular vector of the matrix  $[\operatorname{vec}(\mathbf N_r^T\mathbf H_1^T)\ \dots\ \operatorname{vec}(\mathbf N_r^T\mathbf H_1^T)]$
corresponding to the only  nonzero singular value.

{\em Proof of  \cref{item:th1stat2}.} 
By \cref{linindEEE} of \cref{lemma:new}, the matrices $\mathbf E_1,\dots,\mathbf E_R$ are linearly independent and,
by \cref{item:th1stat1}, we can assume that the matrix $\mathbf A$ is known. Thus, the result follows from \cref{st:hardlemma1} of \cref{lemma:lemma3.1}.
 
{\em Proof of  \cref{item:th1newstat}.}
  We assume that  $\mathcal T$ admits an alternative decomposition of the form \cref{eq:LrLr1}:
\begin{equation*}
\mathcal T = \sum_{r=1}^{\tilde R}\tilde{\mathbf a}_r\circ\tilde{\mathbf E}_r, \quad \tilde{\mathbf a}_r\in\fF^I\setminus\{\mathbf 0\},\quad \tilde{\mathbf E}_r\in\fF^{J\times K},\quad 1\leq r_{\tilde{\mathbf E}_r}\leq L_r,
\end{equation*}
\tcr{in which we obviously  assume that $\tilde R\leq R$.}
First we show that $\tilde R=R$.
From  \cref{item:th1cond3a,item:th1cond1} it follows that
\begin{equation}
\sum\limits_{k=1}^R L_k -\min\limits_{1\leq k\leq R} L_k +1\leq K=r_{\unf{T}{3}}\leq \sum\limits_{k=1}^{\tcr{\tilde R}} r_{\tilde{\mathbf E}_k}\leq \sum\limits_{k=1}^{\tilde R}   L_k.
\label{eq:sumLsumL}
\end{equation}
Assuming that $\tilde R< R$, we obtain, by \cref{eq:sumLsumL},  the contradiction
$$
L_R=L_R+\sum\limits_{k=1}^{\tilde R} L_k-\sum\limits_{k=1}^{\tilde R}   L_k\leq 
\sum\limits_{k=1}^{R} L_k-\sum\limits_{k=1}^{\tilde R}   L_k\leq  \min\limits_{1\leq k\leq R} L_k-1<L_R.
$$
Thus $\tilde R=R$. 

Now we prove that each $\tilde{\mathbf a}_r$ is proportional to a column of $\mathbf A$.
   By definition, set
   $$
   \tilde{d}_r:=\dim\nullsp{\tilde{\mathbf Z}_r},\ \text{where }
   \tilde{\mathbf Z}_r := [\tilde{\mathbf E}_1^T\ \dots\ \tilde{\mathbf E}_{r-1}^T\ \tilde{\mathbf E}_{r+1}^T\ \dots\ \tilde{\mathbf E}_R^T]^T,\qquad 
   r=1,\dots, R.
   $$
   Since $r_{\tilde{\mathbf Z}_r}\leq\min(\sum L_r - \min L_r, K)$, it follows from  \cref{item:th1cond3a}  that  $\tilde{d}_r\geq 1$. Let $\tilde{\mathbf N}_r$ be a $K\times \tilde{d}_r$ matrix whose columns form a basis    of  $\nullsp{\tilde{\mathbf Z}_r}$. 
If $\mathbf f=\tilde{\mathbf N}_r\mathbf x$ for some nonzero $\mathbf x\in\fF^{\tilde{d}_r}$, then we obtain
(see \cref{eq:f1T1fKT1K}) that
\begin{equation*}
f_1\mathbf T_1 + \dots +f_K\mathbf T_K = \tilde{\mathbf a}_r(\tilde{\mathbf E}_r\tilde{\mathbf N}_r\mathbf x)^T,\qquad r=1,\dots,R.
\end{equation*}
By   \cref{item:th1cond1}, the linear combination $f_1\mathbf T_1 + \dots +f_K\mathbf T_K$ is not zero for any $f_1,\dots,f_K$ \tcr{such that $\mathbf f\ne \mathbf 0$}. Hence, for any column
$\tilde{\mathbf a}_r$ there exist $f_1,\dots,f_K$ such that  the column space of  the linear combination $f_1\mathbf T_1 + \dots +f_K\mathbf T_K$  is one-dimensional and is spanned by $\tilde{\mathbf a}_r$.
Thus,  to prove that each $\tilde{\mathbf a}_r$ is proportional to a column of $\mathbf A$, it is sufficient to show that
the following implication holds:
\begin{equation}
f_1\mathbf T_1 + \dots +f_K\mathbf T_K=\mathbf z\mathbf y^T \ \Rightarrow\ \text{there exists }\  r \text{ such that } \mathbf z=c\mathbf a_r.
\label{eq:implftrank1}
\end{equation}
If $r_{f_1\mathbf T_1 + \dots +f_K\mathbf T_K}=1$, then, by \cref{eq:rank1eqR20}, $\mathbf R_2(\mathcal T)(\mathbf f\otimes\mathbf f)=\mathbf 0$. Hence, by \cref{eq:Wsp=R2sp},
 $\mathbf f\otimes \mathbf f$ belongs to the column space of the matrix $\mathbf W$.  
Hence, there exists a block diagonal matrix $\mathbf D$ such that $\mathbf f\mathbf f^T=\mathbf N\mathbf D\mathbf N^T$. Since, by  \cref{eq:Nfcr} of \cref{lemma:new}, $\mathbf N$ has full column rank, the matrix $\mathbf D$ contains exactly one nonzero block and its rank is one. In other words, $\mathbf f$ belongs to the null space of $\mathbf N_r$ for some $r=1,\dots,R$.
Hence implication \cref{eq:implftrank1} follows from \cref{eq:f1T1fKT1K}.

{\em Proof of  \cref{item:th1firstfmhard}.}
Let $\tilde{\mathbf A}$, $\tilde{\mathbf B}$, and $\tilde{\mathbf C}$ denote the factor matrices of an alternative  decomposition of $\mathcal T$ into a sum of \MLatmost terms. 
By \cref{item:th1newstat}, it is sufficient to show that  $\tilde{\mathbf A}$  does not have repeated columns. We argue by contradiction. If   $\tilde{\mathbf a}_i=\tilde{\mathbf a}_j$ for some $i\ne j$, then $\tcr{\wprod{\tilde{\mathbf a}_i}{\tilde{\mathbf a}_j}=\mathbf 0}$. Hence, the matrix $\Phi(\tilde{\mathbf A},\tilde{\mathbf B})$ defined in \cref{eq:matrixPhi}, has at least $L_iL_j$ zero columns, implying that $r_{\Phi(\tilde{\mathbf A},\tilde{\mathbf B})}\tcr{\leq}\sum\limits_{1\leq r_1<r_2\leq R} L_{r_1}L_{r_2} - L_iL_j$. Hence,  by  \cref{sta:lemmaApp1} of \cref{lemma: Q2viaABC},
\begin{multline}
r_{\mathbf Q_2(\mathcal T)} =    r_{\Phi(\tilde{\mathbf A},\tilde{\mathbf B})\mathbf S_2(\tilde{\mathbf C})^T} 
\leq r_{\Phi(\tilde{\mathbf A},\tilde{\mathbf B})}\leq\\
\sum\limits_{1\leq r_1<r_2\leq R} L_{r_1}L_{r_2} - L_iL_j\leq
\sum\limits_{1\leq r_1<r_2\leq R} L_{r_1}L_{r_2} - \tilde{L}_1\tilde{L}_2.\label{eq:oppositeineq}
\end{multline}
On the other hand, from the rank-nullity theorem and \cref{eq:ineqforuni1fm} it follows that 
$$
 r_{\mathbf Q_2(\mathcal T)} = \rubinom{K+1}{2} - Q >  \sum\limits_{1\leq r_1<r_2\leq R} L_{r_1}L_{r_2} - \tilde{L}_1\tilde{L}_2,
$$
which is a contradiction with \cref{eq:oppositeineq}.

{\em Proof of  \cref{item:th1stat4}.}
If \cref{item:th1cond3a,item:th1cond5} hold or \cref{item:th1cond3a,item:th1cond6} hold, then the result follows from \cref{item:th1newstat} and \cref{lemma:lemma3.1}.

Let \cref{item:th1cond4} hold.  Then the matrices $\mathbf C$ and $\mathbf N$ are square nonsingular and, by  \cref{eq:NrinnullCk},
$\mathbf C^T\mathbf N=\Bdiag(\mathbf C_1^T\mathbf N_1,\dots,\mathbf C_R^T\mathbf N_R)$. Hence 
$$
\mathbf C = \mathbf N^{-T}\Bdiag(\mathbf N_1^T\mathbf C_1,\dots,\mathbf N_R^T\mathbf C_R)
$$
 in which the matrices $\mathbf N_r^T\mathbf C_r\in\fF^{L_r\times L_r}$ are also nonsingular. Thus, w.l.o.g.  we can set $\mathbf C=\mathbf N^{-T}$. Finally, by \cref{eq:unf_T_2}, the matrix $\mathbf B$  can be \tcr{uniquely} recovered  from the set of linear equations $[\mathbf a_1\otimes\mathbf C_1\ \dots\ \mathbf a_R\otimes\mathbf C_R]\mathbf B^T=\unf{T}{2} $. We can also avoid the computation of $\mathbf N^{-T}$ and proceed as in steps $8-9$ of \cref{algorithm:1}  (for details we refer to ``Case 1'' after
	\cref{thm: maintheoremABC}).

To prove the uniqueness it is sufficient to show that assumptions \cref{item:th1cond1,item:th1cond2,item:th1cond3} and \cref{item:th1cond4} hold for any decomposition of $\mathcal T$ into a sum of \MLatmost terms.
Assume that $\mathcal T$ admits an alternative decomposition with factor matrices
$\tilde{\mathbf A}=[\tilde{\mathbf a}_1\ \dots\ \tilde{\mathbf a}_{\tilde R}]$, 
$\tilde{\mathbf B}=[\tilde{\mathbf B}_1\ \dots\ \tilde{\mathbf B}_{\tilde R}]$, and 
$\tilde{\mathbf C}=[\tilde{\mathbf C}_1\ \dots\ \tilde{\mathbf C}_{\tilde R}]$, where $\tilde{R}\leq R$, the matrices
$\tilde{\mathbf B}_r\in\fF^{J\times \tilde{L}_r}$ and $\tilde{\mathbf C}_r\in\fF^{K\times \tilde{L}_r}$ have full column rank, and $\tilde{L}_r\leq L_r$ for $1\leq r\leq \tilde{R}$. 
Then, by \cref{eq:unf_T_3}, 
\begin{equation}
\unf{T}{3}=[\mathbf a_1\otimes\mathbf B_1\ \dots\ \mathbf a_R\otimes\mathbf B_R]\mathbf C^T=
[\tilde{\mathbf a}_1\otimes\tilde{\mathbf B}_1\ \dots\ \tilde{\mathbf a}_{\tilde R}\otimes\tilde{\mathbf B}_{\tilde R}]\tilde{\mathbf C}^T.
\label{eq:lastidentity}
\end{equation}
Since $r_{\unf{T}{3}}=K$ and $\mathbf C$ is $K\times K$ nonsingular, it readily follows 
from \cref{eq:lastidentity}
that $\tilde{R}=R$,  that $\tilde{L}_r=L_r$ for all $r$ and that $\tilde{\mathbf C}$ is $K\times K$ nonsingular. 
Hence, the values $d_1,\dots,d_R$ in \cref{item:th1cond2ABC} and the values $d_1,\dots,d_R$ computed for the alternative decomposition are equal to  $L_1,\dots,L_R$, respectively. Thus,  assumptions \cref{item:th1cond1,item:th1cond2,item:th1cond3} and \cref{item:th1cond4}
hold for the alternative decomposition. 
%
%
\end{proof}
 \section{Conclusion}
In this paper we have studied the decomposition of a third-order tensor into a sum of \ML terms.
We have obtained conditions for uniqueness of the first factor matrix and for uniqueness of the overall decomposition.
We have also presented an algorithm that computes the decomposition, 
estimates the  number of \ML terms $R$ and  their ``sizes'' $L_1,\dots,L_R$.
All steps of the algorithm rely on   conventional linear algebra.
In the case where the decomposition is not exact, 
a noisy version of the algorithm can compute an approximate \ML decomposition. In our examples the accuracy of the estimates was of about the same order as the accuracy of the tensor. 

The \ML  decomposition takes an intermediate place between the  
little studied decomposition into a sum of ML rank-$(M_r,N_r,L_r)$ terms and the well studied CPD (the special case where $M_r=N_r=L_r=1$). Namely, the \ML decomposition is the special case where $M_r=1$ and $N_r=L_r$.
The results in this paper may be used as stepping stones towards a better understanding of the ML rank-$(M_r,N_r,L_r)$ decomposition. 
 \section*{Acknowledgments}
The authors would like to thank Yang Qi (The University of Chicago)   for his comments on \cref{sec:genuniq}.


\appendix
\section{On testing  \cref{eqLgenericdimQ_2} over a finite field}\label{Appendix:FF}
	In this appendix we explain how   to verify assumption \cref{eqLgenericdimQ_2} over a finite field.
	We also explain how to test whether
	the decomposition of an $I\times J\times K$ tensor into a sum of \MLatmost terms is generically unique under
	the assumptions in row $6$ of \cref{tab:KoMa14}.
	
	We rely on an idea  proposed in \cite{Nick2014}. The idea is to  generate random integer  matrices $\tilde{\mathbf A}_r$, $\tilde{\mathbf B}_r$, $\tilde{\mathbf C}_r$ and then to perform all computations over a finite field  $GF(p^k)$, where $p$ is prime. Obviously, if \cref{eqLgenericdimQ_2} holds for $\tilde{\mathbf A}_r$, $\tilde{\mathbf B}_r$ and $\tilde{\mathbf C}_r$ considered over $GF(p^k)$, then it will necessarily hold for
	$\tilde{\mathbf A}_r$, $\tilde{\mathbf B}_r$ and $\tilde{\mathbf C}_r$ considered over $\fF$\footnote{\tcr{In the proof of   \cref{thm:maingen} we have explained that this will in turn apply that
			\cref{eqLgenericdimQ_2} holds over $\fF$ for generic  $\tilde{\mathbf A}_r$, $\tilde{\mathbf B}_r$, $\tilde{\mathbf C}_r$.}}.
	On the other hand, if \cref{eqLgenericdimQ_2} does not hold for $\tilde{\mathbf A}_r$, $\tilde{\mathbf B}_r$, $\tilde{\mathbf C}_r$ over $GF(p^k)$, then no conclusion can be drawn. In this case  one
	can  try to repeat the computations for other random integer  matrices $\tilde{\mathbf A}_r$, $\tilde{\mathbf B}_r$, $\tilde{\mathbf C}_r$
	or increment $k$, or choose another prime $p$.
	If  \cref{eqLgenericdimQ_2}  does not hold for several such trials, this can be an indication that \cref{eqLgenericdimQ_2}
	does not hold for any $\tilde{\mathbf A}_r$, $\tilde{\mathbf B}_r$ and $\tilde{\mathbf C}_r$. 
	Note that, by the rank-nullity theorem,  the computation of the null space can be reduced to the  computation of the rank.
	Although the   computation  of the rank over the finite field is more expensive than the numerical estimation of the rank, it has the advantage that the dimension  in \cref{eqLgenericdimQ_2} is computed exactly, i.e., without roundoff errors.
		
	Now we explain how to test whether the bounds in row $6$ of \cref{tab:KoMa14} guarantee generic uniqueness of the decomposition. 
	By \cref{lemma: Q2viaABC}, $\mathbf Q_2(\tilde{\mathcal T})$ can be  factorized as  
	$\mathbf Q_2(\tilde{\mathcal T})=\Phi(\tilde{\mathbf A},\tilde{\mathbf B})\mathbf S_2(\tilde{\mathbf C})$, 
	where $\Phi(\tilde{\mathbf A},\tilde{\mathbf B})$ is an	$\rubinom{I}{2}\rubinom{J}{2}\times \sum\limits_{r_1<r_2} L_{r_1}L_{r_2}$ matrix
	and $\mathbf S_2(\tilde{\mathbf C})$ is an $\rubinom{K+1}{2} \times \sum\limits_{r_1<r_2} L_{r_1}L_{r_2}$ matrix.
	Also, by \cref{sta:lemmaApp3} of \cref{lemma: Q2viaABC}, 	
		$\dim\nullsp{\mathbf S_2(\tilde{\mathbf C})^T}=\sum\rubinom{d_r+1}{2}$ for generic $\tilde{\mathbf C}$.
	It is clear now that if $
	\Phi(\tilde{\mathbf A},\tilde{\mathbf B})$ has full column rank, then  \cref{eqLgenericdimQ_2} holds for $\tilde{\mathbf A}$, $\tilde{\mathbf B}$ and generic $\tilde{\mathbf C}$. 
	
	We claim that the assumptions	$\rubinom{I}{2}\rubinom{J}{2}\geq \sum\limits_{r_1<r_2} \hspace{-2mm}L_{r_1}L_{r_2}$ 
	and $J\geq L_{R-1}+L_R$  in 	row $6$ of \cref{tab:KoMa14} are necessary 	for $\Phi(\tilde{\mathbf A},\tilde{\mathbf B})$ to have full column rank.
	Indeed, the former 	expresses the fact that the number of columns of $\Phi(\tilde{\mathbf A},\tilde{\mathbf B})$ does not exceed the number of its rows. The latter means that $k_{\tilde{\mathbf B}}'\geq 2$ holds for generic $\tilde{\mathbf B}$, which, by \cref{sta:lemmaApp5} of \cref{lemma: Q2viaABC}, is necessary for full column rank of $\Phi(\tilde{\mathbf A},\tilde{\mathbf B})$.
	To verify that $\Phi(\tilde{\mathbf A},\tilde{\mathbf B})$ has full column rank
	for some $\tilde{\mathbf A}$ and $\tilde{\mathbf B}$ we performed computations over   $GF(2^{15})$ as explained above. The computations were done in MATLAB R2018b, where $\tilde{\mathbf A}$ and $\tilde{\mathbf B}$ were generated using the built-in function \texttt{gf} (Galois field arrays) and
	the rank of $\Phi(\tilde{\mathbf A},\tilde{\mathbf B})$ was computed   with the built-in  function \texttt{rank}.  We limited ourselves to the cases where
	$\min(I,J)\geq 2$ and  $\max(I,J)\leq 5$. Together with the assumptions  $J\geq L_{R-1}+L_R$ and $\rubinom{I}{2}\rubinom{J}{2}\geq \sum\limits_{r_1<r_2}  L_{r_1}L_{r_2}$ we ended up with $435$ tuples $(I,J,R,L_1,\dots,L_R)$. The matrix $\Phi(\tilde{\mathbf A},\tilde{\mathbf B})$ did not have full column rank
	in three cases: $(I,R)\in\{(2,3), (4,9),(5,12)\}$, $J=5$, $L_1=\dots,L_{R-1}=1$, and $L_R=4$.  
	
	To show that in the remaining $432$ cases generic uniqueness and computation follow from \cref{item:genstate2} of \cref{thm:maingen}, we need  to verify
	assumptions \eqref{item:th1gencond1},\eqref{item:th1gencond2} and condition \cref{eq:gend}. The assumption $\sum L_r=K$ in
	row $6$ of \cref{tab:KoMa14} coincides with condition \cref{eq:gend} and implies assumption \cref{item:th1gencond2}. From 
	\cref{sta:lemmaApp45} of \cref{lemma: Q2viaABC} it follows that  $[\tilde{\mathbf a}_1\otimes\tilde{\mathbf B}_1\ \dots\ \tilde{\mathbf a}_R\otimes\tilde{\mathbf B}_R]$ has full column rank, and in particular, that $IJ\geq \sum L_r$. Hence, since $\sum L_r=K$, we obtain
	that assumption  \eqref{item:th1gencond1} also holds.

\section{Proofs of Theorems \ref{thm:necessity}, \ref{thm: maintheoremABC}, \cref{corollary:verynew,thm:maingen}}\label{sec:simplethms}
\begin{proof}[Proof of \cref{thm:necessity}]
	{Proof of \cref{item:necc2}.} Assume to the contrary that the matrix $[\operatorname{vec}(\mathbf E_1)\ \dots\ \operatorname{vec}(\mathbf E_R)]$
	does not have full column rank. Then the matrices $\mathbf E_1,\dots,\mathbf E_R$ are linearly dependent. We assume w.l.o.g.  that $\mathbf E_1=\alpha_2\mathbf E_2+\dots+\alpha_R\mathbf E_R$. Then $\mathcal T$ admits a decomposition into a sum of $R-1$ terms:
	$$
	\mathcal T=\sum\limits_{r=1}^R\mathbf a_r\circ\mathbf E_r= \mathbf a_1\circ(\sum\limits_{r=2}^R\alpha_r\mathbf E_r)+\sum\limits_{r=2}^R\mathbf a_r\circ\mathbf E_r=\sum\limits_{r=2}^R(\alpha_r\mathbf a_1+\mathbf a_r)\circ\mathbf E_r,
	$$
	which is a contradiction.
	
	{Proof of \cref{item:necc4}.} Assume to the contrary that the matrix $[\mathbf a_1\otimes \mathbf B_1\ \dots\ \mathbf a_R\otimes \mathbf B_R]$ does not have full column rank. Then there exists $\mathbf f=[\mathbf f_1^T\ \dots\ \mathbf f_R^T]^T\in\fF^{\sum L_r}\setminus\{\mathbf 0\}$ such that $\sum (\mathbf a_r\otimes\mathbf B_r)\mathbf f_r=\mathbf 0$. We assume  w.l.o.g. that the first entry of $\mathbf f$ is nonzero and partition $\mathbf f_1$, $\mathbf B_1$, and $\mathbf C_1$ as
	$$
	\mathbf f=\begin{bmatrix}
	f_1\\ \bar{\mathbf f}_1
	\end{bmatrix},\qquad \mathbf B_1 = [\mathbf b_1\ \bar{\mathbf B}_1],\qquad \mathbf C_1 = [\mathbf c_1\ \bar{\mathbf C}_1].
	$$
	Since $\sum (\mathbf a_r\otimes\mathbf B_r)\mathbf f_r=\mathbf 0$, it follows that
	\begin{multline}
	\mathbf a_1\otimes \mathbf b_1=-\frac{1}{f_1}\left[ (\mathbf a_1\otimes \bar{\mathbf B}_1)\bar{\mathbf f}_1 +\sum\limits_{r=2}^R
	(\mathbf a_r\otimes\mathbf B_r)\mathbf f_r\right]=\\
	-\frac{1}{f_1}\left[ \mathbf a_1\otimes (\bar{\mathbf B}_1\bar{\mathbf f}_1) +\sum\limits_{r=2}^R
	\mathbf a_r\otimes(\mathbf B_r\mathbf f_r)\right].\label{eq:a1kronb1}
	\end{multline}
	Hence, by \cref{eq:unf_T_3,eq:a1kronb1},
	\begin{multline*}
	\unf{T}{3} = \sum\limits_{r=1}^R (\mathbf a_r\otimes\mathbf B_r)\mathbf C_r^T=(\mathbf a_1\otimes\mathbf b_1)\mathbf c_1^T+
	(\mathbf a_1\otimes\bar{\mathbf B}_1)\bar{\mathbf C}_1^T+\sum\limits_{r=2}^R (\mathbf a_r\otimes\mathbf B_r)\mathbf C_r^T=\\
	-\frac{1}{f_1}\left[ \mathbf a_1\otimes (\bar{\mathbf B}_1\bar{\mathbf f}_1) +\sum\limits_{r=2}^R
	\mathbf a_r\otimes(\mathbf B_r\mathbf f_r)\right]\mathbf c_1^T+(\mathbf a_1\otimes\bar{\mathbf B}_1)\bar{\mathbf C}_1^T+\sum\limits_{r=2}^R (\mathbf a_r\otimes\mathbf B_r)\mathbf C_r^T=\\
	\mathbf a_1\otimes\left[ -\frac{1}{f_1}\bar{\mathbf B}_1\bar{\mathbf f}_1\mathbf c_1^T + \bar{\mathbf B}_1\bar{\mathbf C}_1^T
	\right] +\sum\limits_{r=2}^R  \mathbf a_r\otimes\left[   -\frac{1}{f_1}{\mathbf B}_r{\mathbf f}_r\mathbf c_1^T+   \mathbf B_r\mathbf C_r^T\right]=: \sum\limits_{r=1}^R\mathbf a_{\tcr{r}}\otimes \tilde{\mathbf E}_r,
	\end{multline*}
	where $r_{\tilde{\mathbf E}_1}\leq r_{\bar{\mathbf B}_1}=L_1-1$ and 
	$r_{\tilde{\mathbf E}_r}\leq r_{\mathbf B_r}=L_r$ for $r\geq 2$.
	Thus, $\mathcal T$ admits an alternative decomposition into a sum of \MLatmost terms $\mathcal T=\sum \mathbf a_r\circ \tilde{\mathbf E}_r$
	with \tcr{$r_{\tilde{\mathbf E}_1}< r_{{\mathbf E}_1}$ and $r_{\tilde{\mathbf E}_r}\leq r_{{\mathbf E}_r}$  for $r\geq 2$.} This contradiction completes the proof.
	
	{Proof of \cref{item:necc3}.} The proof is similar to the proof of \cref{item:necc4} .
\end{proof}
\begin{proof}[Proof of \cref{thm: maintheoremABC}]
	By \cref{eq:unf_T_3},  assumption \cref{item:th1cond1ABC} is equivalent to assumption
	\cref{item:th1cond1}. Substituting  $\mathbf E_r=\mathbf B_r\mathbf C_r^T$ in the expressions for $\mathbf Z_r$, $\mathbf F$, $\mathbf G$,  and $[\mathbf E_1^T\ \dots\ \mathbf E_R^T]^T$,  we obtain that
	\begin{align*}
	&\mathbf Z_r = \Bdiag(\mathbf B_1,\dots,\mathbf B_{r-1},\mathbf B_{r+1},\dots,\mathbf B_R)[\mathbf C_1\ \dots\ \mathbf C_{r-1}\ \mathbf C_{r+1}\ \dots\  &\mathbf C_R]^T,\\
	&\mathbf F = [\mathbf B_{r_1}\ \mathbf B_{r_2}\ \dots\ \mathbf B_{r_{R-r_{\mathbf A}+2}}]\Bdiag(\mathbf C_{r_1}^T,\mathbf C_{r_2}^T,\dots,\mathbf C_{r_{R-r_{\mathbf A}+2}}^T),\\
	&\mathbf G = [\mathbf C_{r_1}\ \mathbf C_{r_2}\ \dots\ \mathbf C_{r_{R-r_{\mathbf A}+2}}]\Bdiag(\mathbf B_{r_1}^T,\mathbf B_{r_2}^T,\dots,\mathbf B_{r_{R-r_{\mathbf A}+2}}^T),\\
	&[\mathbf E_1^T\ \dots\ \mathbf E_R^T]^T =\Bdiag(\mathbf B_1,\dots,\mathbf B_R)\mathbf C^T.
	\end{align*}
	Since the matrices $\mathbf B_r$ and $\mathbf C_r$ have full column rank, it follows that
	\begin{align}
	d_r=\dim\nullsp{\mathbf Z_r} &=\dim\nullsp{[\mathbf C_1\ \dots\ \mathbf C_{r-1}\ \mathbf C_{r+1}\ \dots\ \mathbf C_R]^T}=\dim\nullsp{\mathbf Z_{r,\mathbf C}},\label{eq:drviaC}
	\end{align}
	that \cref{eq:ranksofFand G,eq:ranksofFand G2} are equivalent to  \cref{eq:ranksofFand GABC} \tcr{and  $ k_{\mathbf C}'\geq R-r_{\mathbf A}+2$, respectively},
	and that \cref{item:th1cond4} in \cref{thm: maintheorem} is equivalent to  $r_{\mathbf C^T}=\sum L_r$. Since, by \cref{item:th1cond1} and \cref{eq:unf_T_3}, $K=r_{\unf{T}{3}}\leq r_{\mathbf C^T}\leq K$, it follows that $r_{\mathbf C}=r_{\mathbf C^T}=K=\sum L_r$. Hence $\mathbf  C$ is a nonsingular $K\times K$ matrix. This in turn, by \cref{eq:drviaC}, implies that $d_r=L_r$. Thus,
	\cref{item:th1cond4} in \cref{thm: maintheorem} is equivalent to \cref{item:th1cond4ABC} in \cref{thm: maintheoremABC}.
\end{proof}
\begin{proof}[Proof of \cref{corollary:verynew}]
	We consider two cases $r_{\mathbf C}=K$ and $r_{\mathbf C}<K$.
	
	i) Let $r_{\mathbf C}=K$.
		Together the assumptions in \cref{eq:corollary261} and conditions in \cref{eq:corollary262} imply  that  assumption \cref{eq:ranksofFand GABC} and \cref{item:th1cond3a} in \cref{thm: maintheoremABC} hold.
		In turn, \cref{item:th1cond3a} implies that assumption \cref{item:th1cond2ABC} holds.
		The two conditions in \cref{eq:corollary262} coincide with \cref{item:th1cond5} and \cref{item:th1cond6} in \cref{thm: maintheoremABC}, respectively.
		Thus, to apply \cref{item:th1stat4} in \cref{thm: maintheoremABC} it only remains to verify that 
		assumption \cref{item:th1cond1ABC} holds. Since $r_{\mathbf C}=K$, it is sufficient to prove that the matrix $[\mathbf a_1\otimes\mathbf B_1\ \dots\ \mathbf a_R\otimes\mathbf B_R]$ has full column rank. This  follows from \cref{sta:lemmaApp4,sta:lemmaApp45} of  \cref{lemma: Q2viaABC}.

\tcr{ii) If $r_{\mathbf C}<K$, then the result follows from i) and \cref{item:thm:sonvention:statement1} of \cref{thm:convention}.}
\end{proof}

\begin{proof}[Proof of   \cref{thm:maingen}]
%
%
%
%
	We show that \cref{item:genstate1,item:genstate2ad,item:th1firstfmhardgen,item:genstate2} in \cref{thm:maingen} correspond, respectively, to statements \ref{item:th1stat1}, \ref{item:th1newstat}, \ref{item:th1firstfmhard},   and \ref{item:th1stat4} in \cref{thm: maintheorem}.	
	\tcr{One can easily check that   assumptions \cref{item:th1gencond1}, \cref{item:th1gencond2}, 
		and conditions 	\cref{eq:genb},  \cref{eq:gend} in \cref{thm:maingen}   are, respectively, the generic versions of assumptions
		 \cref{item:th1cond1}, \cref{item:th1cond2} and conditions \ref{item:th1cond5}, \ref{item:th1cond4} in \cref{thm: maintheorem}.}
	%
	Hence,  to prove   statements \ref{item:genstate1}, \ref{item:genstate2ad},   and \ref{item:genstate2},   it is sufficient to show that  assumption \cref{eqLgenericdimQ_2} implies that \cref{item:th1cond3} holds generically. To prove  \cref{item:th1firstfmhardgen} we should  additionally show that \cref{eq:ineqforuni1fmgen} implies that \cref{eq:ineqforuni1fm} holds generically.
	
	1) We show that  assumption \cref{eqLgenericdimQ_2} implies that \cref{item:th1cond3} holds generically.	We will make use of \cite[Lemma 6.3]{PartII} which states the following: if  the entries of a matrix $\mathbf F(\mathbf x)$ depend analytically on $\mathbf x\in\fF^n$ and if $\mathbf F(\mathbf x_0)$ has full column rank
	for at least one $\mathbf x_0$, then $\mathbf F(\mathbf x)$ has  full column rank for generic $\mathbf x$.   
	Let the vectors $\mathbf x$ and $\mathbf x_0$ be formed by the entries of $\mathbf A$, $\mathbf B$,  $\mathbf C$ and $\tilde{\mathbf A}$, $\tilde{\mathbf B}$, and $\tilde{\mathbf C}$ respectively. We construct $\mathbf F(\mathbf x)$ as follows.
	By \cref{lemma: Q2viaABC},  each entry of $\mathbf Q_2(\mathcal T)$ is a polynomial in $\mathbf x$. By the rank-nullity theorem and  assumption \cref{eqLgenericdimQ_2},
	\begin{equation}
	r_{\mathbf Q_2(\tilde{\mathcal T})}=\rubinom{K+1}{2} - \sum\limits_{r=1}^R \rubinom{K-(L_1+\dots+L_{r-1}+L_{r+1}+\dots+L_R)+1}{2}=:P,
	\label{eq:ranknullity}
	\end{equation}
	implying that  $P$ columns of  $\mathbf Q_2(\tilde{\mathcal T})$ are linearly independent. We define
	$\mathbf F(\mathbf x)$  as the submatrix  formed by the corresponding columns\footnote{The column selection depends
		only on the fixed $\mathbf x_0$.} of $\mathbf Q_2(\mathcal T)$. Then \cref{eq:ranknullity} implies that $\mathbf F(\mathbf x_0)$ has full column rank.
	Now, by \cite[Lemma 6.3]{PartII}, $\mathbf F(\mathbf x)$ has full column rank for generic $\mathbf x$. Hence  $r_{\mathbf Q_2(\mathcal T)}\tcr{\geq} P$. \tcr{ Hence, by the rank-nullity theorem, 
		$\dim\nullsp{\mathbf Q_2({\mathcal T})}=\rubinom{K+1}{2}- 	r_{\mathbf Q_2({\mathcal T})}  \leq \rubinom{K+1}{2}-P= \sum\limits_{r=1}^R \rubinom{d_r+1}{2}$.
		On the other hand, since, by \cref{sta:lemmaApp3} of 	\cref{lemma: Q2viaABC}, 
		$\dim\nullsp{\mathbf Q_2({\mathcal T})}\geq \sum\limits_{r=1}^R \rubinom{d_r+1}{2}$
	we obtain	that	\cref{item:th1cond3} in \cref{thm: maintheorem} holds.} 
	
	2)  We show that assumption \cref{eq:ineqforuni1fmgen} implies that \cref{eq:ineqforuni1fm} holds generically. Let $S=\sum L_r$. Then $d_r = K-\sum\limits_{k=1}^RL_k+L_r=K-S+L_r$. Since $L_1\leq\dots\leq L_R$, the inequality in \cref{eq:ineqforuni1fm} takes the form
	\begin{equation}
	\rubinom{K+1}{2} - \sum\limits_{r=1}^R \rubinom{K-S+L_r\tcr{+1}}{2} > \sum\limits_{1\leq r_1<r_2\leq R} L_{r_1}L_{r_2} - L_1L_2=\frac{S^2-\sum L_r^2}{2}-L_1L_2.\label{eq:ineqforuni1fmgengen}
	\end{equation}	
	Using simple algebraic manipulations one can rewrite  \cref{eq:ineqforuni1fmgengen} as
	\begin{equation}
	K^2 + K(1-2S)+S^2-S-\frac{2L_1L_2}{R-1}<0.\label{eq:quadraticequation}
	\end{equation}
	One can easily check that $K$ is a solution of  \cref{eq:quadraticequation} if and only if
	$$
	S-\frac{1}{2} -\sqrt{\frac{1}{4}+\frac{2L_1L_2}{R-1}} < K < S-\frac{1}{2} +\sqrt{\frac{1}{4}+\frac{2L_1L_2}{R-1}},
	$$
	implying that \cref{eq:ineqforuni1fmgen} is a generic version of \cref{eq:ineqforuni1fm}.
\end{proof}
\section{Proof of \cref{thm:maingenLL1}}\label{sec:AppendixC}
First we recall a result on the generic uniqueness of the
decomposition of a matrix into rank-$1$ terms that admit a particular structure \cite{JSTSP2016IDLDL}. 
Let $p_1,\dots,p_N$ be known polynomials in $l$ variables and let $\mathbf Y\in\fF^{I\times N}$ admit  a decomposition  of the form
\begin{equation}
\mathbf Y=\sum\limits_{r=1}^R\mathbf a_r [p_1(\mathbf z_r)\ \dots\  p_N(\mathbf z_r)],\qquad 
\mathbf a_r\in\fF^I,\quad \mathbf z_r\in\fF^l,\quad r=1,\dots,R.
\label{eq:structuredrank1}
\end{equation}
Decomposition \cref{eq:structuredrank1} can  be interpreted as a matrix factorization $\mathbf Y=\mathbf A\mathbf P^T$ that is structured in the sense that 
the  columns of $\mathbf P$ are  in
\begin{equation}
V: = \{[p_1(\mathbf z)\ \dots\  p_N(\mathbf z)]^N:\ \mathbf z\in\fF^l\}\subset \fF^N.\label{eq:setV}
\end{equation}
We say that the  decomposition   {\em  is unique} if   any two decompositions of the form \cref{eq:structuredrank1}
are the same up to permutation of summands. We say that the decomposition into a sum of structured rank-$1$ matrices  
is {\em generically unique} if 
$$
\mu\{(\mathbf a_1,\dots,\mathbf a_R,\mathbf z_1,\dots,\mathbf z_R):\ \text{decomposition \cref{eq:structuredrank1} is not unique}\}=0,
$$
where $\mu$ denotes a measure on $\fF^{(I+l)R}$ that is  absolutely continuous with respect to the Lebesgue measure.
We will need the following result.
\begin{theorem}(a corollary of \cite[Theorem 1]{JSTSP2016IDLDL})\label{thm:auxgen}
	Assume that
	\begin{assumptions}
		\item\label{appauxgenthm1} $R\leq I$;
		\item\label{appauxgenthm2} $\dim\sspan\{V\}\geq \hat{N}$; 
		\item\label{appauxgenthm3} the set $V$ is invariant under complex scaling, i.e., $\lambda V=V$ for all $\lambda\in C$;
		\item\label{appauxgenthm4} the dimension of the Zariski closure of $V$ is less than or equal to $\hat l$;
		\item\label{appauxgenthm5} $R\leq \hat{N}-\hat{l}$.
	\end{assumptions}
	Then decomposition \cref{eq:structuredrank1} is  generically unique.
\end{theorem}
\begin{proof}[Proof of \cref{thm:maingenLL1}] (i) First we rewrite   \cref{eq:LrLr1mainBC} in the form of the structured matrix decomposition 
	\cref{eq:structuredrank1}.
	In step (ii) we will apply \cref{thm:auxgen} to \cref{eq:structuredrank1}.
	By \cref{eq:unf_T_1}, decomposition \cref{eq:LrLr1mainBC} can be rewritten as 
	$$
	\mathbf Y:=\unf{T}{1}^T=\mathbf A[\operatorname{vec}(\mathbf B_1\mathbf C_1^T)\ \dots\ \operatorname{vec}(\mathbf B_R\mathbf C_R^T)]^T=:\mathbf A\mathbf P^T.
	$$
	So, the columns of $\mathbf P$ are of the form 
	$$
	\operatorname{vec}([\mathbf b_1\ \dots\ \mathbf b_L][\mathbf c_1\ \dots\ \mathbf c_L]^T) = \mathbf c_1\otimes \mathbf b_1+\dots+\mathbf c_L\otimes\mathbf b_L=:
	[p_1(\mathbf z)\ \dots\ p_N(\mathbf z)]^T,
	$$ 
	where
	$$
	\mathbf z=[\mathbf b_1^T\ \dots\ \mathbf b_L^T\ \mathbf c_1^T\ \dots\ \mathbf c_L^T]^T,\quad l=JL+KL,\quad N=JK.
	$$ 
	Hence the set $V$	in \cref{eq:setV} consists of vectorized $J\times K$ matrices whose rank does not exceed $L$.
	
	(ii) Now we check  \cref{appauxgenthm1,appauxgenthm2,appauxgenthm3,appauxgenthm4,appauxgenthm5} in \cref{thm:auxgen}.   \Cref{appauxgenthm1} holds by \cref{eq:RleqI}. Since $V$ contains, in particular,  all vectorized rank-$1$ matrices, it  spans the entire $\fF^N$. Hence we can choose $\hat N=N=JK$ in   \cref{appauxgenthm2}.  \Cref{appauxgenthm3} is trivial.
	It is well-known that the set $V$ is an algebraic variety of dimension $(J+K-L)L$, so    \cref{appauxgenthm4} holds for $\hat l=(J+K-L)L$. Finally,     \cref{appauxgenthm5} holds by \cref{eq:RleqI}:
	$
	R\leq (J-L)(K-L)=JK-(J+K-L)L=\hat N - \hat l
	$.
\end{proof}
\section{Proofs of  \cref{sta:lemmaApp1,sta:lemmaApp2,sta:lemmaApp5} of  \cref{lemma: Q2viaABC} and proof of \cref{Lemma:redtobtdI}}\label{sec:appendixB}
\begin{proof}[Proofs of  \cref{sta:lemmaApp1,sta:lemmaApp2,sta:lemmaApp5} of \cref{lemma: Q2viaABC}]
	1) Since $\mathcal T = \sum\limits_{r=1}^R\mathbf a_r\circ(\mathbf B_r\mathbf C_r^T)$, it follows that
	$
	t_{ijk} = \sum\limits_{r=1}^R a_{ir} \sum\limits_{l=1}^{L_r}b_{jl,r}c_{kl,r}
	$. 
	Hence
	\begin{equation}
	t_{i_1j_1k_1}t_{i_2j_2k_2} = \sum\limits_{r_1=1}^R \sum\limits_{r_2=1}^R a_{i_1r_1}a_{i_2r_2}
	\sum\limits_{l_1=1}^{L_{r_1}} \sum\limits_{l_2=1}^{L_{r_2}} b_{j_1l_1,r_1} b_{j_2l_2,r_2} c_{k_1l_1,r_1} c_{k_2l_2,r_2}.
	\label{eq:tt}
	\end{equation}
	By \cref{def:Q2}, the entry of $\mathbf Q_2(\mathcal T)$ with the index in \cref{eq:ttttindex}
	is equal to \cref{eq: tttt}, where  $1\leq i_1<i_2\leq I$, $1\leq j_1<j_2\leq J$, and $1\leq k_1\leq k_2\leq K$.
	Applying \cref{eq:tt} to each term in \cref{eq: tttt} and making simple algebraic manipulations we obtain that the expression in 
	\cref{eq: tttt} is equal to
	\begin{align*}
	&\sum\limits_{1\leq r_1<r_1\leq R} \Big[ (a_{i_1r_1}a_{i_2r_2} - a_{i_2r_1}a_{i_1r_2})\times \\
	&\qquad\qquad\sum\limits_{l_1=1}^{L_{r_1}} \sum\limits_{l_2=1}^{L_{r_2}} (b_{j_1l_1,r_1} b_{j_2l_2,r_2}-b_{j_2l_1,r_1} b_{j_1l_2,r_2}) (c_{k_1l_1,r_1} c_{k_2l_2,r_2}+c_{k_2l_1,r_1} c_{k_1l_2,r_2})\Big]=\\
	&\sum\limits_{1\leq r_1<r_1\leq R} \left(\wprod{\mathbf a_{r_1}}{\mathbf a_{r_2}}\right)_{i_1+\rubinom{i_2-1}{2}}
	\sum\limits_{l_1=1}^{L_{r_1}} \sum\limits_{l_2=1}^{L_{r_2}}
	\left(\wprod{\mathbf b_{l_1,r_1}}{\mathbf b_{l_2,r_2}}\right)_{j_1+\rubinom{j_2-1}{2}}
	\left(\symprod{\mathbf c_{l_1,r_1}}{\mathbf c_{l_2,r_2}}\right)_{k_1+\rubinom{k_2}{2}},
	\end{align*}
	which, by the definition of $\Phi(\mathbf A,\mathbf B)$ and $\mathbf S_2(\mathbf C)$, is the  entry of  $\Phi(\mathbf A,\mathbf B)\mathbf S_2(\mathbf C)^T$ with the index in \cref{eq:ttttindex}.\par
		2) follows from the identity $\mathbf R_2(\mathcal T)=\mathbf Q_2(\mathcal T)\mathbf P_K^T$ and 1).\par
       6) We assume that $\Phi(\mathbf A,\mathbf B)$ has full column rank. It is sufficient to prove that the identities  $\mathbf h=\mathbf B_{r_1}\mathbf f_1=\mathbf B_{r_1}\mathbf f_2$ are valid only for $\mathbf h=\mathbf 0$.  From the definition of the operation ``$\wprod{}{}$'' it follows that     $(\wprod{\mathbf B_{r_1}}{\mathbf B_{r_2}})(\mathbf f_1\otimes\mathbf f_2) = \wprod{(\mathbf B_{r_1}\mathbf f_1)}{(\mathbf B_{r_2}\mathbf f_2)}=
    \wprod{\mathbf h}{\mathbf h}=\mathbf 0$. 
    Hence $\left[(\wprod{\mathbf a_{r_1}}{\mathbf a_{r_2}})\otimes(\wprod{\mathbf B_{r_1}}{\mathbf B_{r_2}})\right](\mathbf f_1\otimes\mathbf f_2)=
    (\wprod{\mathbf a_{r_1}}{\mathbf a_{r_2}})\otimes 
    \left[(\wprod{\mathbf B_{r_1}}{\mathbf B_{r_2}})(\mathbf f_1\otimes\mathbf f_2)\right]=\mathbf 0$.
    Now, since $(\wprod{\mathbf a_{r_1}}{\mathbf a_{r_2}})\otimes(\wprod{\mathbf B_{r_1}}{\mathbf B_{r_2}})$
    is formed by the columns of the full column rank matrix      $\Phi(\mathbf A,\mathbf B)$, it follows that
    $\mathbf f_1\otimes\mathbf f_2=\mathbf 0$, which easily implies that $\mathbf h=\mathbf 0$.
   \end{proof}
\begin{proof}[Proof of \cref{Lemma:redtobtdI}]
	W.l.o.g. we assume that $i=1$ and  $j=2$. Since $\mathbf C$ has full column rank, and, by \cref{item:th1cond1ABC}, $\mathbf C^T$ has full column rank, it follows that  $\mathbf C$ is $K\times K$ nonsingular and that $K=\sum L_r$. 
	This readily implies that $d_r=L_r$ for all $r$.
	From the rank-nullity theorem and \cref{item:th1cond3ABC} it follows that
	\begin{multline*}
	r_{\Phi(\mathbf A,\mathbf B)}\geq r_{\Phi(\mathbf A,\mathbf B)\mathbf S_2(\mathbf C)^T} = 
	\rubinom{K+1}{2}
	-\dim\nullsp{\Phi(\mathbf A,\mathbf B)\mathbf S_2(\mathbf C)^T}=\\
	\rubinom{\sum L_r+1}{2}-\sum \rubinom{L_r+1}{2}=\sum\limits_{r_1<r_2} L_{r_1}L_{r_2}.
	\end{multline*}
	Since $\Phi(\mathbf A,\mathbf B)$ is a $\rubinom{K+1}{2} \times \sum\limits_{r_1<r_2} L_{r_1}L_{r_2}$ matrix, it follows that
	$\Phi(\mathbf A,\mathbf B)$  has full column rank. In particular, the submatrix $(\wprod{\mathbf a_1}{\mathbf a_2})\otimes(\wprod{\mathbf B_1}{\mathbf B_2})$ has full column rank, implying  that the same holds true for the matrix $\wprod{\mathbf B_1}{\mathbf B_2}$.
	Assume that $[\mathbf B_1\ \mathbf B_2][\mathbf f_1^T\ \mathbf f_2^T]^T=\mathbf 0$ for some $\mathbf f_1\in\fF^{L_1}$ and
	$\mathbf f_2\in\fF^{L_2}$. Then $\mathbf B_2\mathbf f_2 = -\mathbf B_1\mathbf f_1$.  One can  easily verify that $(\wprod{\mathbf B_1}{\mathbf B_2})(\mathbf f_1\otimes \mathbf f_2)=
	\wprod{\mathbf B_1\mathbf f_1}{\mathbf B_2\mathbf f_2}=-\wprod{\mathbf B_1\mathbf f_1}{\mathbf B_1\mathbf f_1}=\mathbf 0$.
	Hence $\mathbf f_1\otimes \mathbf f_2=\mathbf 0$. Thus,  $\mathbf f_1=\mathbf 0$ or $\mathbf f_2=\mathbf 0$, implying that
	$\mathbf B_1\mathbf f_1=\mathbf 0$ or $\mathbf B_2\mathbf f_2=\mathbf 0$.
	Since $\mathbf B_1$ and $\mathbf B_2$ have full column rank  and $\mathbf B_2\mathbf f_2 = -\mathbf B_1\mathbf f_1$,
	it follows that both $\mathbf f_1$ and $\mathbf f_2$ are the zero vectors. Hence  the matrix $[\mathbf B_1\ \mathbf B_2]$ has full column rank.
\end{proof}
\section{Proof of  \cref{sta:lemmaApp3} of  \cref{lemma: Q2viaABC}}\label{sec:appendixLemmas3}
\begin{proof}[Proofs of  \cref{sta:lemmaApp3} of  \cref{lemma: Q2viaABC}] 
	The inequality in \cref{sta:lemmaApp3}
	%
	 follows  immediately from 	
	\cref{sta:lemmaApp1}. We prove the identity $\dim\nullsp{\mathbf S_2(\mathbf C)^T}=\sum\rubinom{d_r+1}{2}$. 	
	Throughout the proof,  $\scol(\cdot)$ denotes the column space of a matrix.
	
	Obviously, $\dim\nullsp{\mathbf S_2(\mathbf C)^T}=\dim\nullsp{\mathbf S_2(\mathbf C)^H}$.
	Since $\vecsym{K}$ is the orthogonal sum of the subspaces $\nullsp{\mathbf S_2(\mathbf C)^H}$ and $\scol(\mathbf S_2(\mathbf C))$,
	it is sufficient to show that there exists a subspace  $S$ such that
	\begin{gather}
	\label{eq:subspaceS1}
	\vecsym{K}=\sspan\{S,\scol(\mathbf S_2(\mathbf C))\},\\
	S\cap\scol(\mathbf S_2(\mathbf C))  = \{\mathbf 0\},\label{eq:subspaceS2}\\ 
	\dim S =\sum\rubinom{d_r+1}{2}. \label{eq:subspaceS3}
	\end{gather}
	We explicitly construct a possible $S$ and show that  \cref{eq:subspaceS1,eq:subspaceS2,eq:subspaceS3} hold.
	
	(i) {\em Construction of $S$.}
	Since $r_{\mathbf C}=K$ and $\dim \nullsp{\mathbf Z_{r,\mathbf C}}=d_r$, it follows that
	$r_{\mathbf Z_{r,\mathbf C}^T}=r_{\mathbf Z_{r,\mathbf C}}=K-d_r$. 	Let $W_r=\scol(\mathbf Z_{r,\mathbf C}^T) \cap \scol(\mathbf C_r)$ and let $V_r$ denote the orthogonal complement of $W_r$ in $\scol(\mathbf C_r)$. Then
	\begin{equation*}
	\begin{split}
	\dim W_r =& \dim \scol(\mathbf Z_{r,\mathbf C}^T) + \dim\scol(\mathbf C_r)\\-&
	\dim \scol([\mathbf C_1\ \dots \mathbf C_{r-1}\ \mathbf C_{r+1}\ \dots\ \mathbf C_R\ \mathbf C_r])
	=K-d_r+L_r-K=L_r-d_r,\\
	\dim V_r =&  \dim\scol(\mathbf C_r) - \dim W_r = L_r-(L_r-d_r) = d_r.
	\end{split}
	\end{equation*}
	Let  $\mathbf V_r\in\fF^{K\times d_r}$ be  a matrix whose columns form a basis of $V_r$.
	We set
	$$
	S = \scol([\mathbf V_1\cdot\mathbf V_1\ \dots\ \mathbf V_R\cdot\mathbf V_R]).
	$$
	
	(ii) {\em Proof of \cref{eq:subspaceS1}}.  Let  $\mathbf W_r\in\fF^{K\times (L_r-d_r)}$ be  a matrix whose columns form a basis of $W_r$. 
	Since $r_{\mathbf C}=K$ and $\scol(\mathbf C_r)=\scol([\mathbf V_r\ \mathbf W_r])$,
	it follows that
	\begin{equation}
	\begin{split}
	&\vecsym{K}
	=\scol([\symprod{\mathbf C}{\mathbf C}])=
	\sspan\{\scol(\symprod{\mathbf C_{r_1}}{\mathbf C_{r_2}}):1\leq r_1,r_2\leq R\}\\
	=&	\sspan\{\scol(\mathbf S_2(\mathbf C)),\scol(\symprod{\mathbf C_{r}}{\mathbf C_{r}}):1\leq r\leq R\}\\
	=&\sspan\{\scol(\mathbf S_2(\mathbf C)),\scol(\symprod{\mathbf V_{r}}{\mathbf V_{r}}),\scol(\symprod{\mathbf V_{r}}{\mathbf W_{r}}),
		\scol(\symprod{\mathbf W_{r}}{\mathbf W_{r}}):1\leq r\leq R\}\\
		=&\sspan\{\scol(\mathbf S_2(\mathbf C)),S,\scol(\symprod{\mathbf V_{r}}{\mathbf W_{r}}),
		\scol(\symprod{\mathbf W_{r}}{\mathbf W_{r}}):1\leq r\leq R\}.
	\end{split}\label{eq:splittedeq1}
	\end{equation}
	From the construction of $\mathbf W_r$, $\mathbf V_r$  and $\mathbf S_2(\mathbf C)$ it follows that
	\begin{equation}
	 \sspan\{ \scol(\symprod{\mathbf V_{r}}{\mathbf W_{r}}),
	 \scol(\symprod{\mathbf W_{r}}{\mathbf W_{r}})\}\subseteq 
	 \scol(\symprod{\mathbf C_{r}}{\mathbf Z_{r,\mathbf C}^T})\subseteq
	 \scol(\mathbf S_2(\mathbf C)),\quad 	 
	 1\leq r\leq R.\label{eq:splittedeq2}
	\end{equation}
	Now,  \cref{eq:subspaceS1} follows from \cref{eq:splittedeq1,eq:splittedeq2}.
		
	(iii) {\em Proof of \cref{eq:subspaceS2}}. From the construction of $V_r$ it follows that
	\begin{equation}
	\label{eq:orthogonal}
	\scol(\mathbf V_r) \text{ is orthogonal to  } \scol(\mathbf C_1),\dots,\scol(\mathbf C_{r-1}),\scol(\mathbf C_{r+1}),\dots,\scol(\mathbf C_R).
	\end{equation}
	Let $\mathbf P_K$ be  defined as in \cref{eq:defPn}.
	Then
	\begin{gather}
	\scol(\mathbf P_K(\symprod{\mathbf V_r}{\mathbf V_r}))=\sspan\{\mathbf x_r\otimes\mathbf y_r+\mathbf y_r\otimes\mathbf x_r:\ \mathbf x_r,\mathbf y_r\in V_r\},\label{eq:new82}\\
	\scol(\mathbf P_K(\symprod{\mathbf C_{r_1}}{\mathbf C_{r_2}}))=\sspan\{\mathbf x_{r_1}\otimes\mathbf y_{r_2}+\mathbf y_{r_2}\otimes\mathbf x_{r_1}:\ \mathbf x_{r_1}\in \scol(\mathbf C_{r_1}),\mathbf y_{r_2}\in \scol(\mathbf C_{r_2})\}.\nonumber
	\end{gather}
	It now easily follows from \cref{eq:orthogonal} that
	$$
	\scol(\mathbf P_K(\symprod{\mathbf V_r}{\mathbf V_r})) \text{ is orthogonal to }
	\scol(\mathbf P_K(\symprod{\mathbf C_{r_1}}{\mathbf C_{r_2}})),\ 1\leq r\leq R,\ 1\leq r_1<r_2\leq R.
	$$
	Hence $\mathbf P_K S$ is orthogonal to $\mathbf P_K\scol(\mathbf S_2(\mathbf C)) $.
	Since $ \mathbf P_K$ is a bijective linear map from $\fF^{\rubinom{K+1}{2}}$ to $\vecsym{K}$, it follows that
	the subspaces 	$S$ and $\scol(\mathbf S_2(\mathbf C))$ are linearly independent, that is, \cref{eq:subspaceS2} holds.
	
	(iii) {\em Proof of \cref{eq:subspaceS3}}. 	Since $ \mathbf P_K$ is a bijective linear map, it is sufficient to prove that 
	$\dim \mathbf P_K S=\sum\rubinom{d_r+1}{2}$. 
	 From the construction of $V_r$ it follows that  $\scol(\mathbf V_{r_1})$ is orthogonal to $\scol(\mathbf V_{r_2})$ for $r_1\ne r_2$.
	 Hence, by \eqref{eq:new82},   $\scol(\mathbf P_K(\symprod{\mathbf V_{r_1}}{\mathbf V_{r_1}}))$ is orthogonal to 
	 $\scol(\mathbf P_K(\symprod{\mathbf V_{r_2}}{\mathbf V_{r_2}}))$ for $r_1\ne r_2$. 
	 Since $\mathbf P_K S= \sspan\{ \scol(\mathbf P_K(\symprod{\mathbf V_r}{\mathbf V_r})):\ 1\leq r\leq R\}$, it follows that
	 $\mathbf P_K S$   is the orthogonal sum of the subspaces 	 $\scol(\mathbf P_K(\symprod{\mathbf V_r}{\mathbf V_r}))$. Hence
	 	$\dim \mathbf P_K S=\sum \dim \scol(\mathbf P_K(\symprod{\mathbf V_r}{\mathbf V_r}))$. 	 
	   To prove that	 $\dim\scol(\mathbf P_K(\symprod{\mathbf V_r}{\mathbf V_r}))= \rubinom{d_r+1}{2}$
	   we show that the $\rubinom{d_r+1}{2}$  columns
	   $\mathbf v_i\otimes \mathbf v_j+\mathbf v_j\otimes \mathbf v_i$, $1\leq i\leq j\leq d_r$ of 
	   $\mathbf P_K(\symprod{\mathbf V_r}{\mathbf V_r})$ are linearly independent, where
	   $\mathbf v_1,\dots,\mathbf v_{d_r}$ denote the columns of $\mathbf V_r$. Indeed, assume that there exist values $\lambda_{ij}$ , $1\leq i\leq j\leq d_r$
	   such that $\mathbf 0=\sum\limits_{1\leq i\leq j\leq d_r}\lambda_{ij}( \mathbf v_i\otimes \mathbf v_j+\mathbf v_j\otimes \mathbf v_i)$.
	   Then
	   \begin{equation}
	   \begin{split}
	   \mathbf 0&=
	   \sum\limits_{1\leq i\leq d_r} \mathbf v_i\otimes \sum\limits_{i\leq j\leq d_r}\lambda_{ij}\mathbf v_j+
	   \sum\limits_{1\leq j\leq d_r} \mathbf v_j\otimes \sum\limits_{1\leq i\leq j}\lambda_{ij}\mathbf v_i\\
	   &= \sum\limits_{1\leq i\leq d_r} \mathbf v_i\otimes \left( 
	   \sum\limits_{i<j\leq d_r}\lambda_{ij}\mathbf v_j+
	   \sum\limits_{1\leq j< i}\lambda_{ji}\mathbf v_j+
	   2\lambda_{ii}\mathbf v_{ii}
	    \right).
	    \end{split}\label{eq:neweq83}
	   \end{equation}
	   Since the vectors  $\mathbf v_1,\dots,\mathbf v_{d_r}$ are linearly independent, it follows from \cref{eq:neweq83} that
	   $\lambda_{ij}=0$ for all values of indices. 
\end{proof}
\section{Proof of  \cref{sta:lemmaApp4,sta:lemmaApp45} of  \cref{lemma: Q2viaABC}}\label{sec:appendixLemmas4}
By definition, set
\begin{align}
\mathcal C_2(\mathbf A) &:= [\wprod{\mathbf a_1}{\mathbf a_2}\ \dots\ \wprod{\mathbf a_{R-1}}{\mathbf a_R}]\in\mathbb F^{\rubinom{I}{2}\times \rubinom{R}{2}}\label{eq:C2A},\\
\mathcal C_2'(\mathbf B)&:= [\wprod{\mathbf B_1}{\mathbf B_2}\ \dots\  \wprod{\mathbf B_{R-1}}{\mathbf B_R}]\in\mathbb F^{\rubinom{J}{2}\times \sum\limits_{r_1<r_2} L_{r_1}L_{r_2}}.\label{eq:C2B}
\end{align}
The matrix $\mathcal C_2(\mathbf A)$ is called the second compound matrix of $\mathbf A$. We will need the following  properties of 
$\mathcal C_2(\cdot)$ and $\mathcal C_2'(\cdot)$.
\begin{lemma} \label{lemma:compound} Let  $\mathbf Y$ be a matrix such that  $\mathcal C_2(\mathbf Y)$, and $\mathcal C_2'(\mathbf Y\mathbf B)$ are defined. Then the following
	statements hold.
	\begin{statements}
	\item	\label{lemma:newcompoundidentity1}
	If $\mathbf A$ has full column rank, then $\mathcal C_2(\mathbf A)$ also has full column rank;
	\item \label{lemma:newcompoundidentity3}
	 $\mathcal C_2(\mathbf A^T)= \mathcal C_2(\mathbf A)^T$;
    \item \label{lemma:newcompoundidentity2}
     $\mathcal C_2(\mathbf Y)\mathcal C_2(\mathbf B) = \mathcal C_2(\mathbf Y\mathbf B)$ (Binet-Cauchy formula);
     \item \label{lemma:newcompoundidentity}
     $\mathcal C_2(\mathbf Y)\mathcal C_2'(\mathbf B) = \mathcal C_2'(\mathbf Y\mathbf B)$.
\end{statements}
\end{lemma}
\begin{proof}
	\Cref{lemma:newcompoundidentity1,lemma:newcompoundidentity3,lemma:newcompoundidentity2} are classical properties of the compound matrices (see, for instance, \cite[pp. 21--22]{HornJohnson}). \Cref{lemma:newcompoundidentity} follows from \cref{lemma:newcompoundidentity2}.
	Indeed, from the definition of $\mathcal C_2(\mathbf B)$ and $\mathcal C_2'(\mathbf B)$ it follows that there exists a column selection matrix $\mathbf P$ such that  $\mathcal C_2'(\mathbf B)=\mathcal C_2(\mathbf B)\mathbf P$. Moreover, for any matrix $\mathbf Y$
	such that  $\mathcal C_2(\mathbf Y)$, and $\mathcal C_2'(\mathbf Y\mathbf B)$ are defined, the identity 
	$\mathcal C_2'(\mathbf Y\mathbf B)=\mathcal C_2(\mathbf Y\mathbf B)\mathbf P$ holds with the same $\mathbf P$.
	Hence, by \cref{lemma:newcompoundidentity2}, $\mathcal C_2(\mathbf Y)\cdot \mathcal C_2'(\mathbf B)=\mathcal C_2(\mathbf Y)\cdot \mathcal C_2(\mathbf B)\mathbf P=\mathcal C_2(\mathbf Y\mathbf B)\mathbf P=\mathcal C_2'(\mathbf Y\mathbf B)$.
\end{proof}
\begin{proof}[Proof of  \cref{sta:lemmaApp4} of  \cref{lemma: Q2viaABC}]
		First we prove that condition \cref{eq:ranksofFand GABC} implies that  $\Phi(\mathbf A,\mathbf B)$ has full column rank. In the case $k_{\mathbf B}'=2$, we have $r_{\mathbf A} =R$. Hence, by \cref{lemma:newcompoundidentity1}
		of \cref{lemma:compound}  the  $\rubinom{I}{2}\times \rubinom{R}{2}$ matrix $C_2(\mathbf A)$ has full column rank. The fact that $k_{\mathbf B}'=2$  further implies that $[\mathbf B_{r_1}\ \mathbf B_{r_2}]$ has full column rank for all $r_1\leq r_2$.
		Hence, by \cref{lemma:newcompoundidentity1} of \cref{lemma:compound}, the matrix $\mathcal C_2([\mathbf B_{r_1}\ \mathbf B_{r_2}])$ also has full column rank. Since $\wprod{\mathbf B_{r_1}}{\mathbf B_{r_2}}$ is formed by columns of $\mathcal C_2([\mathbf B_{r_1}\ \mathbf B_{r_2}])$, it also has full column rank. One can easily prove that full column rank of $C_2(\mathbf A)$ and the matrices $\wprod{\mathbf B_{r_1}}{\mathbf B_{r_2}}$, $r_1\leq r_2$ implies full column rank of  $\Phi(\mathbf A,\mathbf B)$.
	
	We now consider the case $k_{\mathbf B}'>2$. 
	\begin{romannum}
		\item
		Suppose that $\Phi(\mathbf A,\mathbf B)\mathbf f=\mathbf 0$ for some $(\sum\limits_{r_1<r_2} L_{r_1}L_{r_2})\times 1$
		vector $\mathbf f$. We represent $\mathbf f$ as $\mathbf f=[\mathbf f_{1,2}^T\ \dots\ \mathbf f_{R-1,R}^T]^T$, where $\mathbf f_{r_1,r_2}\in \fF^{L_{r_1}L_{r_2}}$.
		Then $\Phi(\mathbf A,\mathbf B)\mathbf f=\mathbf 0$ is equivalent to
		\begin{equation}
		\label{eq:sumfij}
		\sum\limits_{r_1<r_2} (\mathbf a_{r_1}\wedge \mathbf a_{r_2})\otimes (\mathbf B_{r_1}\wedge \mathbf B_{r_2})\mathbf f_{r_1,r_2}=\mathbf 0.
		\end{equation}
		We can further rewrite \cref{eq:sumfij} in  matrix form as 
		\begin{equation}
		\begin{split}
		\label{eq:withoutY}
		\mathbf O &= \sum\limits_{r_1<r_2}(\mathbf B_{r_1}\wedge \mathbf B_{r_2})\mathbf f_{r_1,r_2}(\mathbf a_{r_1}\wedge \mathbf a_{r_2})^T\\
		&=
		\mathcal C_2'(\mathbf B)\Bdiag(\mathbf f_{1,2},\dots,\mathbf f_{R-1,R})\mathcal C_2(\mathbf A)^T.
		\end{split}
		\end{equation}
		\item Let us for now assume that the last $r_{\mathbf A}$ columns of $\mathbf A$ are linearly independent. We show that
		$\mathbf f_{k_{\mathbf B}'-1,k_{\mathbf B}'}=\mathbf 0$. Let us set
		$$
		s_1:=L_1+\dots+L_{k_{\mathbf B}'-2},\quad s_2:=L_{k_{\mathbf B}'-1}+L_{k_{\mathbf B}'},\quad s_3:=L_{k_{\mathbf B}'+1}+\dots+L_R.
		$$ 		
		By definition of $k_{\mathbf B}'$, the matrix $\mathbf X:=\left[\begin{matrix}\mathbf B_1&\dots&\mathbf B_{k_{\mathbf B}}\end{matrix}\right]$
		has full column rank. Hence, $\mathbf X^\dagger \mathbf X  =\mathbf  I_{s_1+s_2}$, where $\mathbf X^\dagger$ denotes the Moore--Penrose pseudo-inverse of $\mathbf X$.
		Denoting $
		\mathbf Y:= [{\mathbf O}_{s_2\times s_1}\ \mathbf I_{s_2}]
		\mathbf X^\dagger $,
		we have 
		\begin{equation*}
		\begin{split}
		\mathbf Y\mathbf B
		=& [{\mathbf O}_{s_2\times s_1}\ \mathbf I_{s_2}]
		\mathbf X^\dagger [\mathbf X\ \mathbf B_{k_{\mathbf B}'+1}\ \dots\ \mathbf B_{R}]\\
		=&	[{\mathbf O}_{s_2\times s_1}\ \mathbf I_{s_2}]
		[\mathbf I_{s_1+s_2}\ \boxplus_{(s_1+s_2)\times s_3}]
		=
		[{\mathbf O}_{s_2\times s_1}\
		\mathbf I_{s_2}\
		\boxplus_{s_2\times s_3}
		]\\
		=&\left[{\mathbf O}_{s_2\times L_1}\ \dots\ {\mathbf O}_{s_2\times L_{k_{\mathbf B}'-2}}\ 
		\left[\begin{array}{l}
		{\mathbf I}_{L_{k_{\mathbf B}'-1}}\\
		\mathbf O_{L_{k_{\mathbf B}'}\times L_{k_{\mathbf B}'-1}}
		\end{array}\right]\ 
		\left[\begin{array}{l}
		\mathbf O_{L_{k_{\mathbf B}'-1}\times L_{k_{\mathbf B}'}}\\
		\mathbf I_{L_{k_{\mathbf B}'}}
		\end{array}\right]	\	
		\boxplus_{s_2\times s_3}	
		\right],
		\end{split}
		\end{equation*}
		where $\boxplus_{p\times q}$ denotes a $p\times q$ matrix that is not further specified.
		From the definition of the matrix $\mathcal C_2'(\cdot)$ it follows  that
		$\mathcal C_2'(\mathbf Y\mathbf B)$ consists of 
		$(R-1) + (R-2) + \dots+(R-k_{\mathbf B}'+2)$ zero blocks followed by the nonzero block
		$\mathbf G:=\left[\begin{array}{l}
		{\mathbf I}_{L_{k_{\mathbf B}'-1}}\\
		\mathbf O_{L_{k_{\mathbf B}'}\times L_{k_{\mathbf B}'-1}}
		\end{array}\right]\wedge 
		\left[\begin{array}{l}
		\mathbf O_{L_{k_{\mathbf B}'-1}\times L_{k_{\mathbf B}'}}\\
		\mathbf I_{L_{k_{\mathbf B}'}}
		\end{array}\right]$ and some other blocks. One can easily verify that $\mathbf G$ is formed by distinct columns of the $\rubinom{s_2}{2}\times \rubinom{s_2}{2}$
		identity matrix, implying that $\mathbf G$ has full column rank.
		Multiplying \cref{eq:withoutY} by $\mathcal C_2(\mathbf Y)$, applying \cref{lemma:newcompoundidentity} of \cref{lemma:compound} and taking into account that
		the first	$(R-1) + (R-2) + \dots+(R-k_{\mathbf B}'+2)$ blocks of $\mathcal C_2'(\mathbf Y\mathbf B)$ are zero,
		we obtain
		\begin{equation}
		\begin{split}
		\label{eq:withoutYY}
		&\mathbf O =\mathcal C_2(\mathbf Y)\mathbf O=
		\mathcal C_2(\mathbf Y)\mathcal C_2'(\mathbf B)\Bdiag(\mathbf f_{1,2},\dots,\mathbf f_{R-1,R})\mathcal C_2(\mathbf A)^T\\
		&=\mathcal C_2'(\mathbf Y\mathbf B)\Bdiag(\mathbf f_{1,2},\dots,\mathbf f_{R-1,R})\mathcal C_2(\mathbf A)^T\\
		&= [\mathbf G  \boxplus\ \dots\ \boxplus]\Bdiag(\mathbf f_{k_{\mathbf B}'-1,k_{\mathbf B}'},\dots,\mathbf f_{R-1,R})
		[{\mathbf a}_{k_{\mathbf B}'-1}\wedge {\mathbf a}_{k_{\mathbf B}'}\ \dots\ \mathbf a_{R-1}\wedge \mathbf a_R]^T,
		\end{split}
		\end{equation}
		where $\boxplus$ denotes a block of  the matrix $\mathcal C_2'(\mathbf Y\mathbf B)$.
		From the definition of $\mathcal C_2(\cdot)$ it  follows that
		$[{\mathbf a}_{k_{\mathbf B}'-1}\wedge {\mathbf a}_{k_{\mathbf B}'}\ \dots\ \mathbf a_{R-1}\wedge \mathbf a_R]=
		\mathcal C_2([{\mathbf a}_{k_{\mathbf B}'-1}\ \dots\ \mathbf a_R])$. 
		Since the last $r_{\mathbf A}$ columns of $\mathbf A$ are linearly independent and  $r_{\mathbf A}\geq R-k_{\mathbf B}'+2$ it follows that
		the vectors ${\mathbf a}_{k_{\mathbf B}'-1}, \dots, \mathbf a_R$ are also linearly independent. Hence, by \cref{lemma:compound}
		the matrix $\mathcal C_2([{\mathbf a}_{k_{\mathbf B}'-1}\ \dots\ \mathbf a_R])$ has full column rank.
		Hence \cref{eq:withoutYY} is equivalent to 
		$$
		\mathbf O=[\mathbf G  \boxplus\ \dots\ \boxplus]\Bdiag(\mathbf f_{k_{\mathbf B}'-1,k_{\mathbf B}'},\dots,\mathbf f_{R-1,R}),
		$$
		implying that $\mathbf G\mathbf f_{k_{\mathbf B}'-1,k_{\mathbf B}'}=\mathbf 0$. Since $\mathbf G$ has full column rank, it follows that
		$\mathbf f_{k_{\mathbf B}'-1,k_{\mathbf B}'}=\mathbf 0$.
		\item
		We show that  $\mathbf f_{r_1,r_2}=\mathbf 0$ for all $1\leq r_1<r_2\leq R$.
		Since $k_{\mathbf A}\geq 2$, the  vectors $\mathbf a_{r_1}, \mathbf a_{r_2}$  are linearly independent.
		Let us extend two vectors $\mathbf a_{r_1},\mathbf a_{r_2}$ to a basis of $\textup{range}(\mathbf A)$
		by adding $r_{\mathbf A}-2$ linearly independent columns of $\mathbf A$. 
		It is clear that there exists an  $R\times R$ permutation matrix $\mathbf \Pi$  such that
		the last $r_{\mathbf A}$ columns of $\mathbf A\mathbf \Pi$ coincide with the chosen basis.   
		Moreover, since $k_{\mathbf B}'-1\geq R-r_{\mathbf A}+1$ we can choose $\mathbf\Pi$ such that 
		the $(k_{\mathbf B}'-1)$th and $k_{\mathbf B}'$th columns of $\mathbf A\mathbf \Pi$ are equal to 
		$\mathbf a_{r_1}$ and $\mathbf a_{r_2}$, respectively.
		We can now reason as under (ii) for $\mathbf A \mathbf{\Pi}$ and $\mathbf B \mathbf{\Pi}$ to obtain that $\mathbf f_{r_1,r_2}=\mathbf 0$.
		
		\item From (iii) we immediately obtain that $\mathbf f=\mathbf 0$.
		Hence, $\Phi(\mathbf A,\mathbf B)$ has  full column rank.
	\end{romannum}

Now we prove that  \cref{eq:ranksofFand G} implies \cref{item:th1cond3}. Substituting  $\mathbf E_r=\mathbf B_r\mathbf C_r^T$ in the expressions for  $\mathbf F$,  we obtain that
	$
	\mathbf F = [\mathbf B_{r_1}\ \mathbf B_{r_2}\ \dots\ \mathbf B_{r_{R-r_{\mathbf A}+2}}]\Bdiag(\mathbf C_{r_1}^T,\mathbf C_{r_2}^T,\dots,\newline \mathbf C_{r_{R-r_{\mathbf A}+2}}^T)$, implying that $r_{[\mathbf B_{r_1}\ \mathbf B_{r_2}\ \dots\ \mathbf B_{r_{R-r_{\mathbf A}+2}}]}\geq r_{\mathbf F}$.
	Hence, by \cref{eq:ranksofFand G}, $k'_{\mathbf B}\geq R-r_{\mathbf A}+2$. Since $k_{\mathbf A }\geq 2$, the result follows from the first part of
	\cref{sta:lemmaApp4}.
\end{proof}
\begin{proof}[Proof of  \cref{sta:lemmaApp45} of  \cref{lemma: Q2viaABC}]
Assume that $(\mathbf a_1\otimes \mathbf B_1)\mathbf f_1+\dots+(\mathbf a_R\otimes \mathbf B_R)\mathbf f_R$ $=\mathbf 0$ for some vectors $\mathbf f_r\in \fF^{L_r}$. It is sufficient to prove that all vectors $\mathbf f_r$ are zero. We rewrite the identity  $(\mathbf a_1\otimes \mathbf B_1)\mathbf f_1+\dots+(\mathbf a_R\otimes \mathbf B_R)\mathbf f_R=\mathbf 0$ in the matrix form
$[\mathbf a_1\ \dots\ \mathbf a_R][
	\mathbf B_1\mathbf f_1\ \dots \ \mathbf B_R\mathbf f_R]^T
=\mathbf O.$ Then from  \cref{lemma:newcompoundidentity2,lemma:newcompoundidentity3} of \cref{lemma:compound} and from the definition of the second compound matrix it follows that
\begin{equation*} 
\begin{split}
\mathcal C_2(\mathbf O) &= \mathcal C_2([\mathbf a_1\ \dots\ \mathbf a_R][
\mathbf B_1\mathbf f_1\ \dots \ \mathbf B_R\mathbf f_R]^T)=
\mathcal C_2([\mathbf a_1\ \dots\ \mathbf a_R])
\mathcal C_2(
[\mathbf B_1\mathbf f_1\ \dots \ \mathbf B_R\mathbf f_R]
)^T\\
&=\sum\limits_{1\leq r_1<r_2\leq R} \left(\wprod{\mathbf a_{r_1}}{\mathbf a_{r_2}}\right)\left(
\wprod{\mathbf B_{r_1}\mathbf f_{r_1}}{\mathbf B_{r_2}\mathbf f_{r_2}}
\right)^T\\
&=
\sum\limits_{1\leq r_1<r_2\leq R} \left(\wprod{\mathbf a_{r_1}}{\mathbf a_{r_2}}\right)\left(
(\wprod{\mathbf B_{r_1}}{\mathbf B_{r_2}})(\mathbf f_{r_1}\otimes \mathbf f_{r_2})
\right)^T,
\end{split}
\end{equation*}
which can be rewritten in  vectorized form as
$\mathbf 0=\Phi(\mathbf A,\mathbf B)[(\mathbf f_1\otimes\mathbf f_2)^T\ \dots\ (\mathbf f_{R-1}\otimes\mathbf f_{R})^T]^T$. Since the matrix $\Phi(\mathbf A,\mathbf B)$ has full column rank, it follows easily  that at least $R-1$ of the vectors $\mathbf f_1,\dots,\mathbf f_R$ are zero. We assume w.l.o.g. that
the last $R-1$ vectors are zero. Then $\mathbf 0= (\mathbf a_1\otimes \mathbf B_1)\mathbf f_1$, which implies that $\mathbf f_1$ is also zero.
\end{proof}
\section{Proofs of \cref{lemma:new,lemma:lemma3.1}}\label{sec:appendixC}
\begin{proof}[Proof of \cref{lemma:new}]
	\ref{eq:Nfcr} Assume that $\mathbf N\mathbf f=\mathbf 0$, where $\mathbf f = [\mathbf f_1^T\ \dots\ \mathbf f_R^T]^T$ and $\mathbf f_r\in\fF^{d_r}$. Then, by construction of $\mathbf N_r$,
	$$
	\mathbf 0=\mathbf C^T\mathbf N\mathbf f =\Bdiag(\mathbf C_1^T\mathbf N_1,\dots,\mathbf C_R^T\mathbf N_R)\mathbf f=
	[(\mathbf C_1^T \mathbf N_1 \mathbf f_1 )^T\ \dots\ (\mathbf C_R^T\mathbf N_R\mathbf f_R)^T]^T,$$
	implying that $\mathbf C_r^T\mathbf N_r\mathbf f_r=\mathbf 0$ for  $r=1,\dots,R$.
	Hence, 
	\begin{equation}
	\mathbf C^T(\mathbf N_r\mathbf f_r)=(\mathbf 0,\dots,\mathbf 0,\mathbf C_r^T\mathbf N_r\mathbf f_r,\mathbf 0,\dots,\mathbf 0)=\mathbf 0,\qquad r=1,\dots,R.
	\label{eq:CTNrfr}
	\end{equation}
	By \cref {eq:unf_T_3,item:th1cond1}, $\mathbf C^T$ has full column rank.  Since $\mathbf N_r$ also has  full column rank, it follows from  \cref{eq:CTNrfr} that $\mathbf f_r=\mathbf 0$ for $r=1,\dots,R$. Hence we must have $\mathbf f=\mathbf 0$. Thus the matrix $\mathbf N$ has full column rank.
	
	\ref{eq:NNfcr} 
	\tcr{
	It follows from   \cref{eq:Nfcr} that $[\mathbf N_1\otimes \mathbf N_1\ \dots\ \mathbf N_R\otimes \mathbf N_R]$ has full column rank.
	Obviously, $\Bdiag(\mathbf M_1,\dots,\mathbf M_R)$ has full column rank.  Since
    $\mathbf W= [\mathbf N_1\otimes \mathbf N_1\ \dots\ \mathbf N_R\otimes \mathbf N_R]\Bdiag(\mathbf M_1,\dots,\mathbf M_R)$, it also has full column rank.}
%
%
%
	
	\ref{linindEEE} Since, by \cref{item:th1cond1}, $r_{\unf{T}{3}}=K$ and, by \cref{eq:unf_T_3},
	$\unf{T}{3}=[\mathbf a_1\otimes \mathbf I_J\ \dots\ \mathbf a_R\otimes \mathbf I_J][\mathbf E_1^T\ \dots\ \mathbf E_R^T]^T$, it follows that the $JR\times K$ matrix $[\mathbf E_1^T\ \dots\ \mathbf E_R^T]^T$ has full column rank. Hence for any $r$ the columns of $[\mathbf E_1^T\ \dots\ \mathbf E_R^T]^T\mathbf N_r = [\mathbf O\ \dots\mathbf O\ (\mathbf E_r\mathbf N_r)^T\ \mathbf O\ \dots\ \mathbf O]^T$ are nonzero. Assume that $\mathbf O=\alpha_1\mathbf E_1+\dots+\alpha_R\mathbf E_R$ for some $\alpha_1,\dots,\alpha_R\in\fF$. Then for any $r$,
	$\mathbf O=(\alpha_1\mathbf E_1+\dots+\alpha_R\mathbf E_R)\mathbf N_r=\alpha_r\mathbf E_r\mathbf N_r$. Since  $\mathbf E_r\mathbf N_r$ is not the zero matrix, it follows that $\alpha_r=0$. Thus, the matrices $\mathbf E_1,\dots,\mathbf E_R$ are linearly independent.
\end{proof}
\begin{proof}[Proof of \cref{lemma:lemma3.1}]
	By \cref{eq:unf_T_1},
	\begin{equation}
	\unf{T}{1}=[\operatorname{vec}(\mathbf E_1)\ \dots\ \operatorname{vec}(\mathbf E_R)]\mathbf A^T=[\operatorname{vec}(\tilde{\mathbf E}_1)\ \dots\ \operatorname{vec}(\tilde{\mathbf E}_{\tilde R})]\tilde{\mathbf A}^T,\label{eq:T1AA}
	\end{equation}
	where $\tilde{\mathbf A}=[\tilde{\mathbf a}_1\ \dots\ \tilde{\mathbf a}_{\tilde R}]$.
	
	{\em Case 1:  \cref{item:th1cond5} holds.} Then, $\mathbf A$ has full column rank. Hence, by \cref{eq:T1AA},
	\begin{equation*}
	[\operatorname{vec}(\mathbf E_1)\ \dots\ \operatorname{vec}(\mathbf E_R)]=[\operatorname{vec}(\tilde{\mathbf E}_1)\ \dots\ \operatorname{vec}(\tilde{\mathbf E}_{\tilde R})](\mathbf A^\dagger\tilde{\mathbf A})^T.
	\end{equation*}
	\tcr{Since any column of $\tilde{\mathbf A}$ is a column of $\mathbf A$, }
	each column of $\mathbf A^\dagger\tilde{\mathbf A}$  contains at most one nonzero entry.
	Since	$\mathbf E_1,\dots,\mathbf E_R$ are nonzero matrices, it follows that the columns of $(\mathbf A^\dagger\tilde{\mathbf A})^T\in\fF^{\tilde R\times R}$ are also nonzero, which is possible only if
	$\tilde R=R$ and  $\tilde{\mathbf A}=\mathbf A\mathbf P$ for some $R\times R$ permutation matrix $\mathbf P$. Hence, by \cref{eq:T1AA},
	$[\operatorname{vec}(\mathbf E_1)\ \dots\ \operatorname{vec}(\mathbf E_R)]=[\operatorname{vec}(\tilde{\mathbf E}_1)\ \dots\ \operatorname{vec}(\tilde{\mathbf E}_{\tilde R})]\mathbf P^T$.
	Thus, the decompositions coincide up to permutation of summands. It is also clear that the matrices $\mathbf E_1,\dots,\mathbf E_R$  can be computed by solving the system of linear equations 
		$[\operatorname{vec}(\mathbf E_1)\ \dots\ \operatorname{vec}(\mathbf E_R)]\mathbf A^T=\unf{T}{1}$.
	
	{\em Case 2:   \cref{item:th1cond6} holds.}
	To prove \cref{st:hardlemma1} it is sufficient to show that the matrices $\mathbf E_1,\dots,\mathbf E_R$ can be computed by EVD up to  scaling. Indeed, if
	$\mathbf E_r=x_r\hat{\mathbf E}_r$ and the matrices $\hat{\mathbf E}_r$ are known, then,  by \cref{eq:unf_T_1},
	the scaling factors $x_r$ can be found as  from the linear equation $[\mathbf a_1\otimes\operatorname{vec}(\hat{\mathbf E}_1)\ \dots\ \mathbf a_r\otimes\operatorname{vec}(\hat{\mathbf E}_R)][x_1\ \dots\ x_r]^T=\operatorname{vec}(\unf{T}{1})$.
	
	We choose arbitrary integers $r_1,\dots,r_{R-r_{\mathbf A}+2}$ such that $1\leq r_1<\dots<r_{R-r_{\mathbf A}+2}\leq R$ and show that the matrices $\mathbf E_{r_1},\dots,\mathbf E_{r_{R-r_{\mathbf A}+2}}$ can be computed by EVD up to scaling. We set 
	\begin{equation}
	\Omega = \{r_1,\dots,r_{R-r_{\mathbf A}+2}\}\ \text{ and }\ \{p_1,\dots,p_{r_{\mathbf A}-2}\}=\{1,\dots,R\}\setminus\Omega.\label{eq:Omega}
	\end{equation}
	Since $k_{\mathbf A}=r_{\mathbf A}$, it follows that the intersection of the null space of the
	$(r_{\mathbf A}-2)\times I$ matrix $[\mathbf a_{p_1}\ \dots\ \mathbf a_{p_{r_{\mathbf A}-2}}]^T$ and the column space of $\mathbf A$ is two-dimensional. Let the intersection be spanned by the vectors $\mathbf h_{\Omega,1},\mathbf h_{\Omega,2}\in\fF^I$,
	where here and later in the proof the subindex ``$ \Omega$'' indicates that a quantity depends on $r_1,\dots,r_{R-r_{\mathbf A}+2}$.
	Then again, since $k_{\mathbf A}=r_{\mathbf A}$, it follows that
	\begin{equation}
	\text{any two columns of }
	\begin{bmatrix}     \mathbf h_{\Omega,1}^T {\mathbf a}_{r_1}\ \dots\ \mathbf h_{\Omega,1}^T {\mathbf a}_{r_{R-r_{\mathbf A}+2}}\\
	\mathbf h_{\Omega,2}^T {\mathbf a}_{r_1}\ \dots\ \mathbf h_{\Omega,2}^T {\mathbf a}_{r_{R-r_{\mathbf A}+2}}
	\end{bmatrix} \ \text{are linearly independent.}\label{eq:bigkrank}
	\end{equation}
	Let $\mathcal Q_{\Omega}$ denote the  $2\times J\times K$ tensor such that $\mathbf Q_{\Omega (1)}=\unf{T}{1}[\mathbf h_{\Omega,1}\ \mathbf h_{\Omega,2}]$.
	Then, by \cref{eq:unf_T_1},
	\begin{equation}
	\mathcal Q_{\Omega} =
	\sum\limits_{r=1}^{R}\begin{bmatrix}\mathbf h_{\Omega,1}^T\mathbf a_r\\ \mathbf h_{\Omega,2}^T\mathbf a_r
	\end{bmatrix}
	\circ\mathbf E_{r}
	=
	\sum\limits_{k=1}^{R-r_{\mathbf A}+2}\begin{bmatrix}\mathbf h_{\Omega,1}^T\mathbf a_{r_k}\\\mathbf h_{\Omega,2}^T\mathbf a_{r_k}
	\end{bmatrix}
	\circ\mathbf E_{r_k} = 
	\sum\limits_{k=1}^{R-r_{\mathbf A}+2}\begin{bmatrix}\mathbf h_{\Omega,1}^T\mathbf a_{r_k}\\\mathbf h_{\Omega,2}^T\mathbf a_{r_k}
	\end{bmatrix}
	\circ(\mathbf B_{r_k}\mathbf C_{r_k}^T),\label{eq:Qmtensdec1}
	\end{equation}
	where $\mathbf B_{r_k}\in\fF^{J\times L_{r_k}}$ and  $\mathbf C_{r_k}\in\fF^{K\times L_{r_k}}$ denote full column rank matrices such that
	$\mathbf E_{r_k}=\mathbf B_{r_k}\mathbf C_{r_k}^T$.
	Since   \cref{item:th1cond6} in \cref{thm: maintheorem} is equivalent to \cref{item:th1cond6ABC} in \cref{thm: maintheoremABC}, it follows that
	 	 $ k_{\mathbf B}'\geq R-r_{\mathbf A}+2$ and $ k_{\mathbf C}'\geq R-r_{\mathbf A}+2$. Hence, 
		\begin{equation}
		[\mathbf B_{r_1}\   \dots\ \mathbf B_{r_{R-\tcr{r_{\mathbf A}}+2}}]\ \text{ and }\ 
		[\mathbf C_{r_1}\  \dots\ \mathbf  C_{r_{R-\tcr{r_{\mathbf A}}+2}}]\ 
		\ \text{have full column rank.}\label{eq:ranksofFand GABCincnew} 
		\end{equation}
	Hence, by \cref{thm:ll1_gevd},  the    decomposition of $\mathcal Q_{\Omega}$ into a sum of max ML rank-$(1,L_{r_k},L_{r_k})$ terms is unique and can be computed by EVD.
	Thus,  the matrices $\mathbf E_{r_1}, \dots,\mathbf E_{r_{R-r_{\mathbf A}+2}}$ can be computed by EVD up to scaling. Since the indices $r_1,\dots,$ $r_{R-r_{\mathbf A}+2}$ were chosen arbitrary, it follows that
	all matrices  $\mathbf E_{r_1},\dots,\mathbf E_{r_{R-r_{\mathbf A}+2}}$ can be computed by EVD up to scaling. The overall procedure is summarized in 
	steps $11-18$ of \cref{algorithm:1}.

	Now we prove \cref{st:hardlemma2}.
	First we show that $\tilde{R}=R$ and that the  $\tilde{\mathbf E}_1,\dots,\tilde{\mathbf E}_{R}$ involves the same matrices
	as $\mathbf E_1,\dots,\mathbf E_R$. 
	Similarly to \cref{eq:Qmtensdec1} we obtain that
	\begin{equation}
	\mathcal Q_{\Omega}=
	\sum\limits_{r=1}^{\tilde{R}}   \begin{bmatrix}\mathbf h_{\Omega,1}^T\tilde{\mathbf a}_{r}\\ \mathbf h_{\Omega,2}^T\tilde{\mathbf a}_{r}
	\end{bmatrix}\circ   \tilde{\mathbf E}_r.\label{eq:Qmtensdec2}
	\end{equation}
	It is clear that there exist $\rubinom{R}{R-r_{\mathbf A}+2}$ sets $\Omega$ of the form \cref{eq:Omega}. Thus, by \cref{eq:Qmtensdec1,eq:Qmtensdec2}, we obtain a system of $\rubinom{R}{R-r_{\mathbf A}+2}$ identities:
	\begin{equation}
	\mathcal Q_{\Omega}=
	\sum\limits_{k=1}^{R-r_{\mathbf A}+2}\begin{bmatrix}\mathbf h_{\Omega,1}^T\mathbf a_{r_k}\\\mathbf h_{\Omega,2}^T\mathbf a_{r_k}
	\end{bmatrix}
	\circ\mathbf E_{r_k}=
	\sum\limits_{r=1}^{\tilde{R}}   \begin{bmatrix}\mathbf h_{\Omega,1}^T\tilde{\mathbf a}_{r}\\ \mathbf h_{\Omega,2}^T\tilde{\mathbf a}_{r}
	\end{bmatrix}\circ   \tilde{\mathbf E}_r,\ 1\leq r_1<\dots<r_{R-r_{\mathbf A}+2}\leq R.\label{eq:QOmegaQOmega}
	\end{equation}
	Hence, by \cref{eq:unf_T_3,eq:Qmtensdec1}, system \cref{eq:QOmegaQOmega} can be rewritten in   matrix form as
	\begin{equation}
	\begin{split}
	\mathbf Q_{\Omega(3)} = &\left[\begin{bmatrix}\mathbf h_{\Omega,1}^T\mathbf a_{r_1}\\\mathbf h_{\Omega,2}^T\mathbf a_{r_1}
	\end{bmatrix}\otimes\mathbf B_{r_1}\ \dots\ \begin{bmatrix}\mathbf h_{\Omega,1}^T\mathbf a_{r_{R-r_{\mathbf A}+2}}\\\mathbf h_{\Omega,2}^T\mathbf a_{r_{R-r_{\mathbf A}+2}}
	\end{bmatrix}\otimes\mathbf B_{r_{R-r_{\mathbf A}+2}}\right][\mathbf C_{r_1}\ \dots\ \mathbf C_{r_{R-r_{\mathbf A}+2}}]^T=\\
	& \sum\limits_{r=1}^{\tilde{R}}   \begin{bmatrix}\mathbf h_{\Omega,1}^T\tilde{\mathbf a}_{r}\\ \mathbf h_{\Omega,2}^T\tilde{\mathbf a}_{r}
	\end{bmatrix}\otimes   \tilde{\mathbf E}_r,\qquad\qquad\qquad \qquad\qquad\ { 1\leq r_1<\dots<r_{R-r_{\mathbf A}+2}\leq R.}
	\end{split}\label{QOmegamatrixform}
	\end{equation}
	From \cref{eq:bigkrank}, \cref{eq:ranksofFand GABCincnew} and the first identity in \cref{QOmegamatrixform}, it follows that $\mathbf Q_{\Omega(3)}$ has rank $L_{r_1}+\dots+L_{r_{R-r_{\mathbf A}+2}}$.
	Since the rank is subadditive, it follows from \cref{QOmegamatrixform}, that
	\begin{equation}
	L_{r_1}+\dots+L_{r_{R-r_{\mathbf A}+2}} \leq \sum\limits_{r=1}^{\tilde{R}} r \left( \begin{bmatrix}\mathbf h_{\Omega,1}^T\tilde{\mathbf a}_{r}\\ \mathbf h_{\Omega,2}^T\tilde{\mathbf a}_{r}
	\end{bmatrix}\right)r_{ \tilde{\mathbf E}_r},\   1\leq r_1<\dots<r_{R-r_{\mathbf A}+2}\leq R,
	\label{eq:allidentities}
	\end{equation}
	where $r(\mathbf f)$  denotes the rank of a $2\times 1$ matrix $\mathbf f$: $r(\mathbf 0)=0$ and $r(\mathbf f) = 1$, if $\mathbf f\ne 0$. 
	It is clear that for each $r$ there exist  exactly $\rubinom{R-1}{R-r_{\mathbf A}+1}$ subsets $\{r_1,\dots,r_{R-r_{\mathbf A}+2}\}\subset\{1,\dots,R\}$ that contain $r$.
	Hence each $L_r$ appears in  exactly $\rubinom{R-1}{R-r_{\mathbf A}+1}$ inequalities in \cref{eq:allidentities}.
	Since  $\tilde{\mathbf a}_1=\mathbf a_r$ for some $r$, it follows that the term
	$
	r \left( \begin{bmatrix}\mathbf h_{\Omega,1}^T\tilde{\mathbf a}_{1}\\ \mathbf h_{\Omega,2}^T\tilde{\mathbf a}_{1}
	\end{bmatrix}\right)r_{ \tilde{\mathbf E}_1} =
	r \left( \begin{bmatrix}\mathbf h_{\Omega,1}^T{\mathbf a}_{r}\\ \mathbf h_{\Omega,2}^T{\mathbf a}_{r}
	\end{bmatrix}\right)r_{ \tilde{\mathbf E}_1}
	$
	appears in the same $\rubinom{R-1}{R-r_{\mathbf A}+1}$ inequalities as $L_r$, implying, by the construction of  $\mathbf h_{\Omega,1}$ and $\mathbf h_{\Omega,2}$, that 
	$\begin{bmatrix}\mathbf h_{\Omega,1}^T{\mathbf a}_{r}\\ \mathbf h_{\Omega,2}^T{\mathbf a}_{r}
	\end{bmatrix}\ne\mathbf 0$. Thus, $r_{ \tilde{\mathbf E}_1}$ appears in  exactly $\rubinom{R-1}{R-r_{\mathbf A}+1}$ inequalities in \cref{eq:allidentities}. In the same fashion one can prove that each of the values $1\cdot r_{ \tilde{\mathbf E}_2},\dots,1\cdot r_{ \tilde{\mathbf E}_{\tilde R}}$ appears in \cref{eq:allidentities}
	exactly $\rubinom{R-1}{R-r_{\mathbf A}+1}$ times.
	Thus, summing  all inequalities in \cref{eq:allidentities} and taking into account that $\tilde R\leq R$ and $r_{ \tilde{\mathbf E}_r}\leq L_r$  for all $r$ we obtain 
	\begin{multline}
	(L_1+\dots+L_R)\rubinom{R-1}{R-r_{\mathbf A}+1}\leq (r_{ \tilde{\mathbf E}_1} + \dots + r_{ \tilde{\mathbf E}_{\tilde R}})\rubinom{R-1}{R-r_{\mathbf A}+1}\leq\\
	(L_1+\dots+L_{\tilde R})\rubinom{R-1}{R-r_{\mathbf A}+1}\leq (L_1+\dots+L_R)\rubinom{R-1}{R-r_{\mathbf A}+1}.
	\label{summofineq}
	\end{multline}
	Hence $\tilde R=R$ and $r_{ \tilde{\mathbf E}_r}= L_r$ for all $r$.
	
	To complete the proof of \cref{st:hardlemma2} we need to show that the terms $\tilde{\mathbf a}_1\circ\tilde{\mathbf E}_1,\dots,\tilde{\mathbf a}_R\circ\tilde{\mathbf E}_R$
	coincide with the terms ${\mathbf a}_1\circ {\mathbf E}_1,\dots, {\mathbf a}_R\circ {\mathbf E}_R$. If we assume that at least one of the inequalities in \cref{eq:allidentities} is strict, then
	the first inequality in \cref{summofineq} should also be strict, which is not possible. Thus,  \cref{eq:allidentities} holds with ``$\leq$'' replaced by
	``$=$''. Hence, by  \cref{thm:ll1_gevd}, the two decompositions of $\mathcal Q_{\Omega}$ in \cref{eq:QOmegaQOmega} coincide up to permutation of their terms. This readily
	implies that the matrices $\tilde{\mathbf E}_1,\dots,\tilde{\mathbf E}_R$ coincide with  $\lambda_1\mathbf E_1,\dots,\lambda_R\mathbf E_R$ for some $\lambda_1,\dots,\lambda_R\in\fF\setminus\{0\}$, i.e.,
	there exists an $R\times R$ permutation matrix $\mathbf P$ such that 
	\begin{equation}
	[\operatorname{vec}(\tilde{\mathbf E}_1)\ \dots\ \operatorname{vec}(\tilde{\mathbf E}_{\tilde R})] = 	[\operatorname{vec}(\mathbf E_1)\ \dots\ \operatorname{vec}(\mathbf E_R)]\operatorname{diag}(\lambda_1,\dots,\lambda_R)\mathbf P.\label{eq:ultimateintheproof}
	\end{equation}
	Substituting \cref{eq:ultimateintheproof} in \cref{eq:T1AA} we obtain that
	\begin{equation}
	[\operatorname{vec}(\mathbf E_1)\ \dots\ \operatorname{vec}(\mathbf E_R)]\mathbf A^T=[\operatorname{vec}(\mathbf E_1)\ \dots\ \operatorname{vec}(\mathbf E_R)]\operatorname{diag}(\lambda_1,\dots,\lambda_R)\mathbf P\tilde{\mathbf A}^T.\label{eq:eeeaeeepa}
	\end{equation}
	Since the matrices $\mathbf E_1,\dots,\mathbf E_R$ are linearly independent, it follows from  \cref{eq:eeeaeeepa} that $\mathbf A^T=\operatorname{diag}(\lambda_1,\dots,\lambda_R)\mathbf P\tilde{\mathbf A}^T$. Hence $\mathbf A=\tilde{\mathbf A}\mathbf P^T\operatorname{diag}(\lambda_1,\dots,\lambda_R)$. Since \tcr{any column of $\tilde{\mathbf A}$ is a column of $\mathbf A$ and since} $k_{\mathbf A}=r_{\mathbf A}\geq 2$, it follows that
	$\lambda_1=\dots=\lambda_R=1$. Hence $\tilde{\mathbf A}=\mathbf A\mathbf P$ and, by \cref{eq:ultimateintheproof},
	$[\operatorname{vec}(\tilde{\mathbf E}_1)\ \dots\ \operatorname{vec}(\tilde{\mathbf E}_{\tilde R})] = 	[\operatorname{vec}(\mathbf E_1)\ \dots\ \operatorname{vec}(\mathbf E_R)]\mathbf P$, i.e.,
	the terms $\tilde{\mathbf a}_1\circ\tilde{\mathbf E}_1,\dots,\tilde{\mathbf a}_R\circ\tilde{\mathbf E}_R$
	coincide with the terms ${\mathbf a}_1\circ {\mathbf E}_1,\dots, {\mathbf a}_R\circ {\mathbf E}_R$.
\end{proof}
\section{Nonuniqueness of the generic decomposition  of a $2\times 8\times 7$ tensor into a sum of $3$ max ML rank-$(1,3,3)$ terms}\label{AppendixE}
Let $\mathcal T$ admit decomposition \cref{eq:LrLr1mainBC} with generic factor matrices $\mathbf A$, $\mathbf B$, and  $\mathbf C$. Then the matrices $\mathbf U:=[\mathbf a_2\ \mathbf a_3]\in\fF^{2\times 2}$, $\mathbf V:=[\mathbf b_2\ \dots\ \mathbf b_9]\in\fF^{8\times 8}$, and $\mathbf W:=[\mathbf c_1\ \dots\ \mathbf c_5\ \mathbf c_7\ \mathbf c_8]\in\fF^{7\times 7}$ are nonsingular. Let $\widehat{\mathcal T}$ denote a tensor such that
$\widehat{\mathbf T}_{(3)} = (\mathbf U^{-1}\otimes\mathbf V^{-1})\unf{T}{3}\mathbf W^{-T}$. Then, by \cref{eq:unf_T_3}, $\widehat{\mathcal T}$ admits the decomposition
of the form \cref{eq:LrLr1mainBC}, where $\mathbf A$, $\mathbf B$, and $\mathbf C$ are replaced by
%
\begin{gather*}
\mathbf U^{-1}\mathbf A = \begin{bmatrix}
\numberx&1&0\\
\numbery&0&1
\end{bmatrix},\quad
\mathbf V^{-1}\mathbf B = [\vectora\ \mathbf I_8],\ \text{ and }\ 
\mathbf W^{-1}\mathbf C = [\mathbf e_1\ \mathbf e_2\ \mathbf e_3\ \mathbf e_4\ \mathbf e_5\ \vectorb\ \mathbf e_6\ \mathbf e_7\ \vectorc],
\end{gather*}
respectively.
It is  clear that  a decomposition of $\widehat{\mathcal T}$  with factor matrices 
$\widehat{\mathbf A}$, $\widehat{\mathbf B}$, and $\widehat{\mathbf C}$ generates
a decomposition of $\mathcal T$ with   factor matrices $\mathbf U\widehat{\mathbf A}$, $\mathbf V\widehat{\mathbf B}$, and $\mathbf W\widehat{\mathbf C}$. In particular, if the decomposition of $\widehat{\mathcal T}$ is not unique, then the decomposition of $\mathcal T$ is not unique either. 
Below we  present a procedure to construct a two-parameter family of decompositions of $\widehat{\mathcal T}$.
First we choose parameters $\parameterhone, \parameterhtwo\in\fF$  and compute the values $\valuep$, $\valuehthree$, $\valuehfour$, and $\valued$:
\begin{flalign*}
\valuep   &\mapsfrommy (\numbera_1\numberb_2 - \numberb_1 + \numbera_2\numberb_3)\parameterhone + (\numbera_1\numberc_2 - \numberc_1 + \numbera_2\numberc_3)\parameterhtwo + 1,    \span\span\span\span   \\
\valuehthree &\mapsfrommy (\numbera_3\numberb_4 - \numbera_5 + \numbera_4\numberb_5)\numberx \parameterhone + (\numbera_3\numberc_4 + \numbera_4\numberc_5)\numberx \parameterhtwo,           \span\span\span\span  \\
\valuehfour &\mapsfrommy (\numbera_6\numberb_6 + \numbera_7\numberb_7)\numbery \parameterhone + (\numbera_6\numberc_6 - \numbera_8 + \numbera_7\numberc_7)\numbery \parameterhtwo,           \span\span\span\span \\
\valued   &\mapsfrommy \valuehthree+\valuep-\valuehfour\valuep.\span\span\span\span\\ 
\intertext{Second, if $\valuep$ and $\valued$ are nonzero, we  also compute the values:}
\valueqone &\mapsfrommy -\parameterhone \valuehfour/\valued,& \valueqtwo&\mapsfrommy -\parameterhtwo \valuehthree/\valued,&
\valueqthree  &\mapsfrommy (\parameterhtwo+\valueqtwo)/\valuep,& \valueqfour&\mapsfrommy \valuep \valueqone - \parameterhone, \\
\valueyone &\mapsfrommy \numberc_1\valueqthree + \numberb_1 \valueqone+1,& \valueytwo&\mapsfrommy \numberc_1 \valueqtwo + \numberb_1 \valueqfour+1,&
\valuezone  &\mapsfrommy \numberc_2 \valueqthree + \numberb_2 \valueqone, & \valueztwo&\mapsfrommy \numberc_2 \valueqtwo + \numberb_2 \valueqfour,\\
\valuetone &\mapsfrommy \numberc_3 \valueqthree + \numberb_3 \valueqone, & \valuettwo&\mapsfrommy \numberc_3 \valueqtwo + \numberb_3 \valueqfour,& \span\span\span\\
\values&\mapsfrommy \numberc_4 \parameterhtwo /\valued, &\valueu   &\mapsfrommy \numberc_5 \parameterhtwo /\valued , &
\valuev   &\mapsfrommy -\numberb_6 \parameterhone /\valued, & \valuew&\mapsfrommy -\numberb_7 \parameterhone /\valued.
\end{flalign*}
Third, we construct   matrices $\tilde{\mathbf E}_1$ , $\tilde{\mathbf E}_2$ , and  $\tilde{\mathbf E}_3$ as 
\begin{gather}
\tilde{\mathbf E}_1 :=
\begin{bmatrix}
\numbera_1& 1& 0& 0& 0& 0& 0\\
\numbera_2& 0& 1& 0& 0& 0& 0\\
\numbera_3 \valueyone& \numbera_3 \valuezone& \numbera_3 \valuetone&   \numbera_3\values&     \numbera_3\valueu&     \numbera_3\valuev&    \numbera_3\valuew\\
\numbera_4 \valueyone& \numbera_4 \valuezone& \numbera_4 \valuetone&   \numbera_4\values&     \numbera_4\valueu&     \numbera_4\valuev&    \numbera_4\valuew\\
\numbera_5 \valueyone& \numbera_5 \valuezone& \numbera_5 \valuetone&   \numbera_5\values&     \numbera_5\valueu&     \numbera_5\valuev&    \numbera_5\valuew\\
\numbera_6 \valueytwo& \numbera_6 \valueztwo& \numbera_6 \valuettwo&   \numbera_6\values\valuep&   \numbera_6\valueu\valuep&   \numbera_6\valuev\valuep&  \numbera_6\valuew\valuep\\
\numbera_7 \valueytwo& \numbera_7 \valueztwo& \numbera_7 \valuettwo&   \numbera_7\values\valuep&   \numbera_7\valueu\valuep&   \numbera_7\valuev\valuep&  \numbera_7\valuew\valuep\\
\numbera_8 \valueytwo& \numbera_8 \valueztwo& \numbera_8 \valuettwo&   \numbera_8\values\valuep&   \numbera_8\valueu\valuep&   \numbera_8\valuev\valuep&  \numbera_8\valuew\valuep
\end{bmatrix},\nonumber \\
\tilde{\mathbf E}_2 := \widehat{\mathbf H}_1-\numberx\tilde{\mathbf E}_1,\qquad
\tilde{\mathbf E}_3 := \widehat{\mathbf H}_2-\numbery\tilde{\mathbf E}_1,\label{eq:hatT1T2}
\end{gather}
where $\widehat{\mathbf H}_1\in\fF^{8\times 7}$ and $\widehat{\mathbf H}_2\in\fF^{8\times 7}$ denote the horizontal slices of $\widehat{\mathcal T}$.
The identities in \cref{eq:hatT1T2} mean that  $\widehat{\mathcal T}=\begin{bmatrix}
\numberx\\ \numbery
\end{bmatrix}\circ \tilde{\mathbf E}_1+\begin{bmatrix}
1\\ 0
\end{bmatrix}\circ \tilde{\mathbf E}_2+\begin{bmatrix}
0\\ 1
\end{bmatrix}\circ \tilde{\mathbf E}_3$, i.e., $\widehat{\mathcal T}$ admits a two-parameter family of decompositions, as indicated above. By symbolic computations in MATLAB we have also verified that all $4\times 4$ minors of 
$\tilde{\mathbf E}_1$, $\tilde{\mathbf E}_2$, and $\tilde{\mathbf E}_3$ are identically zero, that is 
$\tilde{\mathbf E}_1$, $\tilde{\mathbf E}_2$, and  $\tilde{\mathbf E}_3$ are at most rank-$3$ matrices.
\section{Proof of \cref{thm:maingenStrassen}} \label{Appendixthm:maingenStrassen}
The following theorem complements results on uniqueness\footnote{\tcr{
It can be shown that  if $\mathbf C$ has full column rank, then \cref{Thm:determgen} guarantees uniqueness under more relaxed assumptions than
\cref{thm: maintheoremABC}. On the other hand, assumption \cref{eq:U2} in \cref{Thm:determgen} is not easy to verify for particular $\mathbf A$ and $\mathbf B$ and
\cref{Thm:determgen} does not come with an EVD-based algorithm.}}
presented in \cref{subsub251}
and will be used in the proof of \cref{thm:maingenStrassen}. Namely, we will show that
\cref{thm:maingenStrassen} is the generic counterpart of \cref{Thm:determgen}. 
%
\begin{theorem}\label{Thm:determgen} 
	Let $\mathcal T\in\fF^{I\times J\times K}$ admit   decomposition  \cref{eq:LrLr1mainBC}
	with  $\mathbf a_r\ne\mathbf 0$ and $r_{\mathbf B_r}=r_{\mathbf C_r}=L_r$ for all $r$.
	Assume that the matrix $\mathbf C$ has full column rank and that the matrices $\mathbf A$ and $\mathbf B$ satisfy the following assumption:
	\begin{equation}
	\begin{split}
	&\text{if at least two of the vectors }\mathbf g_1\in\mathbb C^{L_1},\dots,\mathbf g_R\in\mathbb C^{L_R} \ \text{are nonzero},\  \\
	&\text{then the rank of }\mathbf a_1(\mathbf B_1\mathbf g_1)^T+\dots +\mathbf a_R(\mathbf B_R\mathbf g_R)^T \text{ is at least } 2.
	\end{split}\label{eq:U2}
	\end{equation}
	Then  the decomposition of $\mathcal T$ into a sum of \MLatmost terms is unique. 
\end{theorem}
\begin{proof}
	Since $\mathbf C$ has full column rank we have that $K\geq \sum L_r$.
	By \cref{item:thm:sonvention:statement1} of \cref{thm:convention}, we can assume that $K=\sum L_r$, i.e., that $\mathbf C$ is square and nonsingular.
	
	i) First we reformulate assumption \cref{eq:U2}. Such reformulation  will immediately imply  that 
	\begin{equation}
	 k_{\mathbf A}\geq 2\ \text{ and  matrix } [\mathbf a_1\otimes\mathbf B_1\ \dots\ \mathbf a_R\otimes\mathbf B_R]\
	\text{has full column rank.}
	\label{eq:kA2andfcr}
	\end{equation}
	If the rank of $\mathbf a_1(\mathbf B_1\mathbf g_1)^T+\dots +\mathbf a_R(\mathbf B_R\mathbf g_R)^T$ is less than $2$, then there exist vectors $\mathbf z\in\fF^{I}$ and $\mathbf y\in\fF^J$ such that
	\begin{equation}
	\mathbf a_1(\mathbf B_1\mathbf g_1)^T+\dots +\mathbf a_R(\mathbf B_R\mathbf g_R)^T=\mathbf z\mathbf y^T.
	\label{eq:beforetandvec}
	\end{equation}
	 Transposing and vectorizing both sides of \cref{eq:beforetandvec} we obtain that
	$(\mathbf a_1\otimes\mathbf B_1)\mathbf g_1+\dots+(\mathbf a_R\otimes\mathbf B_R)\mathbf g_R = \mathbf z\otimes\mathbf y$. Hence
		assumption \cref{eq:U2} can be reformulated as follows:
	\begin{equation}
\begin{split}
&\text{the identity 	}\ (\mathbf a_1\otimes\mathbf B_1)\mathbf g_1+\dots+(\mathbf a_R\otimes\mathbf B_R)\mathbf g_R = \mathbf z\otimes\mathbf y\
\text{ holds}\\
&\text{only if at most one of }\	\mathbf g_1,\dots,\mathbf g_{R}\ \text{is nonzero.}
\end{split}\label{eq:reformulatedU2}
\end{equation}	
One can now easily derive \cref{eq:kA2andfcr} from \cref{eq:reformulatedU2}.

	ii)  Now we prove uniqueness. 	
	Let  $\mathcal T = \sum_{r=1}^{\widehat{R}}\widehat{\mathbf a}_r\circ(\widehat{\mathbf B}_r\widehat{\mathbf C}_r^T)$, where $\widehat{R}\leq R$, $\widehat{\mathbf a}_r\ne \mathbf 0$,  $\widehat{\mathbf B}_r\in\fF^{J\times \widehat{L}_r}$ and $\widehat{\mathbf C}_r\in\fF^{K\times \widehat{L}_r}$ have full column rank, and $\widehat{L}_r\leq L_r$ for  $r=1,\dots,\widehat{R}$.
	Then, by \cref{eq:unf_T_3},
	\begin{equation}
	[\mathbf a_1\otimes\mathbf B_1\ \dots\ \mathbf a_R\otimes\mathbf B_R]\mathbf C^T=\unf{T}{3}
	=[\widehat{\mathbf a}_1\otimes\widehat{\mathbf B}_1\ \dots\ \widehat{\mathbf a}_{\widehat{R}}\otimes\widehat{\mathbf B}_{\widehat{R}}]\widehat{\mathbf C}^T.
	\label{eq:T3T3}
	\end{equation}
	Since, by \cref{eq:kA2andfcr}, $[\mathbf a_1\otimes\mathbf B_1\ \dots\ \mathbf a_R\otimes\mathbf B_R]$
	has full column rank and since $\mathbf C$ is a nonsingular matrix, it follows from \cref{eq:T3T3} that $r_{	\unf{T}{3}}=\sum L_r$.
	Hence the matrices $[\widehat{\mathbf a}_1\otimes\widehat{\mathbf B}_1\ \dots\ \widehat{\mathbf a}_{\widehat{R}}\otimes\widehat{\mathbf B}_{\widehat{R}}]$
	and $\widehat{\mathbf C}$
	 are at least  rank-$\sum L_r$, implying that $\sum\limits_{r=1}^{\widehat{R}}\widehat{L}_r\geq\sum\limits_{r=1}^R L_r$. On the other hand, since $\widehat{R}\leq R$ and $\widehat{L}_r\leq L_r$ for  $r=1,\dots,\widehat{R}$, we also have that $\sum\limits_{r=1}^{\widehat{R}}\widehat{L}_r\leq\sum\limits_{r=1}^R L_r$. Hence  $\sum\limits_{r=1}^{\widehat{R}}\widehat{L}_r=\sum\limits_{r=1}^R L_r$ which is possible only if $\widehat{R}=R$ and $\widehat{L}_r = L_r$ for all $r$.  Multiplying 
	\cref{eq:T3T3} by $\widehat{\mathbf C}^{-T}$ we obtain that
	\begin{equation}
	[\mathbf a_1\otimes\mathbf B_1\ \dots\ \mathbf a_R\otimes\mathbf B_R]\mathbf G
	=[\widehat{\mathbf a}_1\otimes\widehat{\mathbf B}_1\ \dots\ \widehat{\mathbf a}_R\otimes\widehat{\mathbf B}_R],
	\label{eq:bothsides}
	\end{equation}
	where $\mathbf G=\mathbf C^T\widehat{\mathbf C}^{-T}$ is a $\sum L_r\times \sum L_r$ nonsingular matrix.
	Let $\mathbf g_1=[\mathbf g_{1,1}^T\ \dots\ \mathbf g_{1,R}^T]^T$ and $\mathbf g_2=[\mathbf g_{2,1}^T\ \dots\ \mathbf g_{2,R}^T]^T$ be  columns of $\mathbf G$, where  $\mathbf g_{1,r},\mathbf g_{2,r} \in\fF^{L_r}$.
	Then, by  assumption \cref{eq:U2}, at most one of the vectors $\mathbf g_{1,1},\dots,\mathbf g_{1,R}$ is nonzero.
	Since $\mathbf G$ is nonsingular we have that exactly one of the vectors $\mathbf g_{1,1},\dots,\mathbf g_{1,R}$ is nonzero. Let $\mathbf g_{1,i}\ne\mathbf 0$.
	Similarly, we also have that exactly one of the vectors $\mathbf g_{2,1},\dots,\mathbf g_{2,R}$ is nonzero.
	 Let $\mathbf g_{2,j}\ne\mathbf 0$.  
	  We claim that if  $\mathbf g_1$ and $\mathbf g_2$ are columns of the same block $\mathbf G_r\in\fF^{\sum L_r\times L_r}$ of $\mathbf G=[\mathbf G_1\ \dots\ \mathbf G_R]$, then $i=j$. Indeed, by \cref{eq:T3T3},
	 \begin{equation}
	 (\mathbf a_i\otimes \mathbf B_i)\mathbf g_{1,i} = \widehat{\mathbf a}_r\otimes \mathbf y_1 \ \text{ and }\ 
	 (\mathbf a_j\otimes \mathbf B_j)\mathbf g_{2,j} = \widehat{\mathbf a}_r\otimes \mathbf y_2,
	 \label{eq:aiajar}
	 \end{equation}
	 where $\mathbf y_1$ and $\mathbf y_2$ are columns of $\widehat{\mathbf B}_r$. It follows from \cref{eq:aiajar} that
	 $\mathbf a_i$ and $\mathbf a_j$ are proportional to $\widehat{\mathbf a}_r$. Since, by \cref{eq:kA2andfcr}, $k_{\mathbf A}\geq 2$, it follows that $i=j$. Thus, in the partition $\mathbf G_r=[\mathbf G_{1r}^T\ \dots\ \mathbf G_{Rr}^T]^T$ with $\mathbf G_{1r}\in\fF^{L_1\times L_r},\dots\,\mathbf G_{Rr}\in\fF^{L_R\times L_r}$, exactly one block is nonzero.  Since
	 $\mathbf G=[\mathbf G_1\ \dots\ \mathbf G_R]$   is nonsingular,  it follows that the nonzero block of $\mathbf G_r$ is square, i.e. $L_r\times L_r$, and nonsingular, $r=1,\dots,R$.
	 Hence $\mathbf G$ can be reduced to block diagonal form by permuting its  blocks $\mathbf G_1,\dots,\mathbf G_R$. Let $\mathbf P$ denote a  permutation matrix such that
	 $\mathbf G\mathbf P = \Bdiag(\tilde{\mathbf G}_{11},\dots,\tilde{\mathbf G}_{RR})$ with nonsingular $\tilde{\mathbf G}_{rr}\in\fF^{L_r\times L_r}$. It is clear that multiplication of the right hand side 
	 of \cref{eq:bothsides} by $\mathbf P$ corresponds to a permutation  of the summands in $\mathcal T = \sum_{r=1}^{ R}\widehat{\mathbf a}_r\circ(\widehat{\mathbf B}_r\widehat{\mathbf C}_r^T)$. Thus,  the terms in $\mathcal T = \sum_{r=1}^{ R}\widehat{\mathbf a}_r\circ(\widehat{\mathbf B}_r\widehat{\mathbf C}_r^T)$  can can be permuted so that  \cref{eq:bothsides} holds for $\mathbf G=\Bdiag(\tilde{\mathbf G}_{11},\dots,\tilde{\mathbf G}_{RR})$. Hence \cref{eq:bothsides} reduces to the  $R$ identities
	 $$
	 (\mathbf a_r\otimes\mathbf B_r)\tilde{\mathbf G}_{rr} = \widehat{\mathbf a}_r\otimes\widehat{\mathbf B}_r,\qquad r=1,\dots,R
	 $$
	 which imply that 
	 $\widehat{\mathbf a}_r$ is proportional to $\mathbf a_r$ and that the column space of $\widehat{\mathbf B}_r$ coincides with the column space of
	 $\mathbf B_r$. In other words, we have shown that  $\widehat{\mathbf a}_r$ and $\widehat{\mathbf B}_r$ in  $\mathcal T = \sum_{r=1}^{ R}\widehat{\mathbf a}_r\circ(\widehat{\mathbf B}_r\widehat{\mathbf C}_r^T)$ can be chosen to be equal to $\mathbf a_r$ and $\mathbf B_r$, respectively.
	 Since the matrix $[\mathbf a_1\otimes\mathbf B_1\ \dots\ \mathbf a_R\otimes\mathbf B_R]$ has full column rank, we also have from \cref{eq:T3T3} that $\widehat{\mathbf C}=\mathbf C$.
\end{proof}
		\begin{proof}[Proof of \cref{thm:maingenStrassen}]
If $I\geq R$, then the result follows from \cref{thm:genericknown}. So, throughout the proof we assume that $I<R$.

By definition set
\begin{equation}
W_{\mathbf A,\mathbf B,\mathbf C}:=\{(\mathbf A,\mathbf B,\mathbf C):\
\text{ the assumptions in  \cref{Thm:determgen} do not hold} \}.\label{eq:mu=0}
\end{equation}
We show that  $\mu\{W_{\mathbf A,\mathbf B,\mathbf C}\}=0$, where $\mu$ denotes a measure on $\fF^{I\times R}\times\fF^{J\times\sum L_r}\times \fF^{K\times\sum L_r}$ that is absolutely continuous with respect to the Lebesgue measure. Obviously, $W_{\mathbf A,\mathbf B,\mathbf C}= W_{\mathbf C}\cup W_{\mathbf A,\mathbf B}$, where
\begin{align*}
W_{\mathbf C}&:=\{(\mathbf A,\mathbf B,\mathbf C):\ \mathbf C\
\text{does not have full column rank} \} \ \text{  and }\\
W_{\mathbf A,\mathbf B}&:=\{(\mathbf A,\mathbf B,\mathbf C):\
\text{assumption \cref{eq:U2} does not hold} \}.
\end{align*}
It is clear that, by the assumption $\sum L_r\leq K$ in \cref{eq:sumLrleqI1J1}, $\mu\{W_{\mathbf C}\}=0$, so we need to show that $\mu\{W_{\mathbf A,\mathbf B}\}=0$. 
Since \cref{eq:U2} does not depend on $\mathbf C$, we have  $W_{\mathbf A,\mathbf B}=W\times \fF^{J\times \sum L_r}$, where
$$
W:=\{(\mathbf A,\mathbf B):\ \text{assumption \cref{eq:U2} does not hold} \}
$$
is a subset of  $\fF^{I\times R}\times\fF^{J\times\sum L_r}$. From  Fubini’s theorem \cite[Theorem C, p.148]{halmos1974measure} it follows that
$\mu\{W_{\mathbf A,\mathbf B}\}=0$   if and only if $\mu_1\{W\}=0$, where
$\mu_1$ is a measure on $\fF^{I\times R}\times\fF^{J\times\sum L_r}$ that is absolutely continuous with respect to the Lebesgue measure.
Since $R>I$ and $J\geq L_{R-1}+L_R$ ($=\max\limits_{1\leq i<j\leq R}(L_i+L_j)$), it follows that 
$$
\mu_1\{(\mathbf A,\mathbf B ):\ k_{\mathbf A}<I\ \text{ or }\ k_{\mathbf B}'<2\}=0.
$$
Hence we can assume w.l.o.g. that
\begin{equation}
W = \{(\mathbf A,\mathbf B):\ \text{assumption \cref{eq:U2} does not hold, } k_{\mathbf A}=I, \text{ and } k_{\mathbf B}'\geq 2 \}.
\label{eq:defW}
\end{equation}
The remaining part of the proof is based on a well-known algebraic geometry based method. In \cite{AlgGeom1} we have   explained  the method and used it 
to study generic uniqueness of CPD and INDSCAL. We have explained in \cite{AlgGeom1} that to prove that $\mu_1\{W\}=0$, it is sufficient to show that for $\fF=\mathbb C$ the Zariski closure $\overline{W}$ of $W$  is not the entire space $\mathbb C^{I\times R}\times\mathbb C^{J\times\sum L_r}$, which is equivalent to $\dim \overline{W}\leq IR + J\sum L_r-1$.
To estimate the dimension of $\overline{W}$ we will take the following four steps (for a detailed explanation of the steps and examples see \cite{AlgGeom1};
also, for $L_1=\dots=L_r=1$, the overall derivation is similar to the proof of Lemma 2.5 in \cite{Strassen1983}). To simplify the presentation of the steps, we   omit mentioning the isomorphism between $\mathbb C^{k\times l}\times \mathbb C^{m\times n}$ and $\mathbb C^{kl+mn}$; for instance, we consider $W$  as a subset of  $\mathbb C^{d_1}$, where $d_1=IR + J\sum L_r$.  
In the first step we parameterize $W$. Namely, we construct a subset $\widehat{Z}\subseteq \mathbb C^{d_1+I+J+\sum L_r}$  and a projection $\pi: \mathbb C^{d_1+I+J+\sum L_r}\rightarrow \mathbb C^{d_1}$ such that $W=\pi(\widehat{Z})$. 
 In step $2$ we represent $\widehat{Z}$ as a finite union of subsets $Z_{r_1,\dots,r_I}^{l_1,\dots,l_I}$ such that each $ Z_{r_1,\dots,r_I}^{l_1,\dots,l_I}$ is the image of a Zariski open subset of $\mathbb C^{d_1-d_2+1}$ under a rational mapping, where  $d_2:=(I-1)(J-1)-\sum L_r$ is nonnegative by \cref{eq:sumLrleqI1J1}.  In step $3$ we show that $\dim(Z_{r_1,\dots,r_I}^{l_1,\dots,l_I})=d_1-d_2+1$ and that $\dim(\pi(Z_{r_1,\dots,r_I}^{l_1,\dots,l_I}))\leq d_1-d_2 -1$. Finally, in step $4$ we conclude that $\dim \overline{W} = \dim(\pi(\widehat{Z}))\leq \max(\dim (\pi(Z_{r_1,\dots,r_I}^{l_1,\dots,l_I})))=d_1-d_2-1\leq d_1-1$. 
		
\textit{Step 1.} Let $\omega(\mathbf g_1,\dots,\mathbf g_R)$  denote the number of nonzero vectors in the set $\{\mathbf g_1,\dots,\mathbf g_R\}$.
 We claim that if assumption \cref{eq:U2} does not hold, $k_{\mathbf A}=I$, and $k_{\mathbf B}'\geq 2$, then
 $\omega(\mathbf g_1,\dots,\mathbf g_R)\geq I$. Indeed, if $I>\omega(\mathbf g_1,\dots,\mathbf g_R)\geq 2$, then by the Frobenius inequality, 
\begin{align*}
1\geq r_{\mathbf a_1(\mathbf B_1\mathbf g_1)^T+\dots +\mathbf a_R(\mathbf B_R\mathbf g_R)^T}=
r_{\mathbf A \Bdiag(\mathbf g_1^T,\dots,\mathbf g_R^T)\mathbf B^T}\geq\\
r_{\mathbf A \Bdiag(\mathbf g_1^T,\dots,\mathbf g_R^T)} + r_{\Bdiag(\mathbf g_1^T,\dots,\mathbf g_R^T)\mathbf B^T}-
r_{ \Bdiag(\mathbf g_1^T,\dots,\mathbf g_R^T) }=\\ \omega(\mathbf g_1,\dots,\mathbf g_R) + r_{[\mathbf B_1\mathbf g_1\ \dots\ \mathbf B_R\mathbf g_r]}-
\omega(\mathbf g_1,\dots,\mathbf g_R)\geq 2,
\end{align*}
which is a contradiction.
  Hence, $W$ in \cref{eq:defW} can be expressed as
\begin{gather}
W = \Big\{(\mathbf A,\mathbf B):\ 
 \text{there exist } \mathbf g_1\in\mathbb C^{L_1},\dots,\mathbf g_R\in\mathbb C^{L_R},\  \mathbf z\in\mathbb C^I, \text{ and }\mathbf y\in\mathbb C^J\nonumber\\
\text{ such that } \ \mathbf a_1(\mathbf B_1\mathbf g_1)^T+\dots +\mathbf a_R(\mathbf B_R\mathbf g_R)^T =\mathbf z\mathbf y^T,\label{eq:115}\\
k_{\mathbf A}=I,\ k_{\mathbf B}'\geq 2, \text{ and }\label{eq:116}\\
\omega(\mathbf g_1,\dots,\mathbf g_R)\geq I\Big\}.\label{eq:117}
\end{gather}
It is clear that $W=\pi(\widehat Z)$, where
\begin{equation*}
\widehat{Z} = \Big\{
(\mathbf A,\mathbf B,\mathbf g_1,\dots,\mathbf g_R,\mathbf z,\mathbf y): \cref{eq:115}\text{--} \cref{eq:117}\ \text{hold}\Big\}
\end{equation*}
is a subset of $\mathbb C^{I\times R}\times \mathbb C^{J\times \sum L_r}\times \mathbb C^{L_1}\times\dots\times\mathbb C^{L_R}\times\mathbb C^I\times\mathbb C^J$ and $\pi$ is the projection onto the first two factors
$$
\pi: \mathbb C^{I\times R}\times \mathbb C^{J\times \sum L_r}\times \mathbb C^{L_1}\times\dots\times\mathbb C^{L_R}\times\mathbb C^I\times\mathbb C^J\rightarrow 
\mathbb C^{I\times R}\times \mathbb C^{J\times \sum L_r}.
$$

\textit{Step 2.} Let $g_{l,r}$ denote the $l$th entry of $\mathbf g_r$.
Since
$$
\omega(\mathbf g_1,\dots,\mathbf g_R)\geq I \Leftrightarrow \mathbf g_{r_1}\ne\mathbf 0,\dots,\mathbf g_{r_I}\ne\mathbf 0\ \text{for some }1\leq r_1<\dots<r_I\leq R
$$
and since
$$
\mathbf g_{r_1}\ne\mathbf 0,\dots,\mathbf g_{r_I}\ne\mathbf 0 \Leftrightarrow g_{l_1,r_1}\cdots g_{l_I,r_I}\ne 0 \text{ for some } 1\leq l_{1}\leq L_{r_1},\dots,
1\leq l_{I}\leq L_{r_I}, 
$$
we obtain that
\begin{multline*}
\widehat{Z}=\bigcup\limits_{1\leq r_1<\dots<r_I\leq R}\ \bigcup\limits_{1\leq l_{1}\leq L_{r_1},\dots,
	1\leq l_{I}\leq L_{r_I}} \\
\Big\{
(\mathbf A,\mathbf B,\mathbf g_1,\dots,\mathbf g_R,\mathbf z,\mathbf y): \cref{eq:115}\text{--} \cref{eq:116}\ \text{hold and } g_{l_1,r_1}\cdots g_{l_I,r_I}\ne 0\Big\}.
\end{multline*}
Let $\mathbf A_{r_1,\dots,r_I}$ denote the submatrix of $\mathbf A$ formed by columns $r_1,\dots,r_I$.
Since \cref{eq:116} is more restrictive than the condition $\det(\mathbf A_{r_1,\dots,r_I})\ne 0$, it follows that
$$
\widehat{Z}\subseteq\bigcup\limits_{1\leq r_1<\dots<r_I\leq R}\ \bigcup\limits_{1\leq l_{1}\leq L_{r_1},\dots,
	1\leq l_{I}\leq L_{r_I}} Z_{r_1,\dots,r_I}^{l_1,\dots,l_I},
$$
 where
\begin{multline*}
Z_{r_1,\dots,r_I}^{l_1,\dots,l_I} = \\
\Big\{
(\mathbf A,\mathbf B,\mathbf g_1,\dots,\mathbf g_R,\mathbf z,\mathbf y): \cref{eq:115}\ \text{holds, }\ 
\det(\mathbf A_{r_1,\dots,r_I})\ne 0,\ g_{l_1,r_1}\cdots g_{l_I,r_I}\ne 0\Big\}.
\end{multline*}
We show that each subset $Z_{r_1,\dots,r_I}^{l_1,\dots,l_I}$ can be represented as  the image of a Zariski open subset
$Y_{r_1,\dots,r_I}^{l_1,\dots,l_I}$  of  $\mathbb C^{IR+J\sum L_r+\sum L_r-IJ+I+J}$ under a rational map
$\phi_{r_1,\dots,r_I}^{l_1,\dots,l_I}$,
$Z_{r_1,\dots,r_I}^{l_1,\dots,l_I}=\phi_{r_1,\dots,r_I}^{l_1,\dots,l_I}(Y_{r_1,\dots,r_I}^{l_1,\dots,l_I})$. To simplify the presentation we restrict ourselves to the case $r_1=1,\dots,r_I=I$ and $l_1=\dots=l_I=1$. The general case can be proved in the same way. Let $\mathbf A=[\mathbf A_1\  \mathbf A_2]$ with
$\mathbf A_1\in \fF^{I\times I}$ and $\mathbf A_2\in \fF^{I\times (R-I)}$, so that $\mathbf A_1=\mathbf A_{1\dots 1}$. By \cref{eq:115},
\begin{equation}
[\mathbf B_1\mathbf g_1\ \dots\ \mathbf B_I\mathbf g_I] = [\mathbf y\mathbf z^T-[\mathbf B_{I+1}\mathbf g_{I+1}\ \dots\ \mathbf B_R\mathbf g_R]\mathbf A_2^T]\mathbf A_1^{-T}.\label{eq:B1g1etc}
\end{equation}
Let $\mathbf B_r = [\mathbf b_{1,r}\ \mathbf B_{2,r}]$ and $\mathbf g_r=[g_{1,r}\ \mathbf g_{2,r}^T]^T$, so 
\begin{equation}
[\mathbf B_1\mathbf g_1\ \dots\ \mathbf B_I\mathbf g_I]=[\mathbf b_{1,1}\dots\ \mathbf b_{1,I}] \diag(g_{1,1},\dots,g_{1,I})+[\mathbf B_{2,1}\mathbf g_{2,1}\ \dots\ \mathbf B_{2,I}\mathbf g_{2,I}].\label{eq:BBB}
\end{equation}
 Then, by \cref{eq:B1g1etc,eq:BBB},
\begin{equation}
\begin{split}
[\mathbf b_{1,1}\dots\ \mathbf b_{1,I}] &= 
\big(
[\mathbf y\mathbf z^T-[\mathbf B_{I+1}\mathbf g_{I+1}\ \dots\ \mathbf B_R\mathbf g_R]\mathbf A_2^T]\mathbf A_1^{-T}-\\
&\qquad\qquad\qquad[\mathbf B_{2,1}\mathbf g_{2,1}\ \dots\ \mathbf B_{2,I}\mathbf g_{2,I}]\big)\diag(g_{1,1}^{-1},\dots,g_{1,I}^{-1}),
\end{split}\label{eq:expressionrat}
\end{equation}
so the entries of $\mathbf b_{1,1}\dots\ \mathbf b_{1,I}$ are rational functions of the entries of 
$\mathbf A$, $\mathbf B_{2,1}, \dots,\mathbf B_{2,I}$, $\mathbf B_{I+1},\dots,\mathbf B_R$, $\mathbf g_1,\dots,\mathbf g_R$, $\mathbf z$, and $\mathbf y$.
It is clear  that
\begin{multline*}
Y_{1,\dots,I}^{1,\dots,1}: = 
\Big\{
([\mathbf A_1\ \mathbf A_2],[\mathbf B_{2,1}\ \dots\ \mathbf B_{2,I}\ \mathbf B_{I+1}\ \dots\ \mathbf B_R],\mathbf g_1,\dots,\mathbf g_R,\mathbf z,\mathbf y):\\
\det(\mathbf A_{1})\ne 0,\ g_{1,1}\cdots g_{1,I}\ne 0 \Big\}
\end{multline*}
is a Zariski open subset of
$\mathbb C^{I\times R}\times \mathbb C^{J\times \left(\sum\limits_{r=1}^I (L_r-1) + \sum\limits_{r=I+1}^R L_r\right)}\times \mathbb C^{L_1}\times\dots\times\mathbb C^{L_R}\times\mathbb C^I\times\mathbb C^J$ and that
$Z_{1,\dots,I}^{1,\dots,1}=\phi_{1,\dots,I}^{1,\dots,1}(Y_{1,\dots,I}^{1,\dots,1})$, where the rational mapping
\begin{equation*}
\begin{split}
\phi_{1,\dots,I}^{1,\dots,1}:\ &([\mathbf A_1\ \mathbf A_2],[\mathbf B_{2,1}\ \dots\ \mathbf B_{2,I}\ \mathbf B_{I+1}\ \dots\ \mathbf B_R],\mathbf g_1,\dots,\mathbf g_R,\mathbf z,\mathbf y)\rightarrow\\
&([\mathbf A_1\ \mathbf A_2],[[\mathbf b_{1,1}\ \mathbf B_{2,1}]\ \dots\ [\mathbf b_{1,I}\ \mathbf B_{2,I}]\ \mathbf B_{I+1}\ \dots\ \mathbf B_R],\mathbf g_1,\dots,\mathbf g_R,\mathbf z,\mathbf y)=\\
&(\mathbf A,\mathbf B,\mathbf g_1,\dots,\mathbf g_R,\mathbf z,\mathbf y)
\end{split} 
\end{equation*}
is defined by \cref{eq:expressionrat}.

\textit{Step 3.} In this step we prove that $\dim(\pi(Z_{r_1,\dots,r_I}^{l_1,\dots,l_I}))\leq IR+J\sum L_r-1$. W.l.o.g. we  restrict ourselves again to the case 
$r_1=1,\dots,r_I=I$ and $l_1=\dots=l_I=1$.
Since the dimension of the image $\phi_{1,\dots,I}^{1,\dots,1}(Y_{1,\dots,I}^{1,\dots,1})$ cannot exceed the dimension of
$Y_{1,\dots,I}^{1,\dots,1}$ and since $Y_{1,\dots,I}^{1,\dots,1}$ is a Zariski open subset we have
\begin{equation}
\dim (Z_{1,\dots,I}^{1,\dots,1}) \leq\footnote{\tcr{It can be proved that actually ``$=$'' holds but in the sequel we will only need ``$\leq$''.}} \dim(Y_{1,\dots,I}^{1,\dots,1}) = IR+J(-I+\sum\limits_{r=1}^R L_r)+L_1+\dots+L_r +I+J.
\label{eq:dimZ}
\end{equation}
Let $f:\ Z_{1,\dots,I}^{1,\dots,1}\rightarrow \mathbb C^{I\times R}\times \mathbb C^{J\times \sum L_r} $ denote the restriction of $\pi$ to 
$ Z_{1,\dots,I}^{1,\dots,1}$:
$$
f:\ (\mathbf A,\mathbf B,\mathbf g_1,\dots,\mathbf g_R,\mathbf z,\mathbf y)\rightarrow (\mathbf A,\mathbf B),\qquad
(\mathbf A,\mathbf B,\mathbf g_1,\dots,\mathbf g_R,\mathbf z,\mathbf y)\in 
Z_{1,\dots,I}^{1,\dots,1}.
$$
From the definition of $Z_{1,\dots,I}^{1,\dots,1}$
it follows that if $(\mathbf A,\mathbf B,\mathbf g_1,\dots,\mathbf g_R,\mathbf z,\mathbf y)\in 
Z_{1,\dots,I}^{1,\dots,1}$, then
$(\mathbf A,\mathbf B,\alpha\beta\mathbf g_1,\dots,\alpha\beta\mathbf g_R,\alpha\mathbf z,\beta\mathbf y)\in 
Z_{1,\dots,I}^{1,\dots,1}$ for any nonzero $\alpha,\beta\in\mathbb C$. Hence for any $(\mathbf A,\mathbf B)\in f(Z_{1,\dots,I}^{1,\dots,1})$ we have that
$$
f^{-1}((\mathbf A,\mathbf B))\supseteq \{(\mathbf A,\mathbf B,\alpha\beta\mathbf g_1,\dots,\alpha\beta\mathbf g_R,\alpha\mathbf z,\beta\mathbf y):\ \alpha\ne 0,\ \beta\ne 0\},
$$
implying that 
\begin{equation}
\dim (f^{-1}(\mathbf A,\mathbf B))\geq \dim \{(\alpha\mathbf z,\beta\mathbf y):\ \alpha\ne 0,\ \beta\ne 0\}=2, 
\label{eq:dimfminusone}
\end{equation}
where $f^{-1}(\cdot)$ denotes the preimage.
From the  fiber dimension theorem \cite[Theorem 3.7, p. 78]{Perrin2008}, \cref{eq:dimZ}, \cref{{eq:dimfminusone}},  and the assumption
$\sum  L_r\leq (I-1)(J-1)$ in \cref{eq:sumLrleqI1J1}
 it follows that 
\begin{multline*}
\dim(f(Z_{1,\dots,I}^{1,\dots,1}))\leq \dim(Z_{1,\dots,I}^{1,\dots,1})-\dim (f^{-1}(\mathbf A,\mathbf B))=\\
IR+J\sum\limits_{r=1}^R L_r-1+\sum\limits_{r=1}^R L_r-(I-1)(J-1)\l\leq IR+J\sum\limits_{r=1}^R L_r-1.
\end{multline*}
Since $\pi(Z_{1,\dots,I}^{1,\dots,1})=f(Z_{1,\dots,I}^{1,\dots,1})$, we have that
$\dim(\pi(Z_{1,\dots,I}^{1,\dots,1}))\leq IR+J\sum\limits_{r=1}^R L_r-1$.

\textit{Step 4.} Finally, we  have that
$
\dim \overline{W} = \dim(\pi(\widehat{Z}))\leq \max(\dim (\pi(Z_{r_1,\dots,r_I}^{l_1,\dots,l_I})))\leq IR+J\sum  L_r-1. 
$
\end{proof}
\bibliographystyle{siamplain}
\bibliography{BTDpaper}
\end{document}

%% file: defs_BTD_2.tex
\usepackage{wasysym}
\usepackage{amsmath,amssymb,amsbsy,multirow,mathtools }
\usepackage{stmaryrd}
\usepackage{tikz}

\newsiamremark{expl}{Example}

\newcommand{\tcr}[1]{\textcolor{black}{#1}}

\newcommand{\sspan}{\operatorname{span}}
\newcommand{\scol}{\operatorname{col}}
\newcommand{\fF}{\mathbb F}  
\newcommand{\mapsfrommy}{=} 
\newcommand{\kgp}[1]{k_{{\mathbf #1},gen}'}
 
\newcommand{\ML}{ML rank-$(1,L_r,L_r)$ }
\newcommand{\MLatmostLL}[1]{max ML rank-$(1,{#1},{#1})$ }
\newcommand{\MLatmost}{\MLatmostLL{L_r}}
\newcommand{\Bdiag}{\operatorname{blockdiag}} 

\newcommand{\nullsp}[1]{\operatorname{Null}\left( #1 \right) }

\newcommand{\unf}[2]{{\mathbf  #1}_{(#2)}}
\newcommand{\wprod}[2]{{#1}\wedge {#2}}
\newcommand{\symprod}[2]{{#1}\cdot {#2}}
\newcommand{\rubinom}[2]{C_{#1}^{#2}} 
\newcommand{\vecsym}[1]{\operatorname{vec}{(\fF^{#1\times #1}_{sym})}} 
\newcommand{\matrixU}{\textcolor{black}{\mathbf V}}
\newcommand{\vectoru}{\textcolor{black}{\mathbf v}}
\newcommand{\numberx}{d_1}
\newcommand{\numbery}{d_2}
\newcommand{\vectora}{\mathbf f}
\newcommand{\numbera}{f}
\newcommand{\vectorb}{\mathbf g}
\newcommand{\numberb}{g}
\newcommand{\vectorc}{\mathbf h}
\newcommand{\numberc}{h}
\newcommand{\parameterhone}{p_1}
\newcommand{\parameterhtwo}{p_2}
\newcommand{\valuep}{\alpha}
\newcommand{\valuehthree}{\beta}
\newcommand{\valuehfour}{\gamma}
\newcommand{\valued}{\delta}
\newcommand{\valueqone}{\tau_1}
\newcommand{\valueqtwo}{\tau_2}
\newcommand{\valueqthree}{\tau_3}
\newcommand{\valueqfour}{\tau_4}
\newcommand{\valueyone}{q_1}
\newcommand{\valueytwo}{q_2}
\newcommand{\valuezone}{r_1}
\newcommand{\valueztwo}{r_2}
\newcommand{\valuetone}{s_1}
\newcommand{\valuettwo}{s_2}
\newcommand{\valueu}{u}
\newcommand{\values}{t}
\newcommand{\valuev}{v}
\newcommand{\valuew}{w}

\usepackage{enumitem}
\newlist{assumptions}{enumerate}{10}
\setlist[assumptions]{label*=\alph*)}

\crefname{assumptionsi}{assumption}{assumptions}
\Crefname{assumptionsi}{Assumption}{Assumptions}

\crefformat{assumptions}{\textup{#2(#1)#3}}
\crefrangeformat{assumptions}{\textup{#3(#1)#4--#5(#2)#6}}
\crefmultiformat{assumptions}{\textup{#2(#1)#3}}{ and \textup{#2(#1)#3}}
{, \textup{#2(#1)#3}}{, and \textup{#2(#1)#3}}
\crefrangemultiformat{assumptions}{\textup{#3(#1)#4--#5(#2)#6}}%
{ and \textup{#3(#1)#4--#5(#2)#6}}{, \textup{#3(#1)#4--#5(#2)#6}}{, and \textup{#3(#1)#4--#5(#2)#6}}

\usepackage{enumitem}
\newlist{conditions}{enumerate}{10}
\setlist[conditions]{label*=\alph*)}

\crefname{conditionsi}{condition}{conditions}
\Crefname{conditionsi}{Condition}{Conditions}

\crefformat{conditions}{\textup{#2(#1)#3}}
\crefrangeformat{conditions}{\textup{#3(#1)#4--#5(#2)#6}}
\crefmultiformat{conditions}{\textup{#2(#1)#3}}{ and \textup{#2(#1)#3}}
{, \textup{#2(#1)#3}}{, and \textup{#2(#1)#3}}
\crefrangemultiformat{conditions}{\textup{#3(#1)#4--#5(#2)#6}}%
{ and \textup{#3(#1)#4--#5(#2)#6}}{, \textup{#3(#1)#4--#5(#2)#6}}{, and \textup{#3(#1)#4--#5(#2)#6}}

\newlist{statements}{enumerate}{10}
\setlist[statements]{label*=\arabic*)}

\crefname{statementsi}{statement}{statements}
\Crefname{statementsi}{Statement}{Statements}

\crefformat{statements}{\textup{#2(#1)#3}}
\crefrangeformat{statements}{\textup{#3(#1)#4--#5(#2)#6}}
\crefmultiformat{statements}{\textup{#2(#1)#3}}{ and \textup{#2(#1)#3}}
{, \textup{#2(#1)#3}}{, and \textup{#2(#1)#3}}
\crefrangemultiformat{statements}{\textup{#3(#1)#4--#5(#2)#6}}%
{ and \textup{#3(#1)#4--#5(#2)#6}}{, \textup{#3(#1)#4--#5(#2)#6}}{, and \textup{#3(#1)#4--#5(#2)#6}}

\newcommand{\algorithmicinput}{\textbf{Input:}}
\newcommand{\INPUT}{\item[\algorithmicinput]}
\newcommand{\algorithmicoutput}{\textbf{Output:}}
\newcommand{\OUTPUT}{\item[\algorithmicoutput]}

%% file: ex_shared.tex
\usepackage{lipsum}
\usepackage{amsfonts}
\usepackage{graphicx}
\usepackage{epstopdf}
\usepackage{algorithmic}
\ifpdf
  \DeclareGraphicsExtensions{.eps,.pdf,.png,.jpg}
\else
  \DeclareGraphicsExtensions{.eps}
\fi

\numberwithin{theorem}{section}

\newcommand{\TheTitle}{On uniqueness and computation of the decomposition of a  tensor into   multilinear rank-$(1, L_{\lowercase{r}},L_{\lowercase{r}})$ terms} 
\newcommand{\TheAuthors}{I. Domanov  and L. De Lathauwer}

\headers{Decomposition of a  tensor into   multilinear rank-$(1, L_{\lowercase{r}},L_{\lowercase{r}})$ terms}{\TheAuthors}

\title{{\TheTitle}\thanks{Submitted to the editors DATE.
\funding{This work was funded by (1) Research Council KU Leuven: C1 project c16/15/059-nD; (2) the Flemish Government under the ``Onderzoeksprogramma Artifici\"ele Intelligentie (AI) Vlaanderen'' programme; (3) F.W.O.:  project  G.0830.14N, G.0881.14N, 
	 G.0F67.18N (EOS SeLMA); (4)  EU: The research leading to these results has received funding from the European Research Council under the European Union's Seventh Framework Programme (FP7/2007-2013) / ERC Advanced Grant: BIOTENSORS (no.  339804). This paper reflects only the authors' views and the Union is not liable for any use that may be made of the contained information.}}}

\author{
	Ignat Domanov\thanks{
		Group Science, Engineering and Technology, KU Leuven - Kulak,
		E. Sabbelaan 53, 8500 Kortrijk, Belgium and
		Dept. of Electrical Engineering  ESAT/STADIUS KU Leuven,
		Kasteelpark Arenberg 10, bus 2446, B-3001 Leuven-Heverlee, Belgium
		(\email{ignat.domanov@kuleuven.be}, \email{lieven.delathauwer@kuleuven.be}).}
	\and
	Lieven De Lathauwer\footnotemark[2]
}

\usepackage{amsopn}
\DeclareMathOperator{\diag}{diag}


%% file: BTD_2_SIMAX_revision_2_final_black_and_red.bbl
\begin{thebibliography}{10}

\bibitem{realvscomplex}
{\sc E.~Angelini, C.~Bocci, and L.~Chiantini}, {\em Real identifiability vs.
  complex identifiability}, Linear and Multilinear Algebra, 66 (2018),
  pp.~1257--1267.

\bibitem{NVNagain}
{\sc C.~Beltr\'{a}n, P.~Breiding, and N.~Vannieuwenhoven}, {\em Pencil-based
  algorithms for tensor rank decomposition are not stable}, arXiv:1807.04159,
  (2018).

\bibitem{Bocci2013}
{\sc C.~Bocci, L.~Chiantini, and G.~Ottaviani}, {\em Refined methods for the
  identifiability of tensors}, Ann. Mat. Pura. Appl., 193 (2014),
  pp.~1691--1702.

\bibitem{Bro2009}
{\sc R.~Bro, R.~A. Harshman, N.~D. Sidiropoulos, and M.~E. Lundy}, {\em
  Modeling multi-way data with linearly dependent loadings}, Journal of
  Chemometrics, 23 (2009), pp.~324--340.

\bibitem{Cay2017SIMAX}
{\sc Y.~Cai and C.~Liu}, {\em An algebraic approach to nonorthogonal general
  joint block diagonalization}, SIAM J. Matrix Anal. Appl., 38 (2017),
  pp.~50--71.

\bibitem{Cherrak2017}
{\sc O.~Cherrak, H.~Ghennioui, N.~Thirion-Moreau, and E.~H. Abarkan}, {\em
  Preconditioned optimization algorithms solving the problem of the non unitary
  joint block diagonalization: application to blind separation of convolutive
  mixtures}, Multidim. Syst. Sign. Process., 29 (2018), pp.~1373--1396.

\bibitem{Nick2014}
{\sc L.~Chiantini, G.~Ottaviani, and N.~Vannieuwenhoven}, {\em An algorithm for
  generic and low-rank specific identifiability of complex tensors}, SIAM J.
  Matrix Anal. Appl., 35 (2014), pp.~1265--1287.

\bibitem{2016NVNeffective}
{\sc L.~Chiantini, G.~Ottaviani, and N.~Vannieuwenhoven}, {\em Effective
  criteria for specific identifiability of tensors and forms}, SIAM J. Matrix
  Anal. Appl., 38 (2017), pp.~656--681.

\bibitem{LievenCichocki2013}
{\sc A.~Cichocki, D.~Mandic, C.~Caiafa, A.-H. Phan, G.~Zhou, Q.~Zhao, and
  L.~De~Lathauwer}, {\em Tensor decompositions for signal processing
  applications. {F}rom two-way to multiway component analysis}, IEEE Signal
  Process. Mag., 32 (2015), pp.~145--163.

\bibitem{ComoJ10}
{\em Handbook of Blind Source Separation, Independent Component {A}nalysis and
  {A}pplications}, Academic Press, Oxford, UK, 2010.

\bibitem{DeLathauwer2006}
{\sc L.~De~Lathauwer}, {\em A link between the canonical decomposition in
  multilinear algebra and simultaneous matrix diagonalization}, SIAM J. Matrix
  Anal. Appl., 28 (2006), pp.~642--666.

\bibitem{LDLBTDPartII}
{\sc L.~De~Lathauwer}, {\em Decompositions of a higher-order tensor in block
  terms --- {P}art {II}: {D}efinitions and uniqueness}, SIAM J. Matrix Anal.
  Appl., 30 (2008), pp.~1033--1066.

\bibitem{LievenLrLr1}
{\sc L.~De~Lathauwer}, {\em Blind separation of exponential polynomials and the
  decomposition of a tensor in rank-$({{L}}_r,{{L}}_r,1)$ terms}, SIAM J.
  Matrix Anal. Appl., 32 (2011), pp.~1451--1474.

\bibitem{LDL2008}
{\sc L.~De~Lathauwer and A.~de~Baynast}, {\em Blind deconvolution of {DS-CDMA}
  signals by means of decomposition in rank-$(1,{{L}},{{L}})$ terms}, IEEE
  Trans. Signal Process., 56 (2008), pp.~1562--1571.

\bibitem{Otto2016}
{\sc O.~Debals, M.~Van~Barel, and L.~De~Lathauwer}, {\em L\"{o}wner-based blind
  signal separation of rational functions with applications}, IEEE Trans.
  Signal Process., 64 (2016), pp.~1909--1918.

\bibitem{PartI}
{\sc I.~Domanov and L.~De~Lathauwer}, {\em {O}n the uniqueness of the canonical
  polyadic decomposition of third-order tensors --- {P}art {I}: {B}asic results
  and uniqueness of one factor matrix}, SIAM J. Matrix Anal. Appl., 34 (2013),
  pp.~855--875.

\bibitem{PartII}
{\sc I.~Domanov and L.~De~Lathauwer}, {\em {O}n the uniqueness of the canonical
  polyadic decomposition of third-order tensors --- {P}art {II}: {O}verall
  uniqueness}, SIAM J. Matrix Anal. Appl., 34 (2013), pp.~876--903.

\bibitem{LinkGEVD}
{\sc I.~Domanov and L.~De~Lathauwer}, {\em Canonical polyadic decomposition of
  third-order tensors: reduction to generalized eigenvalue decomposition}, SIAM
  J. Matrix Anal. Appl., 35 (2014), pp.~636--660.

\bibitem{AlgGeom1}
{\sc I.~Domanov and L.~De~Lathauwer}, {\em Generic uniqueness conditions for
  the canonical polyadic decomposition and {INDSCAL}}, SIAM J. Matrix Anal.
  Appl., 36 (2015), pp.~1567--1589.

\bibitem{JSTSP2016IDLDL}
{\sc I.~Domanov and L.~De~Lathauwer}, {\em Generic uniqueness of a structured
  matrix factorization and applications in blind source separation}, IEEE J.
  Sel. Topics Signal Process., 10 (2016), pp.~701--711.

\bibitem{BTD1paper}
{\sc I.~Domanov, N.~Vervliet, and L.~De~Lathauwer}, {\em Decomposition of a
  tensor into multilinear rank-$({{M}}_r,{{N}}_r,\cdot)$ terms}, Internal
  Report 18-51, ESAT-STADIUS, KU Leuven (Leuven, Belgium),  (2018).

\bibitem{GolubVanLoan}
{\sc G.~H. Golub and C.~F. Van~Loan}, {\em {M}atrix {C}omputations}, Johns
  Hopkins University Press, Baltimore, 4th~ed., 2013.

\bibitem{halmos1974measure}
{\sc P.~R. Halmos}, {\em Measure theory}, Springer-Verlag, New-York, 1974.

\bibitem{HornJohnson}
{\sc R.~A. Horn and C.~R. Johnson}, {\em {M}atrix {A}nalysis}, Cambridge
  University Press, Cambridge, 1990.

\bibitem{Kolda}
{\sc T.~G. Kolda and B.~W. Bader}, {\em Tensor decompositions and
  applications}, SIAM Review, 51 (2009), pp.~455--500.

\bibitem{Liu2012}
{\sc X.~Liu, T.~Jiang, L.~Yang, and H.~Zhu}, {\em Paralind-based
  identifiability results for parameter estimation via uniform linear array},
  EURASIP Journal on Advances in Signal Processing, 2012 (2012), p.~154.

\bibitem{Mueller_Smith_2016}
{\sc C.~{Mueller-Smith} and P.~Spasojevi\'c}, {\em Column-wise symmetric block
  partitioned tensor decomposition}, in 2016 IEEE International Conference on
  Acoustics, Speech and Signal Processing (ICASSP), 2016, pp.~2956--2960.

\bibitem{lowrankbasis2017}
{\sc Y.~Nakatsukasa, , T.~Soma, and A.~Uschmajew}, {\em Finding a low-rank
  basis in a matrix subspace}, Mathematical Programming, 162 (2017),
  pp.~325--361.

\bibitem{Nion_LDL_LL1}
{\sc D.~Nion and L.~De~Lathauwer}, {\em A link between the decomposition of a
  third-order tensor in rank-$({L},{L},1)$ terms and joint block
  diagonalization}, in 2009 3rd IEEE International Workshop on Computational
  Advances in Multi-Sensor Adaptive Processing (CAMSAP), 2009, pp.~89--92.

\bibitem{Perrin2008}
{\sc D.~Perrin}, {\em Algebraic Geometry. An Introduction}, Springer-Verlag
  London, 2008.

\bibitem{TensRev2017}
{\sc N.~D. Sidiropoulos, L.~De~Lathauwer, X.~Fu, K.~Huang, E.~E. Papalexakis,
  and C.~Faloutsos}, {\em Tensor decomposition for signal processing and
  machine learning}, IEEE Trans. Signal Process., 65 (2017), pp.~3551--3582.

\bibitem{MikaelCoupledPII}
{\sc M.~S{\o}rensen, I.~Domanov, and L.~De~Lathauwer}, {\em Coupled canonical
  polyadic decompositions and (coupled) decompositions in multilinear
  rank-$({{L}}_{r_n},{{L}}_{r_n},1)$ terms---{P}art {II}: {A}lgorithms}, SIAM
  J. Matrix Anal. Appl., 36 (2015), pp.~1015--1045.

\bibitem{Strassen1983}
{\sc V.~Strassen}, {\em Rank and optimal computation of generic tensors},
  Linear Algebra Appl., 52--53 (1983), pp.~645--685.

\bibitem{TenBerge2004}
{\sc J.~M.~F. ten Berge}, {\em Partial uniqueness in {CANDECOMP/PARAFAC}},
  Journal of Chemometrics, 18 (2004), pp.~12--16.

\bibitem{VanDerVeen1996}
{\sc A.-J. Van Der~Veen and A.~Paulraj}, {\em An analytical constant modulus
  algorithm}, IEEE Trans. Signal Process., 44 (1996), pp.~1136--1155.

\bibitem{tensorlab3.0}
{\sc N.~Vervliet, O.~Debals, L.~Sorber, M.~Van~Barel, and L.~De~Lathauwer},
  {\em Tensorlab 3.0}, Mar. 2016, \url{https://www.tensorlab.net}.
\newblock Available online.

\bibitem{Yang2014}
{\sc M.~Yang}, {\em On partial and generic uniqueness of block term tensor
  decompositions}, Annali Dell'Universita'Di Ferrara, 60 (2014), pp.~465--493.

\bibitem{StrangeLL1paper}
{\sc M.~Yang, D.~Che, W.~Liu, Z.~Kang, C.~Peng, M.~Xiao, and Q.~Cheng}, {\em On
  identifiability of 3-tensors of multilinear rank $(1,{{L}}_r, {{L}}_r)$}, Big
  Data \& Information Analytics, 1 (2016), pp.~391--401.

\end{thebibliography}
